\documentclass[letter,10pt]{amsart}
\usepackage{amsmath, amsfonts, amssymb,amscd,graphics,graphicx,color,eucal,setspace,xypic} 
\usepackage{latexsym}   
\usepackage{color}
\usepackage{fullpage}
\usepackage{url}
\usepackage{caption}
\usepackage{mathtools}

\numberwithin{equation}{section}

\newtheorem{theorem}{Theorem}[section]
\newtheorem*{theoremA*}{Theorem A}
\newtheorem*{theoremB*}{Theorem B}
\newtheorem*{corollaryB*}{Corollary of Theorem B}
\newtheorem{lemma}[theorem]{Lemma} 
\newtheorem{corollary}[theorem]{Corollary}
\newtheorem{proposition}[theorem]{Proposition} 
\newtheorem{conjecture}[theorem]{Conjecture} 
\newtheorem{remark}[theorem]{Remark}
\newtheorem{example}[theorem]{Example}
\newtheorem{definition}[theorem]{Definition}

\allowdisplaybreaks[3]

\def\C{\mathbb C}
\def\Q{\mathbb Q}
\def\Z{\mathbb Z}
\def\P{{\mathcal P}}
\def\cf{\check f}

\def\x{x}
\def\varphih{\varphi_h}
\def\Ih{I_h}
\def\f{\theta}

\DeclareMathOperator{\Hess}{Hess}
\DeclareMathOperator{\NR}{NR}

\DeclareMathOperator{\Lie}{Lie}
\DeclareMathOperator{\pt}{pt}

\newcommand{\Flags}{Flag}
\newcommand{\NInv}[1]{D_{#1}}

\newcommand{\hsm}{{\hspace{1mm}}}

\newcommand{\height}{\mathrm{ht}}
\newcommand{\dimX}{d}
\newcommand{\into}{\hookrightarrow}
\newcommand{\TChFlag}{\tau^T}
\newcommand{\SChFlag}{\tau^S}
\newcommand{\ChFlag}{\tau}
\newcommand{\SChNil}{\bar\tau^S}
\newcommand{\ChNil}{\bar\tau}
\newcommand{\TChSemi}{\hat\tau^T}
\newcommand{\ChSemi}{\hat\tau}

\newcommand{\gl}{\mathfrak{gl}}

\def\N{\mathsf{N}}

\def\gjk{g_{j,k}}
\def\Sn{\mathfrak{S}_n}

\def\Abe{\mathcal A}
\def\Tn{T}
\def\sp{s}

\begin{document}
  
\title[Cohomology rings of Hessenberg varieties]{The
  cohomology rings of regular nilpotent Hessenberg varieties in Lie type A}
\author {Hiraku Abe}
\address{Osaka City University Advanced Mathematical Institute, 3-3-138 Sugimoto, Sumiyoshi-ku, Osaka 558-8585, 
Japan / Department of Mathematics, University of Toronto, 40 St. George Street, Toronto, Ontario, Canada, M5S 2E4}
\email{hirakuabe@globe.ocn.ne.jp}

\author {Megumi Harada} 
\address{Department of Mathematics and Statistics, McMaster University, 1280 Main Street West, Hamilton, Ontario L8S4K1, Canada}
\email{Megumi.Harada@math.mcmaster.ca}
\urladdr{\url{http://www.math.mcmaster.ca/~haradam}}

\author {Tatsuya Horiguchi}
\address{Osaka City University Advanced Mathematical Institute, 3-3-138 Sugimoto, Sumiyoshi-ku, Osaka 558-8585, Japan}
\email{tatsuya.horiguchi0103@gmail.com}

\author {Mikiya Masuda}
\address{Department of Mathematics, Osaka City University, 3-3-138 Sugimoto, Sumiyoshi-ku, Osaka 558-8585, Japan}
\email{masuda@sci.osaka-cu.ac.jp}
\date{\today}



\begin{abstract}
 Let $n$ be a fixed positive integer and $h: \{1,2,\ldots,n\} \rightarrow \{1,2,\ldots,n\}$ a
 Hessenberg function. The main results of this paper are twofold. First, 
we give a 
systematic method, depending in a simple manner on the Hessenberg
function $h$, for producing 
an explicit presentation by generators and relations of the
cohomology ring $H^\ast(\Hess(\mathsf{N},h))$ with $\Q$ coefficients
of the corresponding regular nilpotent Hessenberg variety $\Hess(\mathsf{N},h)$. 
Our result generalizes known results in
special cases such as the Peterson variety and also allows us to answer a question posed by Mbirika and Tymoczko. Moreover, our list of
generators in fact forms a regular sequence, allowing us to use
techniques from commutative algebra in our arguments. 
Our second main result
gives an isomorphism between the cohomology ring $H^*(\Hess(\mathsf{N},h))$
of the regular nilpotent Hessenberg variety and the $\Sn$-invariant
subring $H^*(\Hess(\mathsf{S},h))^{\Sn}$ of the cohomology ring of the
regular semisimple Hessenberg variety (with respect to the $\Sn$-action on $H^*(\Hess(\mathsf{S},h))$ defined by Tymoczko). 
Our second main result implies that 
$\mathrm{dim}_{\Q} H^k(\Hess(\mathsf{N},h)) = \mathrm{dim}_{\Q} H^k(\Hess(\mathsf{S},h))^{\Sn}$ for all $k$ and hence partially proves 
the Shareshian-Wachs conjecture in combinatorics, which is in turn
related to the well-known Stanley-Stembridge conjecture. A proof of
the full Shareshian-Wachs conjecture was recently given by Brosnan and
Chow, but in our special case, our methods yield a stronger result (i.e. an isomorphism of rings) 
by more elementary considerations. This paper provides detailed proofs
of results we recorded previously in a research announcement. 
\end{abstract}

\maketitle

\setcounter{tocdepth}{1}

\section{Introduction and statement of main results (Theorem A and Theorem B)}

Hessenberg varieties in type A are subvarieties of the full flag
variety $\Flags(\C^n)$ of nested sequences of linear subspaces in
$\C^n$. Their geometry and (equivariant) topology have been studied
extensively since the late 1980s \cite{DeM, DeMShay, ma-pr-sh}.  This
subject lies at the intersection of, and makes connections between,
many research areas such as geometric representation theory (see for example
\cite{Springer76, Fung03}), combinatorics (see e.g. \cite{Fulm, mb}), and
algebraic geometry and topology (see e.g. \cite{Kostant, br-ca04, ty, InskoYong, precup13a, precup13b}). 
A special case of Hessenberg varieties called the Peterson
variety $Pet_n$ arises in the study of the quantum
cohomology of the flag variety \cite{Kostant, Rietsch}, and more
generally, geometric properties and invariants of many different types
of Hessenberg varieties (including in Lie types other than A) have
been widely studied. 
The (equivariant and ordinary) cohomology rings of
Hessenberg varieties have received particular attention. To cite just two
examples, 
the second author and Tymoczko gave an 
explicit set of generators for $H^*(Pet_n)$ and prove a 
Schubert-calculus-type ``Monk formula'', resulting in a presentation of
$H^*(Pet_n)$ via generators and relations in \cite{ha-ty}, and  
in a different direction, Brion and Carrell
showed an isomorphism between the equivariant cohomology ring of a
regular nilpotent Hessenberg variety with the affine coordinate ring
of a certain affine curve \cite{br-ca04}. Beyond the two manuscripts
just mentioned, there has also been extensive work on 
the equivariant and ordinary 
cohomology rings of Springer
varieties \cite{DeConciniProcesi, Tanisaki, DewittHarada, Horiguchi,
  AbeHoriguchi} and of some types of regular
nilpotent Hessenberg varieties (including Peterson varieties in
different Lie types) \cite{BayeganHarada2,
  fu-ha-ma, ha-ho-ma}. 
However, it has been an open question to give a general and systematic
description of the equivariant cohomology rings of all regular
nilpotent Hessenberg varieties \cite[Introduction, page
2]{InskoTymoczko}, to which our results provides an answer (in Lie type
A). 

In addition, very recent developments provide further evidence that
Hessenberg varieties occupy a central place in the fruitful
intersection of algebraic geometry, combinatorics, and geometric
representation theory. We first recall some background. The
well-known Stanley-Stembridge conjecture in combinatorics states that
the chromatic symmetric function of the incomparability graph of a
so-called $(3+1)$-free poset is $e$-positive. 
In related work, Stanley \cite{Stanley1989} also showed a relation between $q$-Eulerian polynomials and a certain $\Sn$-representation on
the cohomology of the toric variety associated with the Coxeter
complex of type $\mathrm{A}_{n-1}$ studied by Procesi \cite{proc90}. The above toric variety is a special case
of a regular semisimple Hessenberg variety of type A \cite{ma-pr-sh}, and Tymoczko
\cite{tymo08} has defined $\Sn$-representations on their cohomology
rings which generalize the $\Sn$-representation studied by
Procesi. Motivated by the above, Shareshian and
Wachs formulated in 2011 a conjecture \cite{sh-wa11} relating the chromatic quasisymmetric
function of the incomparability graph of a natural unit interval order and
Tymoczko's $\Sn$-representation on the cohomology of the associated
regular semisimple Hessenberg variety. While
the Shareshian-Wachs conjecture does not imply the Stanley-Stembridge
conjecture, it nevertheless represents a significant step towards its
solution. In a 2015 preprint, Brosnan and Chow \cite{br-ch} prove the
Shareshian-Wachs conjecture by showing a remarkable relationship
between the Betti numbers of different Hessenberg varieties; a key
ingredient in their approach is a certain family of Hessenberg
varieties, the (cohomology of the) fibers of which are related via
monodromy. Our second main result (Theorem B) also contributes to this discussion, as we explain below.

We now describe the two main results (Theorem A and Theorem B below)
of this manuscript in more detail. 
Recall that the \textbf{flag variety} $\Flags(\C^n)$ consists of nested sequences of linear subspaces of $\C^n$, 
\[
\Flags(\C^n) := 
\{ V_{\bullet} = (\{0\} \subset  V_1 \subset  V_2 \subset  \cdots V_{n-1} \subset 
V_n = \C^n) \hsm \mid \hsm \dim_{\C}(V_i) = i \ \textrm{for all} \ i=1,\ldots,n\}. 
\]
Additionally, let $h: \{1,2,\ldots,n\} \to \{1,2,\ldots,n\}$ be a
\textbf{Hessenberg function}, i.e. $h$ satisfies
$h(i) \geq i$ for all $i$ and $h(i+1)\geq h(i)$ for all $i<n$. Also
let $\mathsf{N}$ denote 
a regular 
nilpotent matrix in $\gl(n,\C)$, i.e. 
a matrix whose Jordan form consists of exactly one Jordan block with
corresponding eigenvalue equal to $0$. 
Then we may define the
\textbf{regular nilpotent Hessenberg variety (associated to $h$)} to
be the subvariety of $\Flags(\C^n)$ defined by 
\begin{equation}
\Hess(\mathsf{N},h) := \{ V_{\bullet}  \in \Flags(\C^n) \;
\vert \;  \mathsf{N} V_i \subset 
V_{h(i)} \text{ for all } i=1,\ldots,n\} \subset  \Flags(\C^n).
\end{equation}
Our first main theorem gives an explicit presentation via generators
and relations of the cohomology\footnote{Throughout this document
  (unless
  explicitly stated otherwise) we work with
  cohomology with coefficients in $\Q$.} ring $H^*(\Hess(\mathsf{N}, h))$
of the regular nilpotent Hessenberg
variety associated to any Hessenberg function $h$. 
For any pair $i, j$ with $i\geq j$, let $\cf_{i,j}$ be the polynomial 
\begin{align}\label{eq:definition of f check}
\cf_{i,j} := \sum_{k=1}^j \Big( x_k \prod_{\ell=j+1}^{i} (x_k -
x_\ell) \Big) \textup{ for } i\geq j
\end{align}
with the convention $\prod_{\ell=j+1}^{j}(x_k-x_\ell)=1$.

\begin{theoremA*} 
  Let $n$ be a positive integer and $h: \{1,2,\ldots,n\} \to
  \{1,2,\ldots,n\}$ a Hessenberg function. Let $\mathsf{N}$
  denote a regular nilpotent matrix in $\mathfrak{gl}(n,\C)$ and let
 $\Hess(\mathsf{N},h) \subset  \Flags(\C^n)$ be
  the associated regular nilpotent Hessenberg variety. Then the
  restriction map
\[
H^*(\Flags(\C^n)) \to H^*(\Hess(\mathsf{N},h))
\]
is surjective, and there is an isomorphism of graded $\Q$-algebras
\begin{equation}\label{eq:intro Theorem A} 
H^*(\Hess(\mathsf{N},h)) \cong \Q[x_1, \ldots, x_n]/\check \Ih
\end{equation}
where $\check \Ih$ is the ideal of $\Q[x_1, \ldots, x_n]$ defined by
\begin{equation}\label{definition ideal check Ih} 
\check \Ih := (\cf_{h(j),j} \mid 1 \leq j \leq n ). 
\end{equation}
\end{theoremA*}
The following points are worth noting immediately. Firstly, the
equation~\eqref{eq:definition of f check} gives a simple closed
formula for the 
polynomials $\cf_{h(j),j}$ generating the ideal $\check \Ih$ 
in~\eqref{eq:intro Theorem A}; moreover,
the ideal depends in a manifestly simple and systematic manner on the Hessenberg
function $h$. Secondly, these generators $\{\cf_{h(j),j}\}_{j=1}^n$
have algebraic properties which make them particularly
useful. Specifically, the $\cf_{h(j),j}$ (as well as their equivariant
counterparts $f_{i,j}$ which we discuss below) in fact form a \emph{regular
  sequence} (cf. Definition~\ref{def of regular seq for Hilb series})
in the sense of commutative algebra, and it is precisely this property
which allows us to exploit techniques in e.g. the theory of Hilbert series and Poincar\'e duality
algebras to prove both of our main results. Thirdly, we can answer a
question posed by Mbirika and Tymoczko
\cite[Question 2]{mb-ty13}: they asked whether $H^*(\Hess(\mathsf{N},h))$ is
isomorphic to the quotient of $\Q[x_1,\ldots,x_n]$ by a certain ideal,
described in detail in \cite{mb-ty13}, which is generated by ``truncated
symmetric polynomials''. Our Theorem A says that, in general, the
answer is ``No''. 
For instance, in the special case
of the Peterson variety $Pet_n$ of complex dimension $n$ for $n \geq 3$, it is not
difficult to see directly from Mbirika and Tymoczko's definitions in \cite{mb-ty13}
that their ring contains a non-zero element of degree $2$ whose square is equal to $0$,
whereas one can see from our presentation~\eqref{eq:intro Theorem A} that
$H^*(Pet_n)$ contains no such element. 
Finally, our Theorem A generalizes known results: in the special cases of the full flag variety $\Flags(\C^n)$ and 
the Peterson variety $Pet_n$, the presentation given in Theorem A
recovers previously known presentations of the relevant cohomology
rings (cf. Remarks~\ref{rema:peterson}). 

Next we turn to Theorem B, for which we need additional
terminology. Let $h$ be a 
Hessenberg function and this time let $\mathsf{S}$ denote a
regular semisimple matrix in $\gl(n,\C)$, i.e. a matrix which is
diagonalizable with distinct eigenvalues. 
Then the \textbf{regular semisimple
  Hessenberg variety (associated to $h$)} is defined to be
\begin{equation}\label{eq:def reg ss Hess} 
\Hess(\mathsf{S},h) := \{ V_{\bullet}  \in \Flags(\C^n) \;
\vert \;  \mathsf{S} V_i \subset 
V_{h(i)} \text{ for all } i=1,\ldots,n\} \subset  \Flags(\C^n).
\end{equation}
The cohomology rings of these varieties admit an action
of the symmetric group $\Sn$, as Tymoczko pointed out many years ago
\cite{tymo08}. 
In Theorem B, we 
prove - for a fixed Hessenberg function $h$ - that there exists an isomorphism
of graded rings between the cohomology ring of the corresponding regular
nilpotent Hessenberg variety and the $\Sn$-invariant subring of the
cohomology ring of the corresponding regular semisimple Hessenberg
variety. More precisely, we have the following.

\begin{theoremB*}\label{AHHM conjecture}
  Let $n$ be a positive integer and $h: \{1,2,\ldots,n\} \to
  \{1,2,\ldots,n\}$ a Hessenberg function. 
Let $\mathsf{N}$ 
  denote a regular nilpotent matrix and $\mathsf{S}$ denote 
a regular semisimple matrix in $\gl(n,\C)$. Let $\Hess(\mathsf{N},h)$
and $\Hess(\mathsf{S},h)$ be the associated 
regular nilpotent and regular semisimple Hessenberg varieties
respectively. 
Then there exists a unique graded $\Q$-algebra homomorphism $\Abe\colon
  H^*(\Hess(\mathsf{N},h))\rightarrow H^*(\Hess(\mathsf{S},h))$
making the following diagram commute:
\vspace{20pt}
\begin{align}\label{eq:Abe map}
 \ 
\end{align}
\vspace{-50pt}
\begin{center}
\begin{picture}(160,50)
   \put(10,35){$H^*(\Flags(\C^n))$}
   \put(77,35){$\overrightarrow{\qquad}$}
   \put(50,18){\rotatebox[origin=c]{-45}{$\overrightarrow{\qquad\ }$}}
   \put(50,18.7){\rotatebox[origin=c]{-45}{$\overrightarrow{\hspace{22pt}}$}}
   \put(110,18){\rotatebox[origin=c]{45}{$\overrightarrow{\qquad\ }$}}
   \put(100,35){$H^*(\Hess(\mathsf{S},h))$}
   \put(122,15){$\footnotesize{\text{$\Abe$}}$}
   \put(55,0){$H^*(\Hess(\mathsf{N},h))$}
\end{picture} 
\end{center}
\vspace{5pt}
where the maps $H^*(\Flags(\C^n)) \to H^*(\Hess(\mathsf{S}, h))$ and
$H^*(\Flags(\C^n)) \to H^*(\Hess(\mathsf{N},h))$ are induced from the
inclusions $\Hess(\mathsf{S},h) \into \Flags(\C^n)$ and
$\Hess(\mathsf{N},h) \into \Flags(\C^n)$ respectively. 
Moreover, the image of $\Abe$ is precisely the ring $H^*(\Hess(\mathsf{S},h))^{\Sn} $ of
$\Sn$-invariants in $H^*(\Hess(\mathsf{S},h))$, and when the target of
$\Abe$ is restricted to this invariant subring, then
\begin{equation*}
\Abe\colon H^*(\Hess(\mathsf{N},h))\to H^*(\Hess(\mathsf{S},h))^{\Sn} 
\end{equation*}
is an isomorphism of graded $\Q$-algebras. 
\end{theoremB*}
As a special case, we note that Theorem B implies that the cohomology
$H^*(Pet_n)$ of the Peterson variety $Pet_n$ is isomorphic to the $\Sn$-invariant
subring $H^*(X)^{\Sn}$ of the cohomology ring of the toric variety $X$ associated with
the Coxeter complex of type $\mathrm{A}_{n-1}$. This fact
  could be previously seen by comparing the description of
  $H^*(Pet_n)$ given explicitly in \cite{fu-ha-ma} to the description
  of $H^*(X)^{\Sn}$ stated without proof in \cite{klya}. Indeed, the
  striking similarity of the rings in \cite{klya} and \cite{fu-ha-ma}
  was our original motivation to prove Theorem B.

Next we discuss the relationship between Theorem B and recent research
in combinatorics. As briefly discussed above, Shareshian and Wachs conjectured a precise
relationship between the (Frobenius characteristic of) Tymoczko's
$\Sn$-representation on the cohomology group of a regular semisimple
Hessenberg variety $\Hess(\mathsf{S},h)$ and the
chromatic quasisymmetric function $X_G(x,t)$ of a graph $G$ defined
from the Hessenberg function $h$ \cite{sh-wa11, sh-wa14}. (Details are in Section~\ref{sect:
  Shareshian-Wachs conjecture}.) When both sides of their conjectured
equality are expanded in terms of Schur functions $s_\lambda(x)$ for
$\lambda$ a partition of $n$, their conjecture can be
interpreted as a set of equalities of the coefficients (which are polynomials in $t$) of 
$s_\lambda(x)$ on each side. In \cite[Theorem 6.9]{sh-wa14} Shareshian
and Wachs also obtain a closed formula for the coefficient of
$s_n(x)$, i.e. the coefficient corresponding to the trivial
representation, and it agrees with the Hilbert series (also a
  polynomial in $t$) of
$H^*(\Hess(\mathsf{N},h))$. Upon unraveling some definitions, it readily follows
that our Theorem B proves the Shareshian-Wachs conjecture
for the coefficient of the trivial representation. 
During the preparation of this manuscript we
learned that Brosnan and Chow have independently proved the full
Shareshian-Wachs conjecture, i.e. the equality of the coefficients for
\emph{all} Schur functions, not just $s_n(x)$. However, we note that 
while Brosnan and Chow obtain an
equality of dimensions of vectors spaces $\mathrm{dim}\,
H^k(\Hess(\mathsf{N},h)) = \mathrm{dim} \, H^k(\Hess(\mathsf{S},h))^{\Sn}$ for
varying $k$ \cite[Theorem 76]{br-ch}, 
their techniques do not appear to immediately yield
further information about the product structure on the rings
$H^*(\Hess(\mathsf{N},h))$ and $H^*(\Hess(\mathsf{S},h))^{\Sn}$. 
Thus, for our special case, our Theorem B is stronger than the corresponding result
in \cite{br-ch}. Moreover, while Brosnan and Chow's arguments utilize
deep and powerful results in the theory of local systems and perverse
sheaves (specifically, the local invariant cycle theorem of
Beilinson-Bernstein-Deligne), our methods are more elementary,
thus providing a useful alternative perspective on this circle of
ideas. 

We now briefly discuss the methods used in the proofs of Theorems A
and B. Our basic strategy is to exploit the presence of torus actions 
on the varieties in question and to use well-known techniques in
equivariant topology, e.g. localization and
Goresky-Kottwitz-MacPherson (GKM) theory. More specifically, we state
and prove an ``equivariant'' version of Theorem A using certain
polynomials $f_{h(j),j}$ which lift in an appropriate sense the
polynomials $\cf_{h(j),j}$ appearing in Theorem A. For Theorem B, we
use GKM theory to combinatorially describe both the equivariant cohomology
of $\Hess(\mathsf{S},h)$ and the $\Sn$-action on it, and also use some
standard commutative algebra results on Poincar\'e duality algebras
together with the fact that the polynomials $\{\cf_{h(j),j}\}$ form a
regular sequence to obtain our result. 

We take a moment to record an open question motivated by our work. 
It would be of interest to determine
    the ring structure of the full cohomology ring 
    $H^*(\Hess(\mathsf{S},h))$ (i.e. not just the $\Sn$-invariant
    subring) of regular semisimple Hessenberg
    varieties for arbitrary
    Hessenberg functions $h$. We have preliminary results in this
    direction. For example, it turns out that, when
    $h=(h(1),n,\dots,n)$ with an arbitrary value of $h(1)$, the classes
    $g_{1,i}$ and $\ChSemi_i$ (see Section~\ref{sec:definition Abe map}   for the precise definitions) for $i=1,\dots,n$ generate
    $H^*(\Hess(\mathsf{S},h))$ as a ring, and it is possible to give
    an explicit presentation of $H^*(\Hess(\mathsf{S},h))$ with
    respect to these generators. Similarly, the corresponding classes
    also can be shown to generate the ring
    $H^*(\Hess(\mathsf{S},h))$ for the case $h=(m,\dots,m,n,\dots,n)$
    for any $m$. Furthermore, we have found a finite list of
    generators of $H^*(\Hess(\mathsf{S},h))$ for the case
    $h=(h(1),h(2),n,\dots,n)$ with arbitrary values of $h(1)$ and
    $h(2)$. In an ongoing project, we are investigating the problem of
    finding ring
    generators of $H^*(\Hess(\mathsf{S},h))$ for arbitrary Hessenberg
    functions $h$ which behave well with respect to Tymoczko's $\Sn$-representation.

The paper is organized as follows. After briefly reviewing some
background and terminology on regular nilpotent Hessenberg varieties in Section~\ref{sec:H}, we state the
equivariant version of our Theorem A in Section~\ref{sec:f} as
Theorem~\ref{theorem:reg nilp Hess cohomology}. The key properties of
the polynomials $f_{i,j}$, necessary for the proof of
Theorem~\ref{theorem:reg nilp Hess cohomology}, are recorded in
Section~\ref{sec:fproperty}. The fact that 
the equivariant version of the homomorphism \eqref{eq:intro Theorem A}
is well-defined is shown
in Section~\ref{sec:y}. To prove that it
is in fact an isomorphism requires some preparatory arguments using
Hilbert series, which are recorded in Section~\ref{sec:hilbert}. In
Section~\ref{sec:proof of first main theorem}, using the results of the previous sections we are able
to complete the proof of Theorem~\ref{theorem:reg nilp Hess
  cohomology} and hence also of Theorem A. Next, turning our attention to
Theorem B, we quickly recount some background and terminology
concerning regular semisimple Hessenberg varieties in
Section~\ref{sec:background on reg ss Hess}. We recall and also
prove some essential facts about Tymoczko's $\Sn$-action on the
(equivariant and ordinary) cohomology of regular semisimple Hessenberg
varieties in Section~\ref{sec:properties of Tymoczko action}. We
prove Theorem B in Section~\ref{sec:definition Abe map} and 
discuss the
connection between our results and the Shareshian-Wachs conjecture in
Section~\ref{sect: Shareshian-Wachs conjecture}.

\bigskip
\noindent \textbf{Acknowledgements.}  
We are grateful to Satoshi Murai for his invaluable help in the commutative
algebra arguments. 
We also thank John Shareshian and Michelle Wachs
for their kind support and their interest in this project.  
We are also grateful to Takashi Sato for improving the proof of Lemma~\ref{lem:gjk is GKM}. 
The first
author was supported in part by the JSPS Program for Advancing
Strategic International Networks to Accelerate the Circulation of
Talented Researchers: ``Mathematical Science of Symmetry, Topology and Moduli, Evolution of International Research Network based on
OCAMI'', and he is grateful to Yoshihiro Ohnita for enhancing his research projects.
He was also supported by a JSPS Grant-in-Aid for Young
Scientists (B): 15K17544.  The second author was partially supported
by an NSERC Discovery Grant, a Canada Research Chair (Tier 2) Award,
an Association for Women in Mathematics Ruth Michler Award, and a
Japan Society for the Promotion of Science Invitation Fellowship for
Research in Japan (Fellowship ID L-13517).  
The third author was partially supported by a JSPS Grant-in-Aid 
for JSPS
Fellows: 15J09343.
The fourth author was partially supported by a JSPS Grant-in-Aid for Scientific Research 25400095.

\bigskip
\section{Background and preliminaries}\label{sec:H}

In this section we recall some background and establish some terminology 
  for the rest of the paper. Specifically, in
  Section~\ref{subsec:setup} we recall the definitions of the
regular nilpotent 
  Hessenberg varieties, as well as the torus actions on them. We also quickly
  recount some techniques in torus-equivariant cohomology which will
  be used throughout. 
In Section~\ref{subsec:fixed
    points nilp} we analyze the torus-fixed point set of the regular
  nilpotent Hessenberg variety which plays a key role in our later arguments. 

\subsection{The setup}\label{subsec:setup} 
Hessenberg varieties in Lie
type A are 
subvarieties of the \textbf{(full) flag variety}
$\Flags(\C^n)$, which is the 
collection of sequences of nested linear subspaces of $\C^n$:
\[
\Flags(\C^n) := 
\{ V_{\bullet} = (\{0\} \subset  V_1 \subset  V_2 \subset  \cdots V_{n-1} \subset 
V_n = \C^n) \hsm \mid \hsm \dim_{\C}(V_i) = i \ \textrm{for all} \ i=1,\ldots,n\}. 
\]
It is well-known that $\Flags(\C^n)$ can also be realized as a
homogeneous space $\text{GL}(n,\C)/B$ where $B$ is the standard Borel
subgroup of upper-triangular invertible matrices. Thus there is a
natural action of $\text{GL}(n,\C)$ on $\Flags(\C^n)$ given by left
multiplication on cosets. 

A Hessenberg variety in $\Flags(\C^n)$ is specified by two pieces of
data: a Hessenberg function and a choice of an element in the Lie algebra
$\gl(n,\C)$ 
of $\text{GL}(n,\C)$. We begin by discussing the first of these
parameters. 
Throughout this document we use the notation 
\begin{align*}
[n] := \{1,2,\ldots, n\}.
\end{align*} 

\begin{definition}\label{definition:Hessenberg function} 
A \textbf{Hessenberg function} is a function $h: [n] \to [n]$ satisfying the following two conditions
\begin{align*}
&h(i) \geq i \hspace{43pt} \textrm{for} \ \ i\in[n],\\
&h(i+1)\geq h(i) \hspace{12pt} \textrm{for} \ \ i\in[n-1].
\end{align*}
We frequently write a Hessenberg function by listing its values in sequence,
i.e. $h = (h(1), h(2), \ldots, h(n))$.
\end{definition} 
We also define $H_n$ to be the set of 
Hessenberg functions $h: [n] \to [n]$, i.e.
\begin{align}\label{def of the set of Hess fct}
H_n := \{ h:[n]\rightarrow[n] \mid \emph{$h$ is a Hessenberg function} \}.
\end{align}

For the discussion to follow, it will be useful to introduce some
terminology associated to a given Hessenberg function. 

\begin{definition}\label{definition:Hessenberg subspace} 
  Let $h \in H_n$ be a Hessenberg function. Then we define the
  \textbf{Hessenberg subspace} $H(h)$ to be the linear subspace of
  $\gl(n,\C) \cong Mat(n\times n, \C)$ specified as follows:
\begin{align}\label{eq:Hessenberg subspace} 
H(h) := \{ A = (a_{ij})_{i,j\in[n]} \in \gl(n,\C) \mid a_{ij} = 0 \textup{ if } i > h(j) \}.
\end{align} 
\end{definition}

It is important to note that the $H(h)$ is frequently \emph{not} a Lie
subalgebra of $\gl(n,\C)$. However, it \emph{is} stable under the conjugation
action of the usual maximal torus $T$ (of invertible
  diagonal matrices),
and
we may decompose $H(h)$ into eigenspaces with respect to this action
as
\begin{align}\label{eq:H(h) weight spaces} 
H(h) \cong \mathfrak{b} \oplus \left( \ \bigoplus_{\substack{i, j \in [n], \\ j < i \leq h(j)}} \gl(n,\C)_{(i,j)} \right) 
\end{align} 
where $\mathfrak{b} = \Lie(B)$ denotes the Lie algebra of the Borel subgroup of
upper-triangular matrices, and $\gl(n,\C)_{(i,j)}$ denotes the
$1$-dimensional $T$-weight space of $\gl(n,\C)$ 
spanned by the
elementary matrix $E_{i,j}$ with a $1$ in the $(i,j)$-th entry and
$0$'s elsewhere. In Lie-theoretic language, the $(i,j)$ satisfying the
condition in the RHS of~\eqref{eq:H(h) weight spaces} correspond to
the negative roots of $\gl(n,\C)$ whose corresponding root spaces
appear in $H(h)$. It will be useful later on to focus attention on
these roots, so we introduce the notation
\begin{align}\label{eq:def NR(h)} 
\NR(h) := \{ (i,j) \in [n] \times [n] \mid j < i \leq h(j) \}.
\end{align}
It is conceptually useful to express $H(h)$
  pictorially by drawing a configuration of boxes on a square grid of
  size $n\times n$ whose shaded boxes correspond to the roots
  appearing in \eqref{eq:H(h) weight spaces}, and from this perspective,
the set 
$\NR(h)$ corresponds one-to-one with ``the boxes in (the picture associated to) $H(h)$
which lie strictly below the main diagonal''. See
Figure~\ref{picture for NR}. 

\begin{figure}[h]
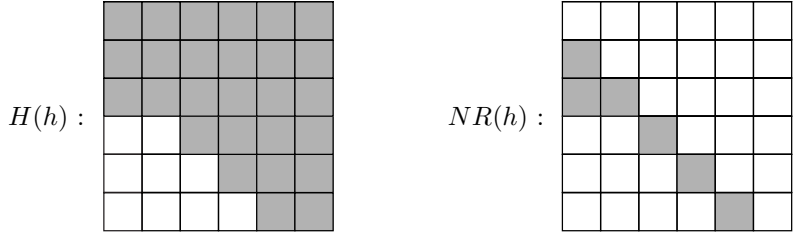

\centering
{\unitlength 0.1in%
}%
\caption{The pictures of $H(h)$ and $\NR(h)$ for $h=(3,3,4,5,6,6)$.}
\label{picture for NR}
\end{figure}
We now introduce the main geometric objects of interest in this
manuscript. 
Let $h:[n]\to[n]$ be a Hessenberg function and let 
$A$ be an $n\times n$ matrix in $\gl(n,\C)$. Then the
\textbf{Hessenberg variety} $\Hess(A,h)$ associated to $h$ and $A$ is
defined to be  
\begin{align}\label{eq:def-general Hessenberg}
\Hess(A,h) := \{ V_{\bullet}  \in \Flags(\C^n) \;
\vert \;  A V_i \subset 
V_{h(i)} \text{ for all } i\in[n]\} \subset  \Flags(\C^n).
\end{align}
In particular, by definition $\Hess(A,h)$ is a subvariety of
$\Flags(\C^n)$, and if $h=(n,n,\ldots,n)$, then it is immediate
from~\eqref{eq:def-general Hessenberg}
that $\Hess(A,h)=\Flags(\C^n)$ for any choice of
$A$. Thus the full flag variety $\Flags(\C^n)$ is itself a special
case of a Hessenberg variety; this will be important later on. We also
remark that if $g \in \text{GL}(n,\C)$, then $\Hess(A,h)$ and
$\Hess(gAg^{-1},h)$ can be identified via the action of
$\text{GL}(n,\C)$ on $\Flags(\C^n)$. In particular, important geometric
features of Hessenberg varieties are frequently dependent only on the
conjugacy class of the element $A \in \gl(n,\C)$, and not on $A$
itself. 

In this paper we focus on two special cases of
Hessenberg varieties, as we now describe. Let $\mathsf{N}$ denote a regular
nilpotent matrix in $\gl(n,\C)$, i.e.  
a matrix whose Jordan form consists of exactly one Jordan block with
corresponding eigenvalue equal to $0$. Similarly let $\mathsf{S}$ denote a
regular semisimple matrix in $\gl(n,\C)$, i.e. a matrix which is
diagonalizable with distinct eigenvalues. 
Then, for any choice of
Hessenberg function $h \in H_n$, we call $\Hess(\mathsf{N},h)$ the 
\textbf{regular nilpotent Hessenberg variety (associated to $h$)}
and call $\Hess(\mathsf{S},h)$ the 
\textbf{regular semisimple Hessenberg variety (associated to $h$)}.
Both of the above types of Hessenberg varieties have been much
studied, and it is known, for example, that $\Hess(\mathsf{N},h)$ is
irreducible \cite{an-ty} and possibly singular  \cite{Kostant, InskoYong},
while $\Hess(\mathsf{S},h)$ is smooth, and possibly 
non-connected \cite{ma-pr-sh}. 
As already noted, the essential geometry of the regular semisimple
Hessenberg variety $\Hess(\mathsf{S},h)$
depends only on the conjugacy class of $\mathsf{S}$. In fact, even
more is true: it can be seen, for instance, that the (ordinary or
equivariant) cohomology of $\Hess(\mathsf{S},h)$ is also independent
of the choices of the (distinct) eigenvalues of $\mathsf{S}$
(see e.g. \cite{tymo08}). For concreteness,
henceforth we will always
assume that $\mathsf{N}$ and $\mathsf{S}$ are of the form
\begin{equation*}
\mathsf{N} = 
\begin{pmatrix}
0 & 1 & &  & \\
 & 0 &  1  & \\
 &    & \ddots & \ddots & \\
 &    &   & 0 & 1 \\
 &    &  & & 0 \\ 
\end{pmatrix} \quad 
\textup{ and } \quad
\mathsf{S} = 
\begin{pmatrix} 
\mu_1 &  &  & & \\
&  \mu_2 &  & & \\ 
&  &  \ddots & & \\
& &  &  & \\
& &  &  & \mu_n \\
\end{pmatrix}
\end{equation*}
with respect to the standard basis of $\C^n$, where $\mu_1,\mu_2,\dots,\mu_n$ are mutually distinct complex numbers.
We also note that the dimensions of $\Hess(\mathsf{N},h)$ and
$\Hess(\mathsf{S},h)$ 
have been computed explicitly in terms of the Hessenberg function
\cite{ma-pr-sh, so-ty} and they coincide: 
\begin{align}\label{eq:dim of HessN and HessS}
\dim_{\C} \Hess(\mathsf{N},h) = \dim_{\C} \Hess(\mathsf{S},h) = \sum_{j=1}^n (h(j)-j).
\end{align}
Note that this number is also the number of boxes in
the picture associated to $\NR(h)$. For example, if $h=(3,3,4,5,6,6)$
as in Figure \ref{picture for NR}, then $\dim_{\C} \Hess(\mathsf{N},h)
= \dim_{\C} \Hess(\mathsf{S},h)=6$. 

An essential ingredient in our discussion 
is the presence of 
torus actions on both $\Hess(\mathsf{N},h)$ and $\Hess(\mathsf{S},h)$. In
both cases, the actions are induced from one on the
ambient variety $\Flags(\C^n)$. 
Let
\begin{align} \label{eq:S and T}
S :=\left\{ \left.\begin{pmatrix}
  g  &    &    &     \\ 
    &  g^2  &    &         \\
    &    &  \ddots  &         \\
    &    &    &      g^n  
\end{pmatrix} \ \right| \   \; g\in\C^* \right\}
\quad \text{and} \quad 
T:=\left\{ \left.\begin{pmatrix}
  g_1  &    &    &     \\ 
    &  g_2  &    &         \\
    &    &  \ddots  &         \\
    &    &    &      g_n  
\end{pmatrix} \ \right| \ \; g_i\in\C^*, \ i \in [n] \right\}
\end{align} 
which are subgroup of $\text{GL}(n,\C)$ and hence naturally act on
$\Flags(\C^n)$. We have $\dim_{\C}S=1$ and $\dim_{\C}T=n$.  Notice
that $S$ commutes with the matrix $\mathsf{N}$ (up to a scalar multiplication)
and that $T$ commutes with $\mathsf{S}$. It 
is straightforward to see that the $S$-action on $\Flags(\C^n)$
preserves 
$\Hess(\mathsf{N},h)$ (\cite[Lemma 5.1]{h-t}) and that the $T$-action on
$\Flags(\C^n)$ preserves 
$\Hess(\mathsf{S},h)$. However, note that the $T$-action does \emph{not}
preserve $\Hess(\mathsf{N},h)$ in general.

These torus actions lead us to a study of the equivariant
cohomology of Hessenberg varieties, so we now quickly recall some 
basic background on equivariant topology. Suppose $X$ is a topological space
which admits a continuous action by the torus $T$. The $T$-equivariant
cohomology $H^\ast_T(X)$ is defined to be the ordinary cohomology
$H^\ast(X \times_T ET)$ where $ET\rightarrow BT$ is the universal principal bundle of $T$. In
particular, $H^\ast_T(\pt) = H^\ast(BT)$ and $H^\ast_T(X)$ is an
$H^\ast_T(\pt)$-module.
 We have 
 \begin{align}\label{eq:equiv coh of pt}
 H^*(BT) \cong \text{Sym}_{\Q}(\text{Hom}(T,\C^*)\otimes_{\Z}\Q) 
 \end{align}
so we may identify $H^*(BT)$ with the polynomial ring $\Q[t_1,\dots,t_n]$ where
the element $t_i$ is the first Chern class of the line
bundle over $BT$ corresponding to the projection $T\rightarrow \C^*, diag(g_1,\dots,g_n)\mapsto g_i$. 

Next we recall some standard constructions on the ambient
space $\Flags(\C^n)$ leading to 
a well-known ring
presentation for the equivariant cohomology of $\Flags(\C^n)$. Let $E_i$ denote the $i$-th tautological vector
bundle over $\Flags(\C^n)$; namely, $E_i$ is the sub-bundle of the
trivial vector bundle $\Flags(\C^n)\times \C^n$ over $\Flags(\C^n)$
whose fiber over a point $V_{\bullet}=(V_1\subset \cdots\subset V_n) \in
\Flags(\C^n)$ is exactly $V_i$.  Let 
\begin{align}\label{def of T eq ch in flag}
\TChFlag_i \in H^2_T(\Flags(\C^n))
\end{align}
denote the $T$-equivariant first Chern class of
the tautological line bundle $E_i/E_{i-1}$.  
It is known that $H^{\ast}_{T}(\Flags(\C^n))$ is generated as a ring by
the elements $\TChFlag_1,\dots ,\TChFlag_n$ together with the
$t_1,\dots ,t_n$ (the latter coming from the $H^\ast_T(\pt)$-module structure). Indeed, there is a ring isomorphism 
\begin{align*}
H^{\ast}_{T}(\Flags(\C^n))\cong \mathbb{Q}[x_1,\dots,x_n,t_1,\dots,t_n]/(e_i(x_1,\dots ,x_n)-e_i(t_1,\dots ,t_n) \mid i\in[n])
\end{align*}
defined by sending the polynomial ring variables $x_i$ on the
RHS to the Chern class $\TChFlag_i$ of the $i$-th tautological line bundle and the variables $t_i$
to the Chern classes (which by slight abuse of notation we denote by
the same) $t_i$, 
and the $e_i$ denotes the degree-$i$ elementary symmetric polynomial
in the relevant variables. 
Here and below it should be noted that the degrees of the variables in
question are $2$, i.e. 
\begin{align*}
\deg x_i = \deg t_i = 2 \textup{ for all } i \in [n]. 
\end{align*}
By setting the variables $t_i$ equal to $0$, we can also describe the non-equivariant cohomology ring $H^*(\Flags(\C^n))$ as follows.
Let
\begin{align}\label{def of ch in flag}
\ChFlag_i \in H^2(\Flags(\C^n))
\end{align}
be the (non-equivariant) first Chern class of
the tautological line bundle $E_i/E_{i-1}$.
Then we have
\begin{equation}\label{eq:cohofl}
H^{\ast}(\Flags(\C^n))\cong \Q[x_1,\dots,x_n]/(e_i(x_1,\dots ,x_n) \mid i\in[n])
\end{equation}
where each $x_i$ corresponds to the first Chern class
  $\ChFlag_i$.

As mentioned above, we will also analyze the action of the 
1-dimensional subgroup $S$ of $T$. Let $\C$ temporarily denote the
1-dimensional representation of $S$ defined by the group
homomorphism $diag(g,g^2,\dots,g^n)\mapsto g$ and consider the
associated line bundle 
$ES\times _{S}\C \to BS$. Let 
\begin{align}\label{deg of t in coh}
t \in H^2(BS)
\end{align} 
denote the first Chern class of this
line bundle. As in the case of $T$ above, we 
identify $H^*(BS)$ with the polynomial ring $\mathbb{Q}[t]$. 

A useful and fundamental technique in
  torus-equivariant topology is the restriction to the fixed point set
  of the torus action. If the Serre spectral sequence of the fibration
  $ET\times_{T} X\rightarrow BT$ collapses at the $E_2$-stage, then
  the equivariant cohomology of $X$ (with $\Q$-coefficients) is a free
  $H_{T}^*(\text{pt})$-module, i.e. as an $H_{T}^*(\text{pt})$-module we have $H_{T}^*(X)\cong
  H_{T}^*(\text{pt})\otimes_{\Q} H^*(X)$. In addition, under some technical
  hypotheses on $X$ which are satisfied by the spaces considered
  in this paper,\footnote{For instance, it would certainly suffice if $X$ is
    locally contractible, compact, and Hausdorff.} it follows from the
  localization theorem \cite[p.40]{hsi} that the inclusion
  $X^{T}\hookrightarrow X$ of the $T$-fixed point set induces an
  injection $H_{T}^*(X)\hookrightarrow H_{T}^*(X^{T})$. 
  Any Hessenberg variety (in Lie type A) admits a paving by complex
  affines \cite[Theorem 7.1]{ty}, so their cohomology
  rings are concentrated in even degrees. Hence the corresponding Serre spectral
  sequence of the fibration associated to a continuous group action
  collapses at the $E_2$-stage, and their equivariant
  cohomology rings are free $H_T^*(\text{pt})$-modules \cite[Ch 3, Theorem 4.2]{mi-to}.
  To summarize, we have 
  \begin{align}
  \notag 
&H_T^*(\Flags(\C^n)) \cong H_T^*(\text{pt})\otimes_{\Q} H^*(\Flags(\C^n)) \quad \text{as $H_T^*(\text{pt})$-modules}, \\
  \label{eq:equiv formality of HessN}
&H_S^*(\Hess(\mathsf{N},h)) \cong H_S^*(\text{pt})\otimes_{\Q} H^*(\Hess(\mathsf{N},h)) \quad \text{as $H_S^*(\text{pt})$-modules},
\end{align}
and we also have injections
\begin{align}
  \notag 
  &\iota_1:H^*_T(\Flags(\C^n)) \into H^*_T(\Flags(\C^n)^T), \\
  \label{eq:localization for HessN}
  &\iota_2:H^*_S(\Hess(\mathsf{N},h)) \into H^*_S(\Hess(\mathsf{N},h)^S) 
\end{align}
where all the maps are induced from the inclusions.

Thus, in order to analyze $H^*_S(\Hess(\mathsf{N},h))$, it suffices to understand their
restrictions to the $S$-fixed point set. This will be a fundamental strategy
employed throughout this paper. As a consequence, it is important to
explicitly describe the relevant fixed point sets, to which we now
turn. 

We begin with the most familiar special case, namely $\Flags(\C^n)$; the general case will be analyzed in 
Section~\ref{subsec:fixed points nilp}. For the standard $T$-action on
the ambient variety $\Flags(\C^n)$, it is well-known that the
$T$-fixed point set $\Flags(\C^n)^T$ can be identified with the
permutation group $\mathfrak{S}_n$ on $n$ letters. Indeed, we now fix once and for
all an identification 
\begin{equation}
  \label{eq:Tfixed flags and S_n}
  \mathfrak{S}_n \stackrel{\cong}{\rightarrow} \Flags(\C^n)^T
\end{equation}
which takes a permutation $w \in \mathfrak{S}_n$ to the flag specified
by $V_i := \text{span}_\C\{e_{w(1)}, \ldots, e_{w(i)}\}$, where
$\{e_1,\ldots,e_n\}$ denotes the standard basis of
$\C^n$. (Alternatively, given the usual identification of
$\Flags(\C^n)$ with $\text{GL}(n,\C)/B$, we take $w$ to the coset represented
by the standard permutation matrix associated to $w$
whose $(w(j),j)$-th entry is required to be $1$ for
  each $j$ and otherwise entries are $0$.)  Restricting our attention
to the subtorus $S \subset  T$, it is straightforward to check that
the $S$-fixed point set $\Flags(\C^n)^S$ of the flag variety
$\Flags(\C^n)$ are also given by the above set $\Flags(\C^n)^T$, i.e. 
\begin{align}\label{S-fixed = T-fixed} \Flags(\C^n)^S =
  \Flags(\C^n)^T.
\end{align}
From here it also quickly follows that 
\[
\Hess(\mathsf{N},h)^S=\mathcal \Hess(\mathsf{N},h)\cap (\Flags(\C^n))^T.
\] 
Thus the set of $S$-fixed point set $\Hess(\mathsf{N},h)^S$ is a subset of
$\Flags(\C^n)^T$, and through our fixed identification $\Flags(\C^n)^T \cong \Sn$
from~\eqref{eq:Tfixed flags and S_n} we henceforth view $\Hess(\mathsf{N},h)^S$ as a subset of
$\Sn$. 

Based on the above discussions, we may consider the commutative diagram 
\begin{equation}\label{eq:cd}
\begin{CD}
H^{\ast}_{T}(\Flags(\C^n))@>{\iota_1}>> H^{\ast}_{T}(\Flags(\C^n)^{T})\cong\displaystyle \bigoplus_{w\in \Sn} \Q[t_1,\dots,t_n]\\
@V{}VV @VV{\pi_1}V\\
H^{\ast}_{S}(\Flags(\C^n))@>{\iota'_1}>> H^{\ast}_{S}(\Flags(\C^n)^{S})\cong\displaystyle \bigoplus_{w\in \Sn} \Q[t]\\
@V{}VV @VV{\pi_2}V\\
H^{\ast}_{S}(\Hess(\mathsf{N},h))@>{\iota_2}>> H^{\ast}_{S}(\Hess(\mathsf{N},h)^{S})\cong\displaystyle \bigoplus_{w\in \Hess(\mathsf{N},h)^{S}\subset  \Sn} \Q[t]
\end{CD}
\end{equation}
where all the maps are induced from the inclusion maps on underlying
spaces.  Note that all of $\iota_1$, $\iota'_1$, and $\iota_2$ are
injective as explained above since $\iota'_1$ is a special case of
$\iota_2$. The vertical map $\pi_1$ is given by
  $\pi_1=\bigoplus_{w\in\Sn}\pi_1^w$, where
  $\pi_1^w:\Q[t_1,\dots,t_n]\rightarrow\Q[t]$ is a ring homomorphism
  sending each $t_i$ to $it$; this is because $\pi_1^w$
  is induced by
  the projection map $\text{Lie}(T)^*\rightarrow \text{Lie}(S)^*$ and
  we use the identification \eqref{eq:equiv coh of pt}.
  The vertical map $\pi_2$ is the obvious projection.

By slight abuse of
notation, for $g\in H^{\ast}_{T}(\Flags(\C^n))$
we denote its image $\iota_1(g)$ 
also by $g$. Also, for an element 
$g\in\bigoplus_{w\in \Sn} \Q[t_1,\dots,t_n]$ 
 we will denote its $w$-th component by $g(w)$. Furthermore, we let
$t_i\in\bigoplus_{w\in \Sn} \Q[t_1,\dots,t_n]$ 
denote the ``constant polynomial'' with value $t_i$ at each $w$,
i.e. $t_i(w)=t_i$ for all $w\in \Sn$. 
We apply the same convention for $H^*_S(\Hess(\mathsf{N},h))\stackrel{\iota_2}{\into} H^*_S(\Hess(\mathsf{N},h)^S)$ 
and so we have the ``constant'' class 
\begin{align}\label{eq:constant polynomials}
t\in\bigoplus_{w\in \Hess(\mathsf{N},h)^{S}} \Q[t]. 
\end{align}
This 
is the image of $t\in H^2(BS)$ (defined in \eqref{deg of t in coh})
under the canonical homomorphism $H^*(BS)\rightarrow H^*_S(\Hess(\mathsf{N},h))$ composed with $\iota_2$.
For the $T$-equivariant Chern class $\TChFlag_i\in H^{2}_{T}(\Flags(\C^n))$ introduced in \eqref{def of T eq ch in flag}, we have 
\begin{align}\label{loc of tilde tau}
 \TChFlag_i(w) = t_{w(i)} \quad \text{for } w\in\Sn.
\end{align} 
This can be seen by reading off the weight of the $T$-action on the
fiber of the tautological line bundle $E_i/E_{i-1}$ over the
permutation flag $V_{\bullet}\in\Flags(\C^n)^T$ given by
$V_i=\text{span}_{\C}\{e_{w(1)},\dots,e_{w(i)}\}$.

\subsection{The $S$-fixed point set of $\Hess(\N,h)$}\label{subsec:fixed
  points nilp} 
The injectivity of $\iota_2$ in~\eqref{eq:localization for HessN}
shows 
  that we
  can analyze $H^*_S(\Hess(\mathsf{N},h))$ by viewing its elements as
  lists of polynomials with one coordinate for each (isolated) fixed
  point in $\Hess(\mathsf{N},h)^S$. To successfully implement this
  strategy, we must understand 
the fixed point set 
  $\Hess(\mathsf{N},h)^S$, viewed 
as a subset of $\Sn$, in more detail. This is the goal of this
section. 
In addition, we introduce some terminology associated to these fixed
points, as well as a Hessenberg function $h_w$ associated to a
permutation $w$
which will be 
useful for our later arguments. Indeed, it turns out that we can characterize
$\Hess(\mathsf{N},h)^S$ in terms of these functions $h_w$
(Proposition~\ref{prop:equivalence of fixed point conditions}). 

We will use the 
standard one-line notation $w = (w(1) w(2) \cdots w(n))$ for
permutations in $\Sn$. 
It will occasionally 
be convenient for us to think of permutations in $\Sn$ as permutations on $\{0\}\cup[n]$, i.e. we use a convention 
\begin{equation}\label{eq:convention on w(0)}
w(0)=0 \quad \text{for all $w\in\Sn$}.
\end{equation}

As a first step, we have the following. 

\begin{lemma} \label{lemma:Hess fixed points}
The $S$-fixed point set $\Hess(\mathsf{N},h)^S \subset  \Sn$ of
$\Hess(\mathsf{N},h)$ is given by 
\begin{equation*} \label{eq:H-1}
 \Hess(\mathsf{N},h)^S=\{w\in \Sn \mid w^{-1}(w(j)-1)\le h(j)\text{ for all } j\in [n] \}.
\end{equation*}
\end{lemma}

\begin{proof}
Since $\Hess(\mathsf{N},h)^S=\Hess(\mathsf{N},h)\cap (\Flags(\C^n))^T$, it suffices to show that for any 
$w\in \Sn \cong \Flags(\C^n)^T$, the condition 
$w^{-1}(w(j)-1)\le h(j) \ (j=1,2,\dots,n)$ is equivalent to the condition 
$w\in\mathcal \Hess(\mathsf{N},h)$. From
\eqref{eq:def-general Hessenberg} and \eqref{eq:Hessenberg subspace} we see immediately that 
$
w\in\Hess(\mathsf{N},h)$ if and only if $w^{-1}\mathsf{N}w\in H(h)
$
where we regard $w$ as a permutation matrix, i.e. the matrix with
$(w(j),j)$-th entry equal to $1$ for each $j$ and all other entries
equal to $0$. 
Since our $\mathsf{N}$ is the regular nilpotent matrix sending $e_1 \mapsto 0$ and $e_j \mapsto e_{j-1}$ for $j>1$, 
we have that $\mathsf{N}w(e_j)=\mathsf{N}(e_{w(j)}) = 0$ if $w(j)=1$ and
$\mathsf{N}w(e_j)=e_{w(j)-1}$ if $w(j) \neq 1$. So
$w^{-1}\mathsf{N}w(e_j)=0$ if $w(j)=1$ and $w^{-1}\mathsf{N}w(e_j)=e_{w^{-1}(w(j)-1)}$
if $w(j)\neq 1$. Thus $w^{-1}\mathsf{N}w\in H(h)$ precisely means that 
$w^{-1}(w(j)-1)\le h(j)$ for all $j\in [n]$, where we follow the
notational convention of~\eqref{eq:convention on w(0)}. 
\end{proof}

In words, the condition in the above lemma can be stated as follows.
Let $w \in \Sn$ and let $w = (w(1) \ w(2) \ \ldots \ w(n))$ be its 
one-line notation. Suppose that a consecutive pair of integers $k, k+1$ is
inverted in the one-line notation of $w$, i.e. $k$ appears to the
\emph{right} of $k+1$, and suppose in this situation that $k+1$
appears in the $j$-th place (so $w(j)=k+1$) while $k$ appears in the
$\ell$-th place (so $w(\ell)=k=w(j)-1$). Then the requirement of the
condition is that $\ell \leq h(j)$. Informally, the Hessenberg
function gives a restriction on ``how far to the right'' of $w(j)=k+1$
the value $w(j)-1=k$ is allowed to appear. Note that, for any $j$, if
the pair $w(j)=k+1$ and $w(j)-1=k$ are \emph{not} inverted in the
one-line notation of $w$, i.e. $k$ appears to the \emph{left} of
$k+1$, then the condition is immediate, since by definition Hessenberg
functions satisfy $h(j) \geq j$. \\

The inverted 
pairs $(k+1,k)$ play a special role in analyzing the $S$-fixed point set 
of $\Hess(\mathsf{N},h)$. Motivated by this, we introduce some terminology. 

\begin{definition}\label{def:N-inversion}
  Let $w \in \Sn$ be a permutation and let $i,j \in [n]$. We
  say that $\P = (i,j)$ is an \textbf{$\mathsf{N}$-inversion} if
  $i<j$ and $w(i)=w(j)+1$. We refer to $i$ $($respectively $j)$ as the
  \textbf{left} $($respectively \textbf{right$)$
    position} of the $\mathsf{N}$-inversion. Given an $\mathsf{N}$-inversion $\P=(i,j)$
  we let $LP(\P) := i$ denote its left position and $RP(\P) := j$ its right
  position. 
\end{definition}

Given a permutation $w \in \Sn$ we now define 
\begin{equation*}
\NInv{w} := \{ \P = (i,j) \in [n] \times [n] \mid \P \textup{ is an
  $\mathsf{N}$-inversion in } w \}. 
\end{equation*}
In the following it will be useful to focus on certain subsets of
$\NInv{w}$. Let 
$j \in [n]$. We define 
\begin{equation*}
D_w(j):= \{\P \in \NInv{w} \mid 1 \leq LP(\P) \leq j \textup{ and } j<RP(\P) \leq n \}.
\end{equation*}
In words, the set $D_w(j)$ consists of the $\mathsf{N}$-inverted pairs whose 
left position is \emph{at or to the left} of the $j$-th place, and
whose right position is \emph{strictly to the right} of the $j$-th place. 
The following is a quick consequence of the definition. 

\begin{lemma} \label{lemma:Dwj empty}
Let $w\in\Sn$ and $j\in[n]$. Then $D_w(j)=\emptyset$ if and only if $\{w(1),w(2),\dots,w(j)\}=\{1,2,\dots,j\}$.
\end{lemma}

\begin{proof}
 If $\{w(1),w(2),\dots,w(j)\}=\{1,2,\dots,j\}$ then clearly
  $D_w(j)=\emptyset$ from the definition. Now suppose
  $D_w(j)=\emptyset$.  
 Take an element $w(p)\in\{w(1),w(2),\dots,w(j)\}$ $(1 \leq p \leq  j)$. Suppose $q$ is such that $w(p)-1=w(q)$. Then from the assumption
  that $D_w(j)=\emptyset$ we must have $q \le j$. That is,  if $w(p)\neq1$, then $w(p)-1(=w(q))$ is also contained in $\{w(1), \ldots, w(j)\}$. 
  This means that $\{w(1), \ldots, w(j)\}$ is of the form $\{k\in[n]\mid k\leq k_0\}$ for some $k_0\in[n]$, but since the cardinality is $j$, it has to be $\{w(1), \ldots, w(j)\}=\{1, \ldots,j\}$ as desired.
\end{proof}

Our next step is to define a map $w \mapsto h_w$ which associates to
any permutation $w \in \Sn$ a Hessenberg function $h_w$. The
Hessenberg function $h_w$ is the minimal Hessenberg function $h$ such
that $w \in \Hess(\mathsf{N},h)$, in a sense to be made precise below.
Specifically, given $w \in \Sn$ we define
\begin{equation}\label{eq:def hw}
h_w(j):=\begin{cases} j \quad&\textup{if $D_w(j)=\emptyset$}\\
\max\{RP(\P) \mid \P \in D_w(j))\} \quad&\textup{if $D_w(j)\not=\emptyset$.}\end{cases}
\end{equation}

We first prove that the function $h_w$ thus defined is in fact a
Hessenberg function. 

\begin{lemma}
Let $h_w$ be as above. Then $h_w\in H_n$.
\end{lemma} 

\begin{proof} 
We must show that $h_w(i)\geq i$ for all $i \in [n]$ and $h_w(i+1)
\geq h_w(i)$ for all $i \in [n-1]$. First notice 
that if $D_w(j) \neq \emptyset$, then every element of $\{RP(\P) \mid
\P \in D_w(j)\}$ is $>j$ by definition of $D_w(j)$. Thus the first
claim follows from the definition of $h_w$. Next we check the second
claim. Fix an $i$ in $[n-1]$. We take cases. First suppose
$D_w(i)=\emptyset$. Then  $h_w(i)=i$ by definition of $h_w$ and since
we have already seen that
$h_w(i+1)\geq i+1$, we obtain $h_w(i) \leq h_w(i+1$) as desired. Next suppose
$D_w(i) \neq \emptyset$ and $D_w(i+1) = \emptyset$. By
Lemma~\ref{lemma:Dwj empty} this means $\{w(1), \ldots, w(i+1)\} =
\{1,2,\ldots,i+1\}$ but $\{w(1),\ldots, w(i)\} \neq
\{1,\ldots,i\}$. It follows that $D_w(i)$ consists of a single
$\mathsf{N}$-inverted pair $\P$, and that $RP(\P)=i+1$. In particular
$h_w(i)=i+1$. Hence $h_w(i+1) \geq i+1 = h_w(i)$ and the claim
holds in this case. Finally suppose both $D_w(i) \neq \emptyset$ and
$D_w(i+1) \neq \emptyset$. If the $\mathsf{N}$-inverted pair achieving the
maximum of the right position of $D_w(i)$ is also an element of $D_w(i+1)$, then clearly
$h_w(i+1) = \max \{RP(\P) \mid \P \in D_w(i+1)\} \geq \max \{RP(\P) \mid
\P \in D_w(i)\} = h_w(i)$ and the claim holds. Otherwise, the
maximum of $\{RP(\P) \mid \P \in D_w(i)\}$ must be $i+1$, and
$h_w(i)=i+1$. Since $h_w(i+1) \geq i+1$, the claim also holds in this
case. We have checked all cases so this completes the proof. 
\end{proof} 

The following reformulation of the definition of $h_w$ is sometimes
useful. In the case when $D_w(j)\not=\emptyset$, it can be seen from
the definitions that the value $h_w(j)$ may also be expressed as
\begin{equation} \label{eq:hw def reformulation}
h_w(j)=\max\{ w^{-1}(w(p)-1)\mid 1\le p\le j\}
\quad \text{if } D_w(j)\not=\emptyset.
\end{equation}
Note also that we have 
\begin{equation} \label{eq:hw def reformulation inequality}
w^{-1}(w(j)-1) \leq h_w(j)
\quad \text{for all $j\in[n]$}
\end{equation}
by~\eqref{eq:hw def reformulation} together with the fact
$w^{-1}(w(j)-1)\leq j$ 
if $D_w(j) = \emptyset$ by Lemma~\ref{lemma:Dwj empty} (see also our convention \eqref{eq:convention on w(0)}). 

Before stating the next proposition we recall a natural partial
ordering on Hessenberg functions. 
\begin{definition}\label{def:partial order} 
  Let $h', h \in H_n$. Then we say $h' \subset h$ if $h'(j)
  \leq h(j)$ for all $j \in [n]$. 
\end{definition}
The relation $h' \subset h$ is evidently a partial order on $H_n$.  Note that from the definition of $\Hess(\mathsf{N},h)$ it is immediate that 
$
h' \subset h$ implies $\Hess(\mathsf{N},h') \subset \Hess(\mathsf{N},h)
$
which explains our choice of notation. 

\begin{remark} 
Mbirika and Tymoczko \cite{mb-ty13} denote the above partial order
with the symbol $\le$ instead of the symbol $\subset$ which we use above. 
In later sections we additionally introduce a refinement of the above
partial order to a total order $\preceq$. 
\end{remark} 
With the terminology in place, we can give equivalent characterizations
of the permutations $w \in \Sn$ which lie in the $S$-fixed point set
of $\Hess(\mathsf{N},h)$. 

\begin{proposition} \label{prop:equivalence of fixed point conditions}
Let $w\in \Sn$ and let $h \in H_n$. Then the following are equivalent: 
\begin{enumerate}
\item $w \in \Hess(\mathsf{N},h)^S$, 
\item $w^{-1}(w(j)-1)\le h(j)$ for all $j \in [n]$, 
\item $h_w\subset h$.
\end{enumerate}
\end{proposition}

\begin{proof}
The equivalence of (1) and (2) is the content of
Lemma~\ref{lemma:Hess fixed points} above. 
Also, it is easy to see that (3) implies (2) since we have \eqref{eq:hw def reformulation inequality} and by assumption we have
$h_w(j) \leq h(j)$ for all $j \in [n]$.
Hence it suffices to prove
that (2) implies (3).

Suppose $w^{-1}(w(j)-1)\le h(j)$ for all $j \in [n]$. We wish to prove
that $h_w(j) \leq h(j)$ for all $j \in [n]$. We take cases. Suppose
$D_w(j)=\emptyset$. Then $h_w(j)=j \leq h(j)$, where the inequality
holds because $h \in H_n$. Hence the claim holds in this case. Now suppose
$D_w(j) \neq \emptyset$. Then by~\eqref{eq:hw def reformulation} we
have $h_w(j) = \max \{ w^{-1}(w(p)-1) \mid 1 \leq p \leq j \}$ but the
assumption shows $w^{-1}(w(p)-1) \leq h(p) \leq h(j)$ for all $p$ with $1 \leq p \leq j$. Hence $h_w(j) \leq h(j)$ also in this case. 
\end{proof}

For a fixed permutation $w\in\Sn$, the above proposition implies that $h_w$ is the unique
minimum with respect to the partial order $\subset$ in the set $\{h \in H_n \mid w \in \Hess(\mathsf{N},h)^S\}$. 

Finally, we record the following property of $h_w$ which we will use in Section~\ref{sec:y}.

\begin{lemma} \label{lemma:corner}
Let $w \in \Sn$ and $j\in[n-1]$. Suppose $D_w(j)\neq\emptyset$ $($i.e. $h_w(j) \geq j+1$$)$. Then
\[
h_w(j) = w^{-1}(w(j)-1) \ \textup{ if and only if } \  h_w(j-1) < h_w(j).
\]
\end{lemma}

\begin{proof}
First suppose $h_w(j-1)<h_w(j)$. We wish to show that 
$h_w(j) = w^{-1}(w(j)-1)$. 
We take cases.
If $D_w(j-1)\neq\emptyset$, then by~\eqref{eq:hw def reformulation} we have 
\[
h_w(j-1) = \text{max}\{w^{-1}(w(p)-1) \mid 1 \leq p \leq j-1\} < \text{max}\{w^{-1}(w(p)-1) \mid 1 \leq p \leq j\}
=h_w(j)
\]
This implies that $h_w(j) = w^{-1}(w(j)-1)$. 
Next suppose that $D_w(j-1)=\emptyset$.
Suppose  
in order to obtain a contradiction that $h_w(j) \neq w^{-1}(w(j)-1)$. 
By assumption, we have $D_w(j) \neq \emptyset$ so
  $h_w(j)=\text{max}\{w^{-1}(w(p)-1) \mid 1 \leq p \leq j\}$. If $h_w(j) \neq w^{-1}(w(j)-1)$
  then the maximum must be achieved by a value $w^{-1}(w(p)-1)$ for $1
  \leq p \leq j-1$, and since $h_w(j) \geq j$, this
  implies $D_w(j-1) \neq \emptyset$. This is a contradiction and we
  conclude $h_w(j)=w^{-1}(w(j)-1)$. 

Now suppose $h_w(j) = w^{-1}(w(j)-1)$. We wish to show $h_w(j-1)<h_w(j)$. 
We again take cases.
If $D_w(j-1) =\emptyset$, then by definition \eqref{eq:def hw} of $h_w$, we have $h_w(j-1) = j-1 < j \leq h_w(j)$, as desired.
If $D_w(j-1) \neq \emptyset$, from \eqref{eq:hw def reformulation} we have
\[
h_w(j-1) = \text{max}\{w^{-1}(w(p)-1) \mid 1 \leq p \leq j-1\}.
\]
But from the assumption $D_w(j) \neq\emptyset$ and also from \eqref{eq:hw def reformulation}, we have
\[
h_w(j) = \text{max}\{w^{-1}(w(p)-1) \mid 1 \leq p \leq j\} = w^{-1}(w(j)-1).
\]
Hence, the maximum of the set is reached at $p=j$, implying that
  the values for $1 \leq p < j$ are strictly less than
  $w^{-1}(w(j)-1)$. Thus $h_w(j-1) < h_w(j)$ as desired. 
\end{proof}

\bigskip
\section{Statement of Theorem~\ref{theorem:reg nilp Hess cohomology}, the equivariant version of Theorem A} \label{sec:f}

In this section we state the equivariant version of Theorem A. 
Consider the restriction homomorphism 
\begin{equation}\label{eq:rest}
H^*_T(\Flags(\C^n)) \to H^*_S(\Hess(\mathsf{N},h)).
\end{equation}
For the tautological line
bundle $E_i/E_{i-1}$ over $\Flags(\C^n)$ restricted to
$\Hess(\mathsf{N},h)$, let
\begin{align}\label{def of S eq ch in Nil}
\SChNil_i\in H_S^2(\Hess(\mathsf{N},h))
\textup{ for } i\in[n]
\end{align}
denote its $S$-equivariant first Chern class.
That is, $\SChNil_i$ is the image of
$\TChFlag_i$ (see \eqref{def of T eq ch in flag}) under
\eqref{eq:rest}.  We next analyze some algebraic relations satisfied
by the $\SChNil_i$. For this purpose, we now introduce some
polynomials
$f_{i,j}(x_1,\dots,x_n,t)\in\mathbb{Q}[x_1,\dots,x_n,t]$ for
$n\geq i\geq j\geq 1$.

First we define 
\begin{equation} \label{eq:f-1}
p_i:=\sum_{k=1}^i(x_k-kt) \in \Q[x_1, \ldots, x_n, t] \quad \textup{ for } i
\in [n]. 
\end{equation} 
For convenience we also set $p_0 :=0$. 

\begin{definition}\label{definition:fij}
Let $(i,j)$ be a pair of natural numbers satisfying  $n \geq i \geq j \geq 1$. 
We define polynomials $f_{i,j}$ inductively as follows. As the base
case, when $i=j$, we
define 
\begin{equation*}
f_{j,j}:=p_j \quad \textup{ for } j \in [n]. 
\end{equation*}
Proceeding inductively, for $(i,j)$ with $n\ge i>j\ge 1$ we define 
\begin{equation} \label{eq:f-3}
f_{i,j}:=f_{i-1,j-1}+\big(x_j-x_i-t\big)f_{i-1,j}
\end{equation}
where we take the convention $f_{*,0}:=0$ for any $*$. 
\end{definition} 

Informally, we may visualize each $f_{i,j}$ as being associated to the lower-triangular $(i,j)$-th entry in an $n \times n$ matrix, as follows: 
\begin{equation}\label{eq:matrix fij}
\begin{pmatrix}
f_{1,1} & 0 & \cdots & \cdots & 0 \\
f_{2,1} & f_{2,2} & 0 & \cdots &  \\
f_{3,1} & f_{3,2} & f_{3,3} & \ddots & \\
\vdots &        &                 &  &  \\
f_{n,1} & f_{n,2} & \cdots &  & f_{n,n} 
\end{pmatrix}.
\end{equation}

\begin{example}\label{exam:f}
Suppose $n=4$. Then the $f_{i,j}$ have the following form. 
\begin{align*}
&f_{i,i}=p_i \ (1\leq i\leq 4) \\
&f_{2,1}=(x_1-x_2-t)p_1 \\
&f_{3,2}=(x_1-x_2-t)p_1+(x_2-x_3-t)p_2  \\
&f_{4,3}=(x_1-x_2-t)p_1+(x_2-x_3-t)p_2+(x_3-x_4-t)p_3 \\
&f_{3,1}=(x_1-x_3-t)(x_1-x_2-t)p_1  \\
&f_{4,2}=(x_1-x_3-t)(x_1-x_2-t)p_1+(x_2-x_4-t)\{(x_1-x_2-t)p_1+(x_2-x_3-t)p_2\} \\
&f_{4,1}=(x_1-x_4-t)(x_1-x_3-t)(x_1-x_2-t)p_1
\end{align*}
\end{example}

Now let $\Q[x_1,\dots,x_n,t]$ denote the polynomial ring equipped with
a grading defined by 
\begin{align*}
\deg x_i =2 \textup{ for all } i \in [n] \quad \textup{and}\quad  \deg t=2. 
\end{align*}
Note that $\Q[x_1,\ldots,x_n,t]$ is evidently a $\Q[t]$-algebra. 
We define a graded $\Q[t]$-algebra homomorphism $\tilde\varphi_h$ by 
\begin{align}\label{def of equiv phi_h}
\tilde\varphi_h : \Q[x_1,\dots,x_n,t] \rightarrow H_S^*(\Hess(\mathsf{N},h))
\quad ;\quad 
x_i \mapsto \SChNil_i, \ t\mapsto t
\end{align}
where $\SChNil_i$ (defined in \eqref{def of S eq ch in Nil}) is the
$S$-equivariant first Chern class of the tautological line bundle
$E_i/E_{i-1}$ restricted to $\Hess(\mathsf{N},h)$ and the class $t \in H_S^*(\Hess(\mathsf{N},h))$ is the Chern class in~\eqref{deg of t in coh}. We are now ready to state the main technical result of this manuscript, the content of
which is that the map $\tilde\varphi_h$ induces an isomorphism of
graded $\Q[t]$-algebras between $H^\ast_S(\Hess(\mathsf{N},h))$ and
the quotient of $\Q[x_1,\ldots,x_n,t]$ by the ideal $\Ih$ generated by
a certain subset of the polynomials $f_{i,j}$ defined above. 
The proof of Theorem~\ref{theorem:reg nilp Hess cohomology} occupies
Sections~\ref{sec:fproperty} through~\ref{sec:proof of first main theorem}. 

\begin{theorem} \label{theorem:reg nilp Hess cohomology}
Let $n$ be a positive integer and $h: [n] \to [n]$ a Hessenberg function. Let $\Hess(\mathsf{N},h) \subset \Flags(\C^n)$ denote the corresponding regular nilpotent Hessenberg variety equipped with the action of the 1-dimensional subgroup $S$ described in Section~\ref{subsec:setup}. Then the restriction map 
\begin{equation*}
H^*_T(\Flags(\C^n)) \to H^*_S(\Hess(\mathsf{N},h))
\end{equation*} 
is surjective, and 
there is an isomorphism of graded $\mathbb{Q}[t]$-algebras
\begin{equation*}
H^{\ast}_{S}(\mathcal \Hess(\mathsf{N},h))\cong \mathbb{Q}[x_1,\dots,x_n,t]/\Ih 
\end{equation*}
sending $x_i$ to $\SChNil_i$ and $t$ to $t$, 
where $\SChNil_i$ (defined in \eqref{def of S eq ch in
  Nil}) is the $S$-equivariant first Chern class of the
tautological line bundle  restricted to $\Hess(\mathsf{N},h)$ and we
identify $H^{\ast}(BS) \cong \mathbb{Q}[t]$.
Here the ideal $\Ih$ is defined by 
\begin{equation}\label{eq:def Ih} 
\Ih:= (f_{h(j),j} \mid 1\leq j\leq n ).
\end{equation} 
\end{theorem}

\bigskip
Using the association of the polynomials $f_{i,j}$
  with the $(i,j)$-th entry of the matrix~\eqref{eq:matrix fij}, the
  ideal $\Ih$ can
visually be described as being generated by the $f_{i,j}$ in the boxes at the bottom of each column in the the picture associated to the Hessenberg subspace $H(h)$ defined in \eqref{eq:Hessenberg subspace} (see Figure~\ref{picture for NR}). For instance, when $h=(3,3,4,5,6,6)$, the generators are $f_{3,1}, f_{3,2}, f_{4,3}, f_{5,4}, f_{6,5}, f_{6,6}$.

\begin{remark}\label{rema:peterson}
   The above generalizes known results.  Specifically, 
    the special case when $h(j)=j+1$ for $j  \in [n-1]$ has been
    well-studied and the corresponding variety $\Hess(\mathsf{N},h)$
    is called the 
    \textbf{Peterson variety}
    $Pet_n$ $($in Lie type A$)$. Our result above generalizes 
    a known presentation of 
    $H^*_S(Pet_n)$ \cite{fu-ha-ma}. Indeed, for $1\leq j\leq n-1$, we obtain from
    \eqref{eq:f-3} and \eqref{eq:f-1} that
\begin{align*}
f_{j+1,j}=f_{j,j-1}+(-p_{j-1}+2p_j-p_{j+1}-2t)p_j
\end{align*}
and since $f_{n,n}=p_n$ we can inductively see that
\begin{align*}
H^{\ast}_{S}(Pet_n)
                 &\cong \mathbb{Q}[p_1,\dots,p_{n-1},t]/\big((-p_{j-1}+2p_j-p_{j+1}-2t)p_j \mid 1\leq j\leq n-1\big).
\end{align*}
where we take the convention $p_0=p_n=0$.
This agrees with \cite{fu-ha-ma}. 

Also, when $h(i)=n$ for $i\in[n]$, the 
  associated variety is the 
  full flag variety $\Flags(\C^n)$. In this case, we will see later
  $($in \eqref{eq:f check and power sum} and \eqref{eq:f check ideal
    and power sum ideal}$)$ an explicit relationship between the
  generators $\check{f}_{n, j}$ of our ideal $\check \Ih = \check
  I_{(n,n,\ldots,n)}$ and the power sums $\mathsf{p}_j(x) =
  \mathsf{p}_j(x_1, \ldots, x_n) := \sum_{k=1}^n x_k^j$, thus relating
  our presentation with the usual Borel presentation of the cohomology
  of $\Flags(\C^n)$ as
  in~\eqref{eq:cohofl}. 
\end{remark}

\bigskip
\section{Properties of the $f_{i,j}$} \label{sec:fproperty}

In this section, in preparation for the proof of
Theorem~\ref{theorem:reg nilp Hess cohomology}, 
we further analyze the polynomials $f_{i,j}$ defined in
Definition~\ref{definition:fij}. 
The results in this section, particularly Corollary~\ref{lemm:f-2}, set the stage for the 
proof in Section~\ref{sec:y} that the map $\tilde\varphi_h$
of~\eqref{def of equiv phi_h} induces a
well-defined map 
\begin{align*}
\varphi_h : \Q[x_1,\dots,x_n,t]/\Ih \rightarrow H_S^*(\Hess(\mathsf{N},h)).
\end{align*}

We begin with the following. 

\begin{lemma}\label{relation for Ih and Ih'}
The ideal $\Ih$ defined in~\eqref{eq:def Ih}  contains $f_{i,j}$ for all $i\geq h(j)$. 
In particular, if $h\subset h'$ in the sense of Definition \emph{\ref{def:partial order}}, then $I_{h'} \subset \Ih$.
\end{lemma}

\begin{proof}
We prove that $\Ih$ contains $f_{i,j}$ for all $i\geq h(j)$ by induction on $j$. 
When $j=1$, by the recursive relation \eqref{eq:f-3} we have 
\[
f_{i+1,1}=(x_1-x_{i+1}-t)f_{i,1},
\]
and hence the assumption $f_{h(1),1}\in\Ih$ implies that $f_{i,1}\in\Ih$ for $i\geq h(1)$.

Now, assume that the claim holds for $j-1$, that is, $f_{i,j-1}\in\Ih$ for all $i\geq h(j-1)$.
We show that $f_{i,j}\in\Ih$ for all $i\geq h(j)$.
Since $f_{h(j),j}\in\Ih$ by definition, again by induction on $i$, we may suppose $f_{i,j}\in\Ih$ for some $i\geq h(j)$ and then we must prove that $f_{i+1,j}\in\Ih$. By the recursive relation \eqref{eq:f-3}, we have
\[
f_{i+1,j}=f_{i,j-1}+\big(x_j-x_{i+1}-t\big)f_{i,j}.
\]
Since we have $i\geq h(j)\geq h(j-1)$, the inductive hypothesis
implies that the RHS of this identity is contained in $\Ih$, and hence
we obtain $f_{i+1,j}\in\Ih$ as desired. 
\end{proof}

The definition of $\tilde\varphi_h$ in~\eqref{def of equiv phi_h} sends the variables $x_i$ to
$\SChNil_i$. 
In the next section we will prove that this map sends each $f_{h(j),j}$ to zero.
In preparation for this, we investigate below some general properties of $f_{i,j}(w)$, i.e.
the $w$-th component of $f_{i,j}$, at each fixed point $w \in
\Flags(\C^n)^S=\Sn$. 
More precisely, letting
\[
\SChFlag_i \in H^{\ast}_{S}(\Flags(\C^n))
\]
be the $S$-equivariant first Chern class of the tautological line bundle $E_i/E_{i-1}$ over $\Flags(\C^n)$ described in Section~\ref{sec:H}, we study the image of 
$f_{i,j}(\SChFlag_1,\dots,\SChFlag_n,t)$ under the localization map
\begin{equation}\label{eq:localization of S-eqvt coh}
\iota'_1: H^{\ast}_{S}(\Flags(\C^n))
\rightarrow \bigoplus_{w\in\Sn} \Q[t].
\end{equation}

For the rest of this section, by slight abuse of notation we write 
\begin{align}\label{def of f as a polynomial of tau}
f_{i,j}=f_{i,j}(\SChFlag_1,\cdots,\SChFlag_n,t)
\end{align} 
i.e. the elements of $H^\ast_S(\Flags(\C^n))$ obtained by ``evaluating at the $\SChFlag_i$''. 
Fix $w\in \Sn$ 
and denote by 
$
f_{i,j}(w) \in\Q[t]
$
the $w$-th component of the image of $f_{i,j}$ under the localization
map $\iota'_1$ above.
Recall from Section~\ref{sec:H} that the $w$-th component of the map $\pi_1=\bigoplus_{w\in\Sn}\pi_1^w$ in the commutative diagram~\eqref{eq:cd} is a ring homomorphism sending each $t_i$ to $it$.
Combined with \eqref{loc of tilde tau}, this
implies that 
\begin{equation*} 
\SChFlag_i(w)=w(i)t \quad \text{for } i\in[n].
\end{equation*} 
It then follows from the definition of the $f_{i,j}$
(Definition~\ref{definition:fij}) that 
\begin{equation} \label{eq:fij at w}
\begin{split}
f_{j,j}(w)&=\sum_{k=1}^j(w(k)-k)t\quad (1\leq j\leq n),\\
f_{i,j}(w)&=f_{i-1,j-1}(w)+(w(j)-w(i)-1)t\cdot f_{i-1,j}(w)\quad (n\ge i>j\ge 1).
\end{split}
\end{equation}

The inductive nature of the $f_{i,j}$ allows us to conclude the
following.

\begin{lemma} \label{lemma:fij vanish} 
Let $h \in H_n$. Let $f_{i,j}=f_{i,j}(\SChFlag_1, \ldots, \SChFlag_n, t)$ be defined as above. 
If $f_{h(j),j}(w)=0$ for all $j\in [n]$, then $f_{i,j}(w)=0$ for all
$j \in [n]$ and $i \geq h(j)$. 
\end{lemma}

\begin{proof}
 Recall that $I_h$ is by definition the ideal of $\Q[x_1,\dots,x_n,t]$ generated by the
 $f_{h(j),j}$ for $j \in [n]$. From Lemma~\ref{relation for Ih and Ih'} we know that if $i
 \geq h(j)$ then $f_{i,j} \in I_h$.  By assumption, each $f_{h(j),j}$
 lies in the kernel of the ring homomorphism 
\[
\Q[x_1, \ldots, x_n, t] \to H^*_S(\Flags(\C^n)) \to \Q[t]
\]
where the first arrow sends $x_i$ to $\tau_i^S$ and the second map is
the $w$-th coordinate of the localization map
in~\eqref{eq:localization of S-eqvt coh}. Thus the ideal $I_h$ also
lies in the kernel as well, and hence also $f_{i,j}$ for $i \geq
h(j)$. 
\end{proof}

To motivate the following discussion, it is useful to observe
  some properties of the $f_{i,j}$ for $j=1$. 
For simplicity, we use the notation
\[
u_1 := (w(1)-1)t.
\]
Since $f_{\ast,0}=0$ for any $\ast$, for the case $j=1$ the inductive
description in~\eqref{eq:fij at w} simplifies, and it is easy to see
that for $i \geq 2$ we have 
\begin{equation}\label{eq:fi1 in u}
f_{i,1}(w)=u_1  \prod_{k=2}^i(u_1-w(k)t)=\sum_{k=1}^{i}(-1)^{i-k}e_{i-k}(w(2),\dots,w(i)) t^{i-k} u_1^{k}
\end{equation}
where $e_\ell$ denotes the $\ell$-th elementary symmetric polynomial
in the given variables, and we take $e_0 :=1$ by convention. Note that
$f_{1,1}(w)=u_1$ by definition, so by the above convention on $e_0$, the
equation~\eqref{eq:fi1 in u} also holds for $i=1$. 

The above computation turns out to be a special case of a general
phenomenon, recorded in Lemma~\ref{lemma:fij in terms of e and b
  equation (short)}. In order to state and prove the result, we need to first introduce
and study some properties of a new set of polynomials.

Let $\Q[u_1,\dots,u_n,t]$ be a graded polynomial ring of indeterminates $u_1,\dots,u_n,t$ with $\deg u_i=2$ for $i\in[n]$ and $\deg t=2$.
We inductively define a collection of polynomials $b_{k,j}\in\Q[u_1,\dots,u_n,t]$ for $n\geq k\geq j\geq1$ as follows. First define 
\begin{equation} \label{eq:definition bjj}
b_{j,j} :=\sum_{k=1}^j(u_k-(k-1)t) \quad \textup{ for all }  j\in [n].
\end{equation}
Then we define $b_{k,j}$ for $k \geq j$ by the equation
\begin{equation} \label{eq:definition bkj}
b_{k+1,j} :=b_{k,j-1}+u_j b_{k,j} - (u_j+t)b_{k-1,j-1} 
\end{equation}
where by convention we take $b_{*,0} :=0$ for any $\ast$. 
Note that $b_{k,j}=b_{k,j}(u_1,\dots,u_j,t)$ depends only on
  $u_1,\dots,u_j$ and $t$, and $\deg b_{k,j}=2(k-j+1)=\deg f_{k,j}$
for $k\geq j$ since $\deg t =2$.

The above observation on $f_{i,1}$ generalizes to the following property of $f_{i,j}$.
\begin{lemma} \label{lemma:fij in terms of e and b equation (short)}
Let $b_{k,j}$ be defined as above. Then for any pair $k,j \in [n]$ with $k\geq j$  the function $b_{k,j}$ is a symmetric polynomial in the variables $u_1, \ldots, u_j$. 
Moreover, we have the following equality in $\Q[t]$;
\begin{equation} \label{eq:fij in terms of e and b}
f_{i,j}(w)=\sum_{k=j}^i(-1)^{i-k}e_{i-k}(w(j+1),\dots,w(i))t^{i-k}b_{k,j}((w(1)-1)t,\dots,(w(j)-1)t,t).
\end{equation}
\end{lemma}

We will use this property and the following Corollary in the next section to prove that $\tilde\varphi_h$
of~\eqref{def of equiv phi_h} sends each $f_{h(j),j}$ to zero, but we
postpone the highly technical proof to the Appendix.

\begin{corollary} \label{lemm:f-2}
Let $m \in [n-1]$ and $w\in \Sn$. If $w'$ is the permutation
obtained by interchanging $w(m)$ and $w(m+1)$, 
then we have $f_{i,j}(w') = f_{i,j}(w)$ for $i, j \in [n]$ with $i \geq j$ and $i\neq m, j
\neq m$.
\end{corollary}

\begin{proof}
From \eqref{eq:fij in terms of e and b} it follows that
$f_{i,j}(w)$ depends only on $\{w(1), \ldots, w(i)\}$. Thus, if $i < m$ then since $f_{i,j}$ is independent of both $w(m)$ and 
$w(m+1)$, the claim follows trivially. If $m<j$ then $w(m), w(m+1) \in \{w(1), \ldots, w(j)\}$, and
since the $b_{k,j}$ are symmetric by Lemma~\ref{lemma:fij in terms of e and b equation (short)}, the claim follows. If $j<m<i$ then $w(m), w(m+1) \in \{w(j+1), \ldots, w(i)\}$, and
since the $e_{i-k}$ are also symmetric, the result follows. 
\end{proof}

\bigskip
\section{First part of proof of Theorem~\ref{theorem:reg nilp Hess cohomology}: well-definedness} \label{sec:y}
In order to prove that the homomorphism
$\tilde\varphi_h$ defined in~\eqref{def of equiv phi_h} induces a
well-defined homomorphism 
\begin{equation}\label{eq:isom} 
\varphi_h : 
\mathbb{Q}[x_1,\dots,x_n,t]/\Ih \to H^{\ast}_{S}(\mathcal
\Hess(\mathsf{N},h)) 
\quad ; \quad  x_i\mapsto \SChNil_i, \ t\mapsto t,
\end{equation}
it suffices to show that the polynomials $f_{h(j),j}$
generating the ideal $\Ih$ lie in the kernel of the map $\tilde\varphi_h : \Q[x_1,\dots,x_n,t] \rightarrow H^{\ast}_{S}(\mathcal
\Hess(\mathsf{N},h))$ given in \eqref{def of equiv phi_h}.
By the commutative diagram~\eqref{eq:cd} and in particular by the
injectivity of the bottom horizontal map $\iota_2$, it in turn
suffices to show that $f_{h(j),j}(w)=0$ for any fixed point $w \in
\Hess(\mathsf{N},h)^{S}$. This is the content of
Proposition~\ref{proposition:fij go to zero} below, whose proof
occupies the bulk of this section.

For the purposes of the argument below it is useful to introduce the
following terminology. Consider the pairs $(i,j)$ for $i, j \in [n]$
in bijective correspondence with the entries in an $n \times n$
matrix. 
For a fixed integer $\ell \geq 0$, we refer to the pairs $\{(i,j)\mid
i>j, \ i-j = \ell\}$ as the \textup{$\ell$-th lower diagonal}.  We say
that the $\ell$-th lower diagonal is \textbf{lower than} the $k$-th
lower diagonal if $\ell > k$.  For example, the dots 
  in Figure~\ref{picture for lower diagonals 2} indicate the 
  $\ell$-th lower diagonals for $\ell=1$ and $\ell=2$.

Given a Hessenberg function $h$, we have already defined a
corresponding Hessenberg subspace $H(h)$ 
(Definition~\ref{definition:Hessenberg subspace}). We can then ask which is the
\emph{lowest} lower diagonal which the Hessenberg subspace meets. 
More precisely, we say that a Hessenberg function
\textbf{meets the $\ell$-th lower diagonal} if there exists some $j
\in [n]$ such that $h(j)-j \geq \ell$. 
For example, for $h=(2,3,4,5,5)$ a Peterson Hessenberg function, the
Hessenberg subspace meets the $0$-th and $1$st lower diagonals, whereas for
$h=(3,4,4,5,5)$, the Hessenberg subspace also meets the $2$nd lower
diagonal. The \textbf{lowest lower diagonal which
$h$ meets} is evidently $\max_{j \in [n]} \{ h(j)-j \}$. Finally, we
shall say that $m$ is the \textbf{last time that $h$ meets its lowest
  lower diagonal} if $\ell = \max_{j \in [n]} \{h(j)-j\}$ and $m =
\max_{j \in [n]} \{ j \mid h(j)-j = \ell\}$. See Figure~\ref{picture for
  lower diagonals 2}. 
\begin{figure}[h]
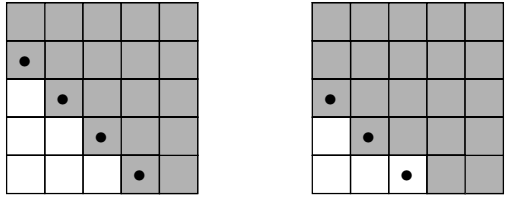

\centering
{\unitlength 0.1in%
}%
\caption{$h=(2,3,4,5,5)$ meets the 1st lower diagonal and $m=4$ (on the left), and $h=(3,4,4,5,5)$ meets the 2nd lower diagonal and $m=2$ (on the right).}
\label{picture for lower diagonals 2}
\end{figure}

The following lemma proven by Drellich 
will be useful to prove Proposition~\ref{proposition:fij go to zero}
below. Recall from \eqref{def of the set of
  Hess fct} that $H_n$ is the set of Hessenberg functions on $[n]$.
\begin{lemma}\label{remark:split}
\emph{(}\cite[Theorem 4.5]{dr}\emph{)}
  Let $h \in H_n$. Suppose $h(r)=r$ for some $r$ and let 
  $h_1=(h(1),\dots,h(r))$ and $h_2=(h(r+1)-r,\dots,h(n)-r)$.
Then $h_1\in H_r$ and $h_2\in H_{n-r}$. Moreover, for any $V_{\bullet}\in\Hess(\mathsf{N},h)$ we
have $V_r=\C^r=\C e_1\oplus\cdots\oplus\C e_r$ where $e_1,\dots,e_n$ denote the standard basis of $\C^n$. In particular
$
\Hess(\mathsf{N},h) \cong \Hess(\mathsf{N}_1,h_1)\times \Hess(\mathsf{N}_2,h_2)
$
where $\mathsf{N}_1$ and $\mathsf{N}_2$ are the regular nilpotent matrices in Jordan canonical form of size $r$ and $n-r$, respectively.
\end{lemma}

The following is straightforward and will be used later.

\begin{corollary}\label{corollary:split}
Let $h \in H_n$. Suppose $h(r)=r$ for some $r$ and let $h_1, h_2$ be
  as above. Let $S \subset T$ be the subtorus defined
  in~\eqref{eq:S and T}. The $S$-action on $\Hess(\mathsf{N},h)$ preserves
  each factor in the decomposition $\Hess(\mathsf{N},h) \cong
  \Hess(\mathsf{N}_1, h_1) \times \Hess(\mathsf{N}_2,h_2)$. In
  particular, 
$
\Hess(\mathsf{N},h)^S \cong \Hess(\mathsf{N}_1,h_1)^{S_1} \times
\Hess(\mathsf{N}_2, h_2)^{S_2}
$
where $S_1, S_2$ are the subgroups of $\emph{GL}(r,\C)$ and $\emph{GL}(n-r,\C)$
respectively defined in the same manner as in~\eqref{eq:S and T}. 
\end{corollary}

We are ready for the main assertion of this section. 
\begin{proposition}\label{proposition:fij go to zero}
Let $h\in H_n$ and $j \in [n]$. If $w \in \Hess(\mathsf{N},h)^S, \textup{ then } f_{h(j),j}(w) = 0.$
\end{proposition}

\begin{proof} 

We will first reduce the argument to the case when $n \geq 2$ and
  $h(j) \geq j+1$ for all $j \in [n-1]$. To see this, suppose that
  $n=1$. Then 
    $\Hess(\mathsf{N},h)=\Flags(\C^1)=\{\text{id}\}$
    where $\text{id}\in\mathfrak{S}_1$ is the identity permutation.
    Hence, in this case the claim is obvious by the recursive
    description \eqref{eq:fij at w} of $f_{j,j}(\text{id})$. 
    Now suppose that $n>1$ and that the claim holds for 
    all $n'<n$. Suppose also there exists $r$, $1 \leq r < n$,
      such that $h(r)=r$ and without loss of generality let $r$
      be the smallest such.  From
    Corollary~\ref{corollary:split} we have that in 
(writing permutations in one-line notation)
\begin{equation}\label{eq:split into smaller}
    \Hess(\mathsf{N},h)^S = \{(u(1) \ \dots \ u(r) \ v(1)+r \ \dots \
    v(n-r)+r)\in\Sn \mid u\in\Hess(\mathsf{N}_1,h_1)^{S_1},
    v\in\Hess(\mathsf{N}_2,h_2)^{S_2}\} 
\end{equation}
 where $S_1\subset \text{GL}(r,\C)$
    and $S_2\subset \text{GL}(n-r,\C)$ are as in
    Corollary~\ref{corollary:split}. 
    By assumption on $r$ and the definition of Hessenberg
      functions we have that if $1\leq j\leq r$
    then $1\leq h(j)\leq r$, and if $r+1\leq j\leq n$ then $r+1\leq
    h(j)\leq n$. 
Now let $w \in \Hess(\mathsf{N},h)^S$. We wish to show that
  $f_{h(j),j}(w)=0$ for all $j  \in [n]$. First consider the case when
  $1 \leq j \leq r$. 
  By Lemma~\ref{lemma:fij in terms of e and b equation (short)}, 
  we know
  $f_{h(j),j}(w)$ depends only on the values $\{w(1), \ldots,
  w(h(j))\}$. By assumption on $h$ and $j$ we know $h(j)\leq r$, so
  $f_{h(j),j}(w)$ depends only on $\{w(1), \ldots, w(r)\}$. Since $w
  \in \Hess(\mathsf{N},h)^S$, from~\eqref{eq:split into smaller} we
  know $\{w(1),\ldots,w(r)\} = \{u(1),\ldots,u(r)\}$ for some $u \in
  \Hess(\mathsf{N}_1,h_1)^{S_1} \subseteq \mathfrak{S}_r$. Now the
  inductive hypothesis applied to $n'=r<n$ implies
  $f_{h(j),j}(w)=f_{h_1(j),j}(u)=0$ as desired. Second, consider the
  case when $r+1 \leq j \leq n$. For this case, note first that since
  $h(j)\leq r$ for $j\leq r$, the above argument together with
  Lemma~\ref{lemma:fij vanish} implies that $f_{i,j}(w)=0$ for all $j
  \leq r$ and $i\geq r$. 
  From the inductive definition of the $f_{i,j}$, it
  follows that for $j$ with $r+1 \leq j \leq n$, the value
  $f_{h(j),j}(w)$ agrees with the value of $f_{h_2(j-r), j-r}(v)$
  where $v$ appears in~\eqref{eq:split into smaller}. Since $n-r<n$,
  the inductive hypothesis again implies
  $f_{h(j),j}(w)=f_{h_2(j-r),j-r}(v)=0$ as desired. 
    
An induction on $n$ and the argument above shows that it now
  suffices to prove the claim for the case when $n \geq 2$ and  
    \[ h(j) \geq j+1 \quad \text{for all } j \in [n-1]. \]
Fix such an $n$. The set of Hessenberg functions associated to
  $n$ which we must analyze is exactly
\[
H'_n := \{ h \in H_n \mid h(j) \geq j+1 \textup{ for all } j \in [n-1] \}.
\]
The remainder of our argument is by induction using the total
  order on $H'_n$, denoted by $\prec$, defined by 
\begin{align*}
h' \prec h  \Leftrightarrow \exists \hsm m\in [n-1] \textup{ such that for all } j \in [n] \textup{ with } j > m\textup{ we have } h'(j)=h(j), \textup{ and } h'(m) < h(m).
\end{align*}
The order $\prec$ is the usual reverse lexicographic order on
  $\Z^n_{\geq 0}$ if we view a Hessenberg function $h$ as a sequence
  $(h(1),h(2),\ldots,h(n))$ of positive integers. We also note that
  the above
total order is a refinement of the partial
  order $h' \subset h$ of Definition~\ref{def:partial order}.  Moreover, the
  unique minimal element in $H'_n$ with respect to $\preceq$ is the
  Hessenberg function satisfying $h(j)=j+1$ for all $j\in [n-1]$.
The base case of our induction therefore exactly corresponds to
  the Peterson variety, and as discussed in Remark~\ref{rema:peterson} the results
  of \cite{fu-ha-ma} imply that the claim of the proposition holds in this case.
  Thus we may now assume that the claim
  is true for all $h' \prec h$ and we must now prove the
  claim for $h$.

Suppose 
$w \in \Hess(\mathsf{N},h)^S$. Then by Proposition~\ref{prop:equivalence of fixed point
  conditions} we know $h_w \subset h$, from which it follows that
$h_w \preceq h$. If $h_w \neq h$, by the
inductive hypothesis we may conclude that $f_{h_w(j), j}(w)=0$ for all $j$. From
Lemma~\ref{lemma:fij vanish} and the definition of the partial order
$h_w \subset h$ it then follows that $f_{h(j),j}(w)=0$ for all $j$, as
desired. Thus it remains to check the claim for those $w$ with the property that 
$h=h_w$. 

Since we may assume that $h=h_w$ is strictly larger than the
Hessenberg function associated to the base case of the Peterson
variety, there exists a $j$ such that $h(j) \geq j+2$. Note that such
a $j$ must satisfy $j\leq n-2$ due to the definition
of Hessenberg functions. Now let $m$ be the last time that $h=h_w$
meets its lowest lower diagonal, in the sense discussed above. By the
above, 
such an $m$ must satisfy $h(m)\geq m+2$ and
  $m\leq n-2$. 
We also have that $h(m-1)<h(m)$, since
otherwise $h$ meets a lower diagonal which is lower than that
containing $(h(m),m)$, contradicting the definition of $m$. From
Lemma~\ref{lemma:corner} it then follows that
$h(m)=h_w(m)=w^{-1}(w(m)-1)$, or in other words
\begin{align}\label{eq:notice}
w(h(m)) =w(m)-1.
\end{align}
Define a permutation
$w'\in\Sn$ obtained from $w$ by interchanging the values in the $m$-th and
$(m+1)$-st positions of the one-line notation of $w$, i.e.
$w'(m)=w(m+1)$, $w'(m+1)=w(m)$, and $w'(j)=w(j)$ for
all $j \neq m, m+1$. 
Let $h':=h_{w'}$ denote the
corresponding Hessenberg function.

We claim that 
\begin{align}\label{eq:claim for h and h'}
h'(j)=h(j) \ \ \text{for all $j\ge m+1$} \ \ \ \text{and} \ \ \ h'(m) < h(m).
\end{align}
To see this, first consider the case $j \geq
m+1$. Recall that the definition of $h(j)=h_w(j)$ and
$h'(j)=h_{w'}(j)$ is in terms of the sets $D_w(j), D_{w'}(j)$ which
are in turn constructed from the sets $\NInv{w}$ and $\NInv{w'}$ by
looking at $\mathsf{N}$-inverted pairs $\mathcal{P}$ with $LP(\mathcal{P}) \leq
j$. Since $w$ and $w'$ only differ in the $m$-th and $(m+1)$-st
entries, if $j \geq m+1$ then from the definition it follows that
$D_w(j)$ is obtained from $D_{w'}(j)$ by replacing any
  $m$ which appears in a left position with an $m+1$,
 and hence $h(j)=h_w(j)=h_{w'}(j)=h'(j)$ as
desired.  
Next, we wish to prove that $h'(m)<h(m)$. To see this, first
consider the case that $D_{w'}(m)=\emptyset$. In this case, from the
definition of $h'=h_{w'}$ we have $h'(m)=m$, but since $h(m)
\geq m+2$ as we observed above, we conclude  
$h'(m)<h(m)$ as
desired. Second, we consider the case $D_{w'}(m) \neq
\emptyset$. 
From \eqref{eq:notice}
we know $(m,h(m)) \in D_w(m)$
and since $h(m)$ achieves the maximum
of the set $\{RP(\mathcal{P}) \hsm \mid \hsm \mathcal{P} \in
D_w(m)\}$ by definition~\eqref{eq:def hw}, 
there are no $\mathsf{N}$-inverted pairs
$\mathcal{P} \in \NInv{w}(m)$ with $RP(\mathcal{P}) > h(m)$. Now recall that we wish to show $h'(m)<h(m)$
and we assume that $D_{w'}(m) \neq \emptyset$. Since 
$w$ and $w'$ swapped their $m$-th and $(m+1)$-st
places, the pair $(m,h(m))$ is no longer an $\mathsf{N}$-inverted pair in
$D_{w'}(m)$. 
Let $p:=w^{-1}(w(m)+1)$ and $q:=w^{-1}(w(m+1)-1)\geq0$ (see our convention \eqref{eq:convention on w(0)}). Then it follows that 
\begin{align}\label{eq:inclusion of Dw'(m)}
D_{w'}(m) \subset  (D_w(m) \setminus \{(m,h(m))\} ) \cup \{(p,m+1)\} \cup \{(m,q)\}.
\end{align}
We have 
\[
q=w^{-1}(w(m+1)-1) \leq h_w(m+1) = h_w(m) = w^{-1}(w(m)-1)
\]
where the first inequality follows from \eqref{eq:hw def reformulation
  inequality} and the middle
 equality is because $m$ is the last time $h$ meets its lowest
   lower-diagonal. Since $w(m) \neq w(m+1)$ it cannot happen that
  $q=w^{-1}(w(m+1)-1)=w^{-1}(w(m)-1)$, so we conclude
$q<h_w(m)=h(m)$.
Since $(m,h(m))\in D_w(m)$ and $D_{w'}(m) \neq \emptyset$,
by~\eqref{eq:def hw} we have $h_w(m)=\max RP(D_w(m))$ and $h_{w'}(m)=\max RP(D_{w'}(m))$.
It then
follows from \eqref{eq:inclusion of Dw'(m)} and $h(m)\geq m+2$ that $h'(m)=h_{w'}(m) < h_w(m)=h(m)$, as desired. 

We have just seen that 
$h$ and $h'$ agree in all coordinates to the right
of $m$, and that $h'(m) < h(m)$. Thus 
$h' \prec h$, 
and by the inductive hypothesis,
the claim of the proposition holds for $h'$, so $f_{h'(j),j}(w')=0$
for all $j$. Moreover, since $w'$ is obtained from $w$ by
interchanging the entries $w(m)$ and $w(m+1)$, from
Corollary~\ref{lemm:f-2} we may conclude
that $f_{h'(j),j}(w) = 0$ for all $j \geq m+1$. 
Since we have also seen above that $h(j) = h'(j)$ for $j \geq m+1$,
this implies $f_{h(j),j}(w)=0$ for all $j \geq m+1$. 

Next we compute for $j=m$. First, we have from the recursive relations \eqref{eq:fij at w} that
\begin{equation} \label{eq:fhmm} 
\begin{split}
f_{h(m),m}(w)&=f_{h(m)-1,m-1}(w)+(w(m)-w(h(m))-1)t\cdot
f_{h(m)-1,m}(w) \\ 
&=f_{h(m)-1,m-1}(w) \\ 
&=f_{h(m)-1,m-1}(w') 
\end{split}
\end{equation}
where the second equality holds because we have shown that $w(m)-1=w(h(m))$ in \eqref{eq:notice} and the third equality follows by Corollary~\ref{lemm:f-2} since $h(m)\ge m+2$ so $h(m)-1 \neq m$.
Recall that $h_{w'}=h' \prec h$, so by the inductive hypothesis we have that
$f_{h'(j),j}(w') = 0$ for all $j$. But then by Lemma~\ref{lemma:fij
  vanish} we know that $f_{i,j}(w')=0$ for any $j \in [n]$ and $i \geq
h'(j)$. In particular, since $h'(m-1)\leq h'(m)$ by definition of
Hessenberg functions and since $h'(m) \leq h(m)-1$ as shown in \eqref{eq:claim for h and h'}, we
have $h'(m-1) \leq h(m)-1$ and hence $f_{h(m)-1,
  m-1}(w')=0$. From~\eqref{eq:fhmm} this implies $f_{h(m),m}(w)=0$, as
desired. 

It remains to check that $f_{h(j),j}(w)=0$ for $j \leq m-1$. 
We will again argue by comparing the computations for $h'$ with those
for $h$. Note that in general it may not be the case that $h' \subset
h$. Recalling that $h'=h_{w'}$ is defined in terms of the permutation
$w'$ which differs from $w$ only in the $m$-th and $(m+1)$-st spots,
it is useful to define
\begin{equation}\label{eq:def r s} 
r:=\min\{j\mid m\le h(j)\}, \quad s:=\min\{j\mid m+1\le h(j)\}.
\end{equation}
We also define $r_0$ (respectively $s_0$) as the position of $w(m)+1$
(respectively $w(m+1)+1$):  
\begin{equation*}
w(r_0)-1=w(m)=w'(m+1),\quad w(s_0)-1=w(m+1)=w'(m). 
\end{equation*}
It is clear from the definitions that $r_0\neq m$ and $s_0\neq m+1$.
If $r_0<m$, then $(r_0, m)$ is an $\mathsf{N}$-inverted pair in $D_w(r_0)$, so
$m \leq h_w(r_0) =h(r_0)$ by the definition \eqref{eq:def hw} of $h_w$, and if $r_0 >m$ then since $h(r_0)=h_w(r_0)
\geq r_0$, we also have $m \leq h(r_0)$. Thus, from the
definition~\eqref{eq:def r s} of $r$ we see that $r \leq
r_0$. Similarly $s \leq s_0$. 
Moreover, since $h(m) \geq m+2$ and $m \leq n-2$ by the definition of $m$, we also have $r\le s\le m \leq n-2$. In summary, we have
\begin{equation}\label{eq:r and s less than m} 
r\leq r_0, \quad s\leq s_0, \quad 
r\le s\le m \leq n-2.
\end{equation}
Note also that from the definition of $r$ it follows that
$h(r-1)<h(r)$. Furthermore, from~\eqref{eq:r and s less than m} we know $r<n$
and hence from the original assumption on the Hessenberg
function $h$ we know $h(r) \geq r+1$. Thus we may apply Lemma~\ref{lemma:corner} 
to conclude that
\begin{align}\label{eq:w(h(r)) = w(r)-1}
h(r) = w^{-1}(w(r)-1) \textup{ and hence } w(h(r)) = w(r)-1.
\end{align}
That is, we have $(r,h(r))\in \NInv{w}$.

The next observation will be useful in what follows.
By assumption on the Hessenberg function $h$ we have $h(j)\geq j+1$ for all $j\leq n-1$ and hence from the definition \eqref{eq:def hw} of $h=h_w$ we see that $D_w(j) \neq \emptyset$ for all $j\leq n-1$.
Since $w'$ and $w$ only differ in the $m$-th and $(m+1)$-st spots we also have $D_{w'}(j) \neq \emptyset$ for $j\leq m-1$.
Hence, the description \eqref{eq:hw def reformulation} for $h_w$ and $h_{w'}$ shows that for $j\leq m-1$ 
we can express $h(j)$ and $h'(j)$ by 
\begin{equation} \label{eq:hw reformulation take 2}
\begin{split}
&h(j)=h_w(j)=\max\{w^{-1}(w(p)-1)\mid 1\le p\le j\}, \\
&h'(j)=h_{w'}(j)=\max\{w'^{-1}(w'(p)-1)\mid 1\le p\le j\}.
\end{split}
\end{equation}

Recall that we wish to show $f_{h(j),j}(w)=0$ for $j \leq m-1$. We
will argue on a case-by-case basis according to the value of $h(r)$,
where $r$ is the value defined in~\eqref{eq:def r s}. 

\medskip
{\bf Case 1.} Suppose $h(r)\ge m+2$. Then 
from the
definitions of $r$ and $s$ in~\eqref{eq:def r s} it immediately
follows that $r=s$. We already know that $r \leq r_0$ from \eqref{eq:r and s less than m}, but in this case from \eqref{eq:w(h(r)) = w(r)-1}
we in fact have 
\begin{equation*} 
w^{-1}(w(r)-1)=h(r)\ge m+2> m=w^{-1}(w(r_0)-1)
\end{equation*}
so $r\not=r_0$, from which it follows $r<r_0$.  
It similarly follows that $s=r<s_0$. From this we claim that 
$h(j)=h'(j)$ for $j\leq m-1$. Indeed, recall from~\eqref{eq:hw reformulation take
  2} that $h(j)$ (respectively $h'(j)$)
can be described as the maximum of the right positions
$RP(\mathcal{P})$ of $\mathsf{N}$-inverted pairs $\mathcal{P}$ whose left
positions go from $1$ up to $j$. 
Our assumption that $h(r)\geq m+2$ implies that the one-line
notation of $w$ is of the form 
\[
(\ldots, w(r), \ldots, w(m), w(m+1), \ldots, w(h(r))=w(r)-1, \ldots)
\]
where the position $h(r)$ of $w(r)-1$ is, by assumption, to the right 
of both $w(m)$ and $w(m+1)$. We have just argued that $r<r_0$ and
$r<s_0$, which is to say that if $w(m)$ and $w(m)+1$ (respectively
$w(m+1)$ and $w(m+1)+1$) appear in inverted order in $w$, then the
larger value $w(m)+1=w(r_0)$ (respectively $w(m+1)+1=w(s_0)$) must appear to the
\emph{right} of $w(r)$. (If they do not appear in inverted order, then
they cannot be an inverted pair and hence never contribute to the
computation of $h=h_w$.) But since $h(j)$ is computed by looking
for the maximum of the $RP(\mathcal{P})$ for such $\mathcal{P}$ whose left position is
\emph{up to $j$}, and since the $\mathsf{N}$-inverted pair $(r, h(r))$ occurs
before $r_0$ and $s_0$ (i.e. $r<r_0$ and $r<s_0$) but has a larger
$RP(\mathcal{P})$ (i.e. $w(h(r))=w(r)-1$ occurs to the right of $w(m)$ and $w(m+1)$), this
implies that the inverted pairs (if any) with right positions $m$ and
$m+1$ never achieve the maximum in the computation of $h(j)$. Since
$w'$ differs from $w$ only by interchanging the $w(m)$ and $w(m+1)$,
this assertion remains true for $w'$. Hence the computation for $h(j)$
and $h'(j)$ remains unchanged, and we conclude 
\[
h(j)=h'(j) \ \ \text{for $j\leq m-1$}.
\]

Now from the inductive hypothesis we know $f_{h'(j),j}(w')=0$ for all $j\in[n]$. Since we just saw 
$h(j)=h'(j)$ for $j \leq m-1$ in this case, we then obtain that $f_{h(j),j}(w')=0$
for $j\leq m-1$.  Finally, observe that $h(j)<m$ for any $j<r$ by
definition of $r$, and for $j \geq r$ the assumption that $h(r) \geq m+2$
implies that $h(j) \neq m$. Hence we may apply Corollary~\ref{lemm:f-2}
and conclude that $f_{h(j),j}(w)=0$ for $j \leq m-1$, as desired. 

\medskip
{\bf Case 2.} Next we consider the case $h(r)=m+1$. We immediately see that $r=s$ in this case as well. 
Recall from \eqref{eq:w(h(r)) = w(r)-1} that
$h(r)=w^{-1}(w(r)-1)$. Since $h(r)=m+1$
by assumption we have $w(m+1)=w(r)-1$. Recall that the definition of $m$
guarantees that $h(m)\geq
m+2$, so $r<m$. Based on this discussion we conclude that the one-line
notation for $w$ looks like 
\[
(\ldots, w(r)=w(m+1)+1, \ldots, w(m), w(m+1), \ldots) 
\]
so we can see that $r=s=s_0$ in this case. Also, arguing as in the
case above, we know that $r<r_0$. 
We claim that 
\begin{align}\label{eq:case 2 claim}
h'(j)\leq h(j) \ \ \text{for all $j\leq m-1$}.
\end{align}
We take cases. Recall that $r_0\neq m$ since $r_0$ is defined by $w(r_0)=w(m)+1$. First suppose
$r_0\geq m+1$, i.e. the value $w(m)+1$ occurs to the right of $w(m)$
in the one-line notation for $w$. In this case, the integers $w(m)$
and $w(m)+1$ are not inverted in $w$ and hence 
never contributes
to the computation of any $h(j)$. Hence we may conclude \eqref{eq:case 2 claim} in this case. Next consider the case $r_0 \leq
m-1$, so the one-line notation for $w$ looks like 
\[
(\ldots, w(r)=w(m+1)+1, \ldots, w(r_0) = w(m)+1, \ldots, w(m), w(m+1), \ldots)
\]
and the one-line notation for $w'$ then looks like 
\[
(\ldots, w'(r)=w(m+1)+1, \ldots, w'(r_0)=w(m)+1, \ldots, w(m+1), w(m), \ldots).
\]
In what follows we prove \eqref{eq:case 2 claim} by
  looking at the one-line notations.
Recall that the only difference between $w$ and $w'$ is that $w(m)$ and
$w(m+1)$ have been interchanged, and that the computation
of $h(j)$ involves looking at $\mathsf{N}$-inverted pairs in $\NInv{w}$ with left position up to
$j$ (and similarly for $h'(j)$). 
For the cases $1\leq j<r$ or $r_0\leq j\leq m-1$ (i.e. the cases in which $j$ is outside of the range between $r$ and $r_0-1$), we see from this observation together with the above one-line notations that $h'(j)=h(j)$.
For the case $r\leq j<r_0$, we have $h'(j) \leq h(j)$ by the same reasoning.
Thus we conclude \eqref{eq:case 2 claim}, as desired. 

Now from the inductive hypothesis we know that $f_{h'(j),
  j}(w')=0$ for all $j$, so from Lemma~\ref{lemma:fij vanish} and \eqref{eq:case 2 claim} we may
conclude that $f_{h(j),j}(w')=0$ for all $j \leq m-1$. From the assumption that
$h(r)=m+1$ it follows as in the argument for Case 1 that there does
not exist any $j$ with $h(j)=m$, and since $j \leq m-1$ we have $j\neq m$. Hence we may apply Corollary~\ref{lemm:f-2} to conclude
that $f_{h(j),j}(w)=0$ for $j\leq m-1$, as desired. This completes the argument for
Case 2. 

\medskip
{\bf Case 3.}  It remains to consider the case when $h(r)=m$. This
means that $m$ is actually achieved as a value of $h$, so by the
definition of $s$ it follows that $r \neq s$, $r<s$, and $h(s) \geq
m+1$. From \eqref{eq:w(h(r)) = w(r)-1} we also know
$w(r)-1=w(h(r))$, and since $h(r)=m$ we have $w(r)-1=w(m)$. Hence
$r=r_0$. 
The one-line notation for $w$ therefore looks like 
\[
(\ldots, w(r)=w(r_0)=w(m)+1, \ldots, w(s), \ldots, w(m),w(m+1), \ldots)
\]
and the one-line notation for $w'$ looks like 
\[
(\ldots, w(r)=w(r_0)=w(m)+1, \ldots, w(s), \ldots, w(m+1), w(m), \ldots)
\]
where the only difference is the interchanging of $w(m)$ and
$w(m+1)$. 

We now concretely analyze the difference between $h$ and $h'$.
Recalling $s\leq s_0$ from \eqref{eq:r and s less than
    m} where we recall that $s_0$ is the position of $w(m+1)+1$ in the
  one-line notation for $w$, it follows from arguments similar to
  those in the previous cases that
\begin{equation} \label{eq:case s_0 ge m}
\begin{split}
h'(j)&=h(j)<m \quad \textup{ for } 1\le j\le r-1,\\
h'(j)&=m+1=h(j)+1 \quad \textup{ for } r\le j\le s-1, \textup{ and } \\
h'(j)&=h(j)>m \quad \textup{ for } s\le j\le m-1 \textup{ (note that if $s=m$ then
  there are no such $j$).}
\end{split}
\end{equation}

Now, consider $j$ with $1 \leq j \leq r-1$ or $s \leq j
\leq m-1$. By~\eqref{eq:case s_0 ge m} we have in these cases that
$h'(j)=h(j)$. Hence by the inductive hypothesis
$f_{h(j),j}(w')=f_{h'(j),j}(w')=0$. Moreover, since 
$j \neq m$ and $h(j)\neq m$ by \eqref{eq:case s_0 ge m}, we may again apply Corollary~\ref{lemm:f-2} to conclude that
$f_{h(j),j}(w)=0$, as desired. 

It remains to consider the case of $j$ with $r \leq j \leq s-1$. 
For such $j$ we have $h(j)=m$, so we wish to prove that $f_{m,j}(w)=0$
for $r \leq j \leq s-1$. We argue by induction, with the base case
being $j=r$.  For what follows we introduce the temporary notation 
\begin{equation}\label{eq:def Aj}
A_j := f_{m-1,j}(w) = f_{m-1,j}(w') \quad (j \leq m-1).
\end{equation}
where the second equality holds because $j \neq m$ by assumption
and by Corollary~\ref{lemm:f-2}.
Also, from~\eqref{eq:case s_0 ge m} we have
that 
$h'(r-1)=h(r-1)\le m-1$, so 
from our inductive hypothesis on $h'$ and Lemma~\ref{lemma:fij vanish} we may
conclude $A_{r-1}=f_{m-1,r-1}(w')=0. $
Since $m>r$, using the recursive equation \eqref{eq:fij at w} of the $f_{i,j}(w)$ we may now compute 
\begin{equation*}
\begin{split}
f_{m,r}(w)&=f_{m-1,r-1}(w)+(w(r)-w(m)-1)f_{m-1,r}(w)t=A_{r-1}+(w(r)-w(m)-1)A_r t= 0
\end{split}
\end{equation*}
where we have used the fact that $r=r_0$ and
thus $w(r)=w(m)+1$. This proves the claim for the base case $j=r$. 

Now suppose by induction that $f_{m,k}(w)=0$ for some $k$ with $r \leq k \leq s-2$, and
we wish to prove the statement for $k+1$. 
We know from~\eqref{eq:case s_0 ge m} that $h'(k+1)=m+1$, so
from our inductive hypothesis on $h'$ we have
$f_{m+1,k+1}(w')=0$. Since $m>k+1$, using
(repeatedly) the recursive equation \eqref{eq:fij at w} of the $f_{i,j}(w)$ and~\eqref{eq:def Aj} we have 
\begin{equation}\label{eq:reexpress in A}  
\begin{split}
0&=f_{m+1,k+1}(w')\\
&=f_{m,k}(w')+\big(w'(k+1)-w'(m+1)-1\big)f_{m,k+1}(w') t \\
&=A_{k-1}+(w'(k)-w'(m)-1)A_{k}t\\
&\qquad +\big(w'(k+1)-w'(m+1)-1\big)\Big(A_{k}+\big(w'(k+1)-w'(m)-1\big)A_{k+1}t\Big)t\\
&=A_{k-1}+\big(w(k)-w(m+1)+w(k+1)-w(m)-2\big)A_{k}t\\
&\qquad +\big(w(k+1)-w(m)-1\big)\big(w(k+1)-w(m+1)-1\big)A_{k+1}t^2
\end{split}
\end{equation}
where the last equality also uses the definition of $w'$ in terms
  of $w$. 
By our assumption on $k$ we have 
\begin{equation*}
f_{m,k}(w)=A_{k-1}+\big(w(k)-w(m)-1\big)A_k t=0, 
\end{equation*} 
and hence we can further
simplify the last expression in~\eqref{eq:reexpress in A} as
\begin{equation} \label{eq:product}
0=\big(w(k+1)-w(m+1)-1\big)\Big(A_k+\big(w(k+1)-w(m)-1\big)A_{k+1}t\Big)t.
\end{equation}
Now remember that by assumption $k+1\le s-1$ and also $s \leq s_0$ from \eqref{eq:r and s less than m}, which
means $w(k+1)\not=w(s_0)=w(m+1)+1$. Thus from~\eqref{eq:product} we
finally obtain 
\[
f_{m,k+1}(w)=A_k+\big(w(k+1)-w(m)-1\big)A_{k+1}t=0,
\]
as desired. This proves the result that $f_{h(j),j}(w)=0$ for all $r \leq j \leq s-1$, so we
have checked all cases and the result is proved. 
\end{proof} 

\smallskip

We now prove that the ring homomorphism \eqref{eq:isom} is
well-defined. 
Recall that $\Ih$ is the ideal of $\Q[x_1,\dots,x_n,t]$ generated by $f_{h(j),j}$ for $j=1,\dots,n$.
\begin{corollary}\label{cor for well-def}
The graded $\Q[t]$-algebra homomorphism 
\begin{align*}
\varphi_h : \Q[x_1,\dots,x_n,t]/\Ih \rightarrow H^*_S(\Hess(\N,h))
\end{align*}
which sends each $x_i$ to the first Chern class
$\SChNil_i$ and $t$ to $t$ is well-defined, where we identify $H^*(BS)
\cong \Q[t]$.
\end{corollary}
\begin{proof}
For $w\in \Hess(\N,h)^S$, 
the $w$-th component of the image of $f_{h(j),j}(\SChNil_1,\dots,\SChNil_n,t)$ under the localization map $\iota_2: H^{\ast}_{S}(\Hess(\mathsf{N},h))
\rightarrow \bigoplus_{w\in \Hess(\mathsf{N},h)^{S}} \Q[t]$
is precisely the $f_{h(j),j}(w)$ considered in
Proposition~\ref{proposition:fij go to zero}. Thus,
Proposition~\ref{proposition:fij go to zero} together with the
injectivity of $\iota_2$ implies that 
\[
f_{h(j),j}(\SChNil_1,\dots,\SChNil_n,t)=0\in H^{\ast}_{S}(\Hess(\mathsf{N},h))
\]
for all $j\in[n]$.
Hence, the ring homomorphism $\tilde\varphi_h$ 
defined in \eqref{def of equiv phi_h} factors through the quotient by
$\Ih$, inducing the map $\varphi_h$ as desired. 
\end{proof}

\bigskip
\section{Hilbert series}\label{sec:hilbert} 

The main result of this section,
  Proposition~\ref{prop:hilbert series equal},
  takes a further step in the proof of Theorem~\ref{theorem:reg nilp
    Hess cohomology} by proving that 
the two rings are \emph{additively} isomorphic as graded $\Q$-vector
spaces, i.e. that their Hilbert series (to be defined below) are
equal. This will be useful in our arguments in Section~\ref{sec:proof
  of first main theorem}
because, if a map between two graded vector spaces is injective 
 and
we know the dimensions of the graded pieces are equal, then the map
must be an isomorphism. 

We outline the content of this section. 
We first record some preliminary
definitions and recall some properties of regular sequences.  To prove
Proposition~\ref{prop:hilbert series equal} it will turn out to be
useful to first compute the Hilbert series for the ordinary cohomology. 
As a first step, by using results of
  Mbirika-Tymoczko \cite{mb-ty13} and a small trick involving the Hessenberg space's
  negative roots, we rewrite the Hilbert series of the ordinary
  cohomology $H^*(\Hess(\N,h))$ in
  terms of $h$. 
Next, we show in Lemma~\ref{lem:formula for f check} that the
polynomials $\cf_{i,j}$ defined in~\eqref{eq:definition of f check} can be obtained from the
$f_{i,j}$ by setting the variable $t$ equal to $0$, and then  
prove that the homogeneous polynomials $\cf_{h(1),1}, \ldots,
\cf_{h(n),n}$ described in \eqref{formula for f check} form a regular sequence in $\Q[x_1,\ldots,x_n]$. Since
the degrees of the $\cf_{h(i),i}$ are known, this allows us to
conclude that the Hilbert series of $\Hess(\N,h)$ and
$\Q[x_1,\ldots,x_n]/\check \Ih$ are equal. Here recall that $\check
\Ih$, defined in \eqref{definition ideal check Ih}, is the ideal 
\[
\check \Ih = (\cf_{h(1),1},\dots,\cf_{h(n),n})
\]
generated by the $f_{h(j),j}$ for $j \in [n]$. Now some straightforward
arguments, involving on the one hand some elementary considerations
using module bases and the $S$-equivariant formality of $\Hess(\N,h)$
on the other, yield the fact that the Hilbert series of the
$S$-equivariant cohomology $H^*_S(\Hess(\N,h))$ and
$\Q[x_1,\ldots,x_n,t]/\Ih$ are equal. 

As before, we equip the polynomial rings $\Q[x_1, \ldots, x_n]$ and
$\Q[x_1, \ldots, x_n,t]$ with the gradings defined by 
\begin{align*}
\deg x_i=2 \textup{ for all } i \in [n] \quad \textup{ and } \quad \deg t=2. 
\end{align*} 

We begin by recalling the definition of Hilbert series. Let $R=\bigoplus_{i=0}^\infty
R_i$ be 
a graded $\Q$-vector space where each $R_i$ is finite-dimensional.
Then we define its \textbf{Hilbert series} to be 
\[
F(R,\sp):=\sum_{i=0}^\infty (\dim_{\Q}R_{i}) \sp^i \in \Z[[\sp]]
\]
where $\sp$ is a formal parameter. 
We also take a moment to 
recall the definition and some properties of regular
sequences, which we use extensively for the remainder of this
section. 

\begin{definition}\label{def of regular seq for Hilb series}
\emph{(}\cite[Section 16]{matsumura}\emph{)}
For a ring $S$, a sequence $\f_1,\dots,\f_r\in S$ is called a
\textbf{regular sequence} if: 
\begin{itemize}
\item[(i)] $\f_i$ is non-zero, and not a zero-divisor, in $S/(\f_1,\dots,\f_{i-1})$ for $i=1,\dots,r$,
\item[(ii)] $S/(\f_1,\dots,\f_r)\neq0$.
\end{itemize}
\end{definition}

\begin{remark}\label{remark:characterizations of regular sequences}
There are other useful characterizations of regular sequences which we
shall employ in our arguments below. 
\begin{enumerate} 
\item \label{rem:req seq alg-ind} 
If $S$ is a graded $\Q$-algebra and $\f_1, \dots, \f_r$ are positive-degree homogeneous elements, 
then it is well-known that $\{\f_1, \dots, \f_r\}$ is a regular sequence if and only if $\{\f_1,\dots,\f_r\}$ is algebraically
independent over $\Q$ and $S$ is a free
$\Q[\f_1,\dots,\f_r]$-module \emph{(e.g. \cite[Chapter 1, Section 5.6]{stan96})}. 
\item \label{rem:req seq Hilb}
Let the polynomial ring $S=\Q[x_1,\dots,x_n]$ be graded with $\deg x_i=2$ for $1\leq i\leq n$.
It is also well-known 
$($see for instance \cite[p.35]{stan96}$)$ 
that 
a sequence $\f_1,\dots,\f_r \in \Q[x_1,\dots,x_n]$ 
of positive-degree homogeneous polynomials
is a regular sequence if and only if 
\begin{align*}
F(\Q[x_1,\dots,x_n]/(\f_1,\dots,\f_r),s)
=\frac{1}{(1-s^2)^n}\prod_{k=1}^r(1-s^{\deg{\f_k}}).
\end{align*}
\item \label{rem:req seq zero-set}
Finally, continuing in the special setting of the previous item, 
if $r=n$ in addition, 
it is known 
that  a sequence of positive-degree
  homogeneous elements $\f_1,\dots,\f_n$ in 
  $\Q[x_1,\dots,x_n]$ is a regular sequence if and only if the
  solution set of the equations $\f_1=0,\dots,\f_n=0$ in $\C^n$
  consists only of the origin $\{0\}$. 
  $($This is because the above characterization $\eqref{rem:req seq Hilb}$ is valid for any coefficient field and hence \cite[Proposition 5.1]{fu-ha-ma} gives the claim.$)$
\end{enumerate} 
\end{remark}

The following simple fact, which follows from \cite[Chapter 1, Theorem 5.9]{stan96}, is also useful.

\begin{lemma}\label{lemma:finite generation}
Let $g_1, \ldots, g_n$ be positive-degree homogeneous polynomials in
$\Q[x_1, \ldots, x_n]$. Suppose $\{g_1, \ldots, g_n\}$ is a regular
sequence in $\Q[x_1, 
\ldots, x_n]$. 
Then $\Q[x_1,\ldots, x_n]$ is finitely generated as a
$\Q[g_1, \ldots, g_n]$-module. 
\end{lemma} 

With these preliminaries in place, we begin our computation of the Hilbert series $F(H^*_S(\Hess(N,h)), s)$ of the equivariant
cohomology ring $H^*_S(\Hess(\mathsf{N},h))$. Our first step towards this goal is 
to compute the Hilbert series of 
the ordinary cohomology ring $H^*(\Hess(\mathsf{N},h))$ using results
of Mbirika \cite{mb} (we also took inspiration from related work of Peterson and Brion-Carrell
as in \cite{br-ca04}). 

\begin{lemma}\label{lemma:hesshil} 
Let $n$ be a positive integer and $h: [n] \to
[n]$ a Hessenberg function. Let $\Hess(\mathsf{N},h)$ denote the
corresponding regular nilpotent Hessenberg variety. Then the Hilbert series of the ordinary
cohomology ring $H^*(\Hess(\mathsf{N},h))$ $($equipped with the usual grading$)$ is 
\begin{equation*} 
F(H^*(\Hess(\mathsf{N},h)),s)=\prod_{j=1}^{n}\frac{1-s^{2(h(j)-j+1)}}{1-s^2}.
\end{equation*}
\end{lemma}

The following proof of Lemma~\ref{lemma:hesshil} uses a
trick which re-writes certain expressions as a product over negative
roots $\NR(h)$ contained in the Hessenberg space as in~\eqref{eq:def
  NR(h)}. 

\begin{proof}[Proof of Lemma~\ref{lemma:hesshil}]
Following \cite{mb}, we define integers $\beta_i$ for $i\in[n]$ by 
\[
\beta_i :=i-|\{k\in[n]\mid h(k)<i\}|. 
\]
It is straightforward to see that $\beta_i-1$ is the number of
elements in $\NR(h)$ which are contained in the $i$-th row, i.e.  pairs
in $\NR(h)$ whose first coordinates equal to $i$.
In particular, $\beta_i > 0$ for all $i \in [n]$. 
For a positive integer $\beta$, denote by $h_\beta(x_1, \ldots, x_k)$ the degree-$\beta$ complete symmetric polynomial in the listed
variables. Following \cite{mb} we define $J_h$ to be the ideal in $\Q[x_1, \ldots, x_n]$ generated by 
the polynomials $h_{\beta_n}(x_n), h_{\beta_{n-1}}(x_{n-1},x_n),
\dots, h_{\beta_1}(x_1,\dots,x_n)$. 
It turns out \cite[Theorem 3.4.3]{mb} that the Hilbert series of $H^*(\Hess(\mathsf{N},h))$ and
$\Q[x_1,\dots,x_n]/J_h$ coincide: 
\begin{equation*}
F(H^*(\Hess(\mathsf{N},h)),s)=F(\Q[x_1,\dots,x_n]/J_h,s) .
\end{equation*}
We next claim that this sequence 
$h_{\beta_n}(x_n),h_{\beta_{n-1}}(x_{n-1},x_n)$,$\dots,$ $h_{\beta_1}(x_1,\dots,x_n)$ forms a regular sequence. Since
there are precisely $n$ elements in the sequence, which is equal to
the number of variables in the ambient polynomial ring, we may use the
characterization in Remark~\ref{remark:characterizations of regular
  sequences}\eqref{rem:req seq zero-set} above; in particular, it suffices to see that their
common zero locus in $\C^n$ is just the origin $0 \in \C^n$. Noting as above
that each $\beta_i$ is positive, we see first that
$h_{\beta_n}(x_n) = x_n^{\beta_n} =0$ implies $x_n = 0$. But if
$x_n=0$ then 
$h_{\beta_{n-1}}(x_n, x_{n-1}) = h_{\beta_{n-1}}(0,x_{n-1}) =
h_{\beta_{n-1}}(x_{n-1}) = x_{n-1}^{\beta_{n-1}} = 0$ and we may conclude
$x_{n-1}=0$. Continuing in this manner we see that all $x_i=0$, i.e. 
the common zero locus is $\{0\}$ as desired. Using the
characterization of regular sequences in
Remark~\ref{remark:characterizations of regular sequences}\eqref{rem:req seq Hilb} we then have 
\begin{equation}\label{eq:beta_i}
F(\Q[x_1,\dots,x_n]/J_h,s) = F(\Q[x_1,\dots,x_n],s)\prod_{i=1}^{n}(1-s^{2\beta_i}) 
 = \prod_{i=1}^{n}\prod_{k=1}^{\beta_i-1}\frac{1-s^{2(k+1)}}{1-s^{2k}}. 
\end{equation} 
As we already observed, $\beta_i-1$ counts the number of pairs $(i,j)$
in $\NR(h)$ in the
$i$-th row. Put another way, the set of pairs in $\NR(h)$ with first
coordinate equal to $i$ can also be expressed as 
$\{(i, i-k) \mid 1 \leq k \leq \beta_i-1\}
$
and in particular we see that 
the differences $i-(i-k)=k$ of the coordinates
range precisely between $1$ and $\beta_i-1$. Using the same reasoning
for each $i \in [n]$, 
the last expression in~\eqref{eq:beta_i} can be
re-written as 
\begin{equation}\label{eq:heights} \prod_{(i,j) \in \NR(h)} \frac{1 - s^{2(\height(i,j)+1)}}{1 - s^{2 \cdot \height(i,j)}}
\end{equation}
where $\height(i,j):=i-j$ is called the height\footnote{Here, contrary to customary usage, we require that the height of a negative root is a \emph{positive} integer.}.
But now we may decompose the terms in~\eqref{eq:heights} according to
columns instead of rows. In this case, from the definition of $\NR(h)$
it is straightforward to rewrite~\eqref{eq:heights} as
\begin{align*}
\prod_{j=1}^{n}\prod_{k=1}^{h(j)-j}\frac{1-s^{2(k+1)}}{1-s^{2k}} 
=&\prod_{j=1}^{n}\frac{1-s^{2(h(j)-j+1)}}{1-s^2}. 
\end{align*}
This proves the claim. 
\end{proof} 

We now wish to relate the Hilbert series of
$H^\ast(\Hess(\mathsf{N},h))$ to the Hilbert series of the
quotient ring $\Q[x_1,\ldots,x_n]/\check \Ih$. This will in turn allow
us to compute and compare the Hilbert series of
$H^\ast_S(\Hess(\mathsf{N},h))$ and $\Q[x_1,\ldots,x_n,t]/\Ih$.
However, 
in order to accomplish this, we must first analyze the relationship between the series $f_{i,j}$
defined in Section~\ref{sec:f} and the 
series $\cf_{i,j}\in\Q[x_1, \ldots, x_n]$ 
defined in \eqref{eq:definition of f check}. It turns out that
$\cf_{i,j}$ is obtained from $f_{i,j}$ by setting the variable $t$
equal to $0$. 

\begin{lemma}\label{lem:formula for f check}
For all $n\geq i\geq j\geq 1$, we have
\begin{align}\label{formula for f check}
f_{i,j}(x_1,\ldots, x_n, t=0) = 
\check{f}_{i, j} 
=\sum_{k=1}^j\Big(\x_k\prod_{\ell=j+1}^i(\x_k-\x_\ell)\Big).
\end{align}
\end{lemma}
\begin{proof}
The second equality is just the definition \eqref{eq:definition of f
  check}, so we only need to prove the first equality.
Let $f'_{i,j}:=f_{i,j}(x_1,\dots,x_n, 0) \in \Q[x_1,\ldots,x_n]$. We
wish to show $f'_{i,j} = \cf_{i,j}$. From Definition~\ref{definition:fij}
we immediately see that 
these polynomials satisfy the following recursion relations: 
\begin{equation}\label{eq:recursion} 
\begin{split} 
&f'_{j,j} = \sum_{k=1}^j x_k \quad \textup{ for } j \in [n], \\
&f'_{i,j} = f'_{i-1,j-1} + (x_j - x_i) f'_{i-1,j} \quad \textup{ for } n \geq
i > j \geq 1.
\end{split}
\end{equation}
We introduce a total order on the set $\{(i,j) \in [n] \times [n] \mid
i \geq j\}$ by the condition 
\[
(i', j') \leq (i,j) \quad \textup{ if and only if } \quad  (i'-j' < i-j) \textup{ or }
(i'-j' = i-j \textup{ and } i' \leq i)
\]
and we prove the claim by induction on this total order. When $i=j$,
it is clear by~\eqref{eq:recursion} that the claim holds. Let $i>j$
and assume the claim holds for $(i',j')$ less than $(i,j)$. Since we
have $(i-1,j-1) < (i,j)$ and $(i-1,j)<(i,j)$ by definition of our
total order, the inductive hypothesis and~\eqref{eq:recursion} show
that 
\begin{align*}
f'_{i,j} & = f'_{i-1,j-1} + (x_j- x_i) f'_{i-1,j} \\
 & = \sum_{k=1}^{j-1} \Big(x_k \prod_{\ell=j}^{i-1} (x_k - x_\ell)\Big) + (x_j
 - x_i) \sum_{k=1}^j \Big(x_k \prod_{\ell=j+1}^{i-1} (x_k - x_\ell) \Big) \\
 & = \sum_{k=1}^{j-1} \Big(x_k \prod_{\ell=j+1}^{i-1} (x_k - x_\ell)\Big)
 \Big((x_k-x_j) + (x_j - x_i)\Big) + (x_j - x_i) x_j \prod_{\ell=j+1}^{i-1}
 (x_j - x_\ell) \\
 & = \sum_{k=1}^{j-1} \Big(x_k \prod_{\ell=j+1}^i (x_k - x_\ell)\Big) + x_j
 \prod_{\ell=j+1}^i (x_j - x_\ell) = \sum_{k=1}^j \Big(x_k
 \prod_{\ell=j+1}^i (x_k - x_\ell)\Big) = \cf_{i,j}
\end{align*} 
as desired. 
\end{proof}

For future reference, we also record the degrees of the polynomials
$f_{i,j}$ and $\cf_{i,j}$, both of which are immediate from their definitions. 

\begin{lemma} 
The degree of $f_{i,j}$ and $\cf_{i,j}$ in the variables $x_i$ and $t$
is $i-j+1$. With respect to the grading in $\Q[x_1,\ldots,x_n,t]$ and
$\Q[x_1,\ldots,x_n]$ respectively, we have 
\begin{align}\label{eq:deg of f and f check}
\mathrm{deg}(f_{i,j}) = \mathrm{deg}(\cf_{i,j}) = 2(i-j+1).
\end{align}
\end{lemma} 

We denote by 
\begin{align}\label{eq:def of Ih check}
\check \Ih := (\cf_{h(j),j} \mid 1\leq j\leq n ) \subset \Q[x_1,\ldots,x_n]
\end{align}
the ideal of
$\Q[x_1,\ldots,x_n]$ generated by the polynomials $\cf_{h(j),j}$. 
Our next goal is to relate the polynomials $\cf_{h(j),j}$ (and the
ideal $\check \Ih$ they generate) with the Hilbert series
$F(H^\ast(\Hess(\mathsf{N},h)), s)$. Armed with Lemma~\ref{lemma:hesshil} and the
results summarized in Remark~\ref{remark:characterizations of regular
  sequences}, we can accomplish this goal once we show that the
$\cf_{h(j),j}$ form a regular sequence, which we do in the next two
lemmas. 

\begin{lemma} \emph{(}\cite[Exercise 1, page 74]{fult97}\emph{)}
Let $m$ be an arbitrary positive integer and $\mathsf{y}_1, \ldots,
\mathsf{y}_m$ be indeterminates. For $i$ a positive integer let  $e_i(\mathsf{y}) := \sum_{1 \leq k_1< \cdots < k_i \leq m}
\mathsf{y}_{k_1} \cdots \mathsf{y}_{k_i}$ and $\mathsf{p}_i(\mathsf{y}) := \sum_{k=1}^m \mathsf{y}_k^i$
be the $i$-th elementary symmetric polynomial and the $i$-th power sum
respectively. 
Then we have
\begin{align}\label{eq:4-2 =0 form}
-\sum_{r=1}^{q-1}(-1)^{r}e_{r}(\mathsf{y})\emph{$\mathsf{p}$}_{q-r}(\mathsf{y})=(-1)^q qe_q(\mathsf{y})+ \emph{$\mathsf{p}$}_q(\mathsf{y}) \quad \text{for any $q\geq 1$}.
\end{align}
\end{lemma}

\begin{lemma} \label{lemm:4-2}
The polynomials $\cf_{h(1),1}, \cf_{h(2),2}, \cdots, \cf_{h(n),n}$ form a regular sequence in $\Q[x_1,\dots,x_n]$. 
\end{lemma}

\begin{proof}
We use Remark \ref{remark:characterizations of regular sequences}\eqref{rem:req seq zero-set} to prove this claim, that is, we show that the solution set in $\C^n$ of the equations $\cf_{h(j),j}=0$ for all $j\in [n]$ consists of only the origin $\{0\}$.
Observe that if $\cf_{h(j),j}=0$ for all $j\in [n]$ then from Lemma~\ref{relation for Ih and Ih'} we also have $\cf_{n,j}=0$ for all $j\in [n]$ since we have \eqref{formula for f check} and the substitution $t=0$ is a ring homomorphism from $\Q[x_1,\dots,x_n,t]$ to $\Q[x_1,\dots,x_n]$. 
Hence it suffices to prove the statement of the lemma
in the special case 
when $h(j)=n$ for all $j \in [n]$, i.e.  
that if $\cf_{n,j}=0$ for all $j \in [n]$ then $\x_j=0$ for all $j \in [n]$. 

To prove this, we first claim that for $j\in[n]$ we have
\begin{align}\label{eq:f check and power sum}
\cf_{n,j}=\sum_{i=0}^{j-1}(-1)^ie_i(\x_{n+2-j},\dots,\x_n)\mathsf{p}_{j-i}(x)
\end{align}
where we denote $\mathsf{p}_{j-i}(x)=\mathsf{p}_{j-i}(x_1,\dots,x_n)$.
This equality holds since 
from \eqref{formula for f check} and \eqref{eq:4-2 =0 form} the LHS is 
\begin{align*} 
\cf_{n,j}&=\sum_{k=1}^{n+1-j} \Big(x_k\prod_{\ell=n+2-j}^n(x_k-x_\ell)\Big)\\
&=\sum_{k=1}^{n+1-j}x_k\Big(\sum_{i=0}^{j-1}(-1)^ie_i(x_{n+2-j},\dots,x_n)x_k^{j-1-i}\Big)\\
&=\sum_{i=0}^{j-1}(-1)^ie_i(x_{n+2-j},\dots,x_n)\mathsf{p}_{j-i}(x_1,\dots,x_{n+1-j})\\
&=\mathsf{p}_{j}(x_1,\dots,x_{n+1-j})+\sum_{i=1}^{j-1}(-1)^ie_i(x_{n+2-j},\dots,x_n)\Big(\mathsf{p}_{j-i}(x)-\mathsf{p}_{j-i}(x_{n+2-j},\dots,x_n)\Big)\\
&=\mathsf{p}_{j}(x_1,\dots,x_{n+1-j})+\mathsf{p}_{j}(\x_{n+2-j},\dots,\x_n) +\sum_{i=1}^{j-1}(-1)^ie_i(x_{n+2-j},\dots,x_n)\mathsf{p}_{j-i}(\x) \\
&\qquad\qquad\qquad\qquad\qquad\qquad\qquad\qquad\qquad\qquad\qquad\qquad\qquad\text{by \eqref{eq:4-2 =0 form} and $e_{j}(\x_{n+2-j},\dots,\x_n)=0$} \\
&=\sum_{i=0}^{j-1}(-1)^ie_i(\x_{n+2-j},\dots,\x_n)\mathsf{p}_{j-i}(x).
\end{align*}
Now \eqref{eq:f check and power sum} shows that the transition matrix
from $\mathsf{p}_1(x),\dots,\mathsf{p}_n(x)$ to
$\cf_{n,1},\dots,\cf_{n,n}$ is lower-triangular with diagonal
  entries all equal to $1$, and hence the ideal of $\Q[x_1,\dots,x_n]$ generated by $\cf_{n,1},\dots,\cf_{n,n}$ and the ideal of $\Q[x_1,\dots,x_n]$ generated by power sums $\mathsf{p}_1(x),\dots,\mathsf{p}_n(x)$ are the same ;
\begin{align}\label{eq:f check ideal and power sum ideal}
(\cf_{n,1},\dots,\cf_{n,n})=(\mathsf{p}_1(x),\dots,\mathsf{p}_n(x))
\subset\Q[x_1,\dots,x_n].
\end{align}

Recall that we assume that $\cf_{n,j}=0$ for all $j$  with $1 \leq j \leq n$. In particular, 
we obtain $\mathsf{p}_j(x)=0$ for all $j$ with $1 \leq j \leq n$. It is well-known and easy to prove that this implies that $x_1=\cdots=x_n=0$. Now the claim follows from the characterization of regular sequences in Remark~\ref{remark:characterizations of regular sequences}\eqref{rem:req seq zero-set}.
\end{proof}

A computation of the Hilbert series is now straightforward. 

\begin{corollary}\label{cor:hilbert of ordinary coh} 
The Hilbert series of the graded $\Q$-algebras $H^*(\Hess(\N,h))$ and
$\Q[x_1,\ldots,x_n]/\check \Ih$ are equal, i.e.  
\[
F(H^*(\Hess(\N,h)), s) = F(\Q[x_1,\ldots,x_n]/\check \Ih, s) = 
\prod_{j=1}^{n}\frac{1-s^{2(h(j)-j+1)}}{1-s^2}.
\]
\end{corollary}

\begin{proof} 
Recalling that $\deg \cf_{h(j),j}=2(h(j)-j+1)$ from \eqref{eq:deg of f and f check}, Lemma \ref{lemm:4-2} and Remark \ref{remark:characterizations of regular sequences}\eqref{rem:req seq Hilb} show that 
\[
F(\Q[x_1,\ldots,x_n]/\check \Ih, s) = 
\prod_{j=1}^{n}\frac{1-s^{2(h(j)-j+1)}}{1-s^2}.
\]
Thus, together with Lemma \ref{lemma:hesshil}, we obtain the claim.
\end{proof}

We now turn our attention to the main goal of this section, which is
  the computation of the Hilbert series
  $F(H^*_S(\Hess(\mathsf{N},h)),s)$ in terms of the ideal $\Ih$
  generated by the polynomials $f_{h(j),j}$. We
  continue to use the technique of regular sequences. Indeed, our
  first step, Lemma~\ref{lemm:5-1} below, states that 
the $n+1$ homogeneous
polynomials 
$\{f_{h(1),1}, \dots, f_{h(n),n}, t\}$ form a regular sequence in
$\Q[x_1, \ldots, x_n, t]$; this in fact follows easily from the above
arguments. 

\begin{lemma} \label{lemm:5-1}
The polynomials $f_{h(1),1},\dots, f_{h(n),n}, t$ form a regular
sequence in $\Q[x_1, \ldots, x_n,t]$.
Moreover, $\Q[x_1,\dots,x_n,t]$ is a finitely generated and free
$\Q[f_{h(1),1},\dots, f_{h(n),n}, t]$-module.
\end{lemma}
\begin{proof}
By Remark~\ref{remark:characterizations of regular sequences}\eqref{rem:req seq zero-set}, the first claim follows from Lemma~\ref{lem:formula for f check} and Lemma~\ref{lemm:4-2}.
The second
claim then follows from the characterization in
Remark~\ref{remark:characterizations of regular sequences}\eqref{rem:req seq alg-ind} and
Lemma~\ref{lemma:finite generation}. 
\end{proof} 

As we have just seen, $\Q[x_1,\ldots, x_n,t]$ is a free and finitely
generated $\Q[f_{h(1),1}, \ldots, f_{h(n),n}, t]$-module. A
straightforward argument (using, for instance, a choice of
basis together with the fact that $f_{h(1),1}, \ldots, f_{h(n),n}, t$ are algebraically independent over $\Q$) then shows that the quotient $\Q[x_1,\ldots,x_n,t]/(f_{h(1),1},
\ldots, f_{h(n),n}) = \Q[x_1,\ldots,x_n,t]/\Ih$ is a free and finitely
generated module over $\Q[t]$.  
We record the
following.

\begin{corollary}\label{lemm:5-1.5}
As $\Q[t]$-modules and hence as graded $\Q$-vector spaces, we have 
\[
\Q[x_1, \ldots, x_n, t]/\Ih \cong \Q[t] \otimes_{\Q} \left( \Q[x_1, \ldots,
  x_n]/\check \Ih \right).
\]
In particular, 
\[
F(\Q[x_1,\ldots,x_n,t]/\Ih, s) = \frac{F(\Q[x_1,\ldots,x_n]/\check \Ih, s)}{1-s^2}.
\]
\end{corollary}

\begin{proof}
Since $\Q[x_1,\ldots,x_n,t]/\Ih$ is a finitely generated free $\Q[t]$-module as observed above, we have 
\[
\Q[x_1,\ldots,x_n,t]/\Ih \cong \Q[t] \otimes_{\Q}
(\Q[x_1,\ldots,x_n,t]/(f_{h(1),1}, \ldots, f_{h(n),n}, t))
\]
as $\Q[t]$-modules.
The module in the right hand side is naturally isomorphic to $\Q[t] \otimes_{\Q} (\Q[x_1, \ldots, x_n]/\check \Ih)$ by the definition \eqref{eq:def of Ih check} of $\check \Ih$.
Hence, we obtain the first claim.
In particular this means 
\begin{equation*}
  \begin{split}
    F(\Q[x_1,\ldots,x_n,t]/\Ih, s) & =
    F(\Q[t],s)F(\Q[x_1,\ldots,x_n]/\check \Ih,s) = \frac{F(\Q[x_1,\ldots,x_n]/\check \Ih,s)}{1-s^2 } 
  \end{split}
\end{equation*}
as desired. 
\end{proof} 

The main result of this section now follows easily. 

\begin{proposition} \label{prop:hilbert series equal} 
The Hilbert series of the graded $\Q$-algebras $H^*_S(\Hess(\N,h))$
and $\Q[x_1,\ldots,x_n,t]/\Ih$ are equal, i.e. 
$F\big(\Q[x_1,\dots,x_{n},t]/\Ih,s\big)=F\big( H^*_S(\Hess(\mathsf{N},h)),s).$
\end{proposition}

\begin{proof}
Since $\Hess(\N,h)$ admits a paving by complex affines  (cf. discussion
before~\eqref{eq:equiv formality of HessN}), we have that 
$
H^*_S(\Hess(\N,h)) \cong H^*_S(\pt) \otimes_{\Q} H^*(\Hess(\N,h))
$
as $H^*_S(\pt)$-modules and hence also as graded $\Q$-vector
spaces. In particular, 
\begin{equation*}
    F(H^*_S(\Hess(\N,h)), s)  = F(H^*_S(\pt),s) F(H^*(\Hess(\N,h)), s).
\end{equation*}
Also since $H^*_S(\pt) \cong \Q[t]$ is a
polynomial ring in one variable we have $F(H^*_S(\pt),s) =
\frac{1}{1-s^2}$ and we obtain 
\begin{equation*}
  \begin{split}
    F(H^*_S(\Hess(\N,h)), s) = \frac{F(H^*(\Hess(\N,h)),
    s)}{1-s^2} = \frac{F(\Q[x_1,\ldots,x_n]/\check I_h, s)}{1-s^2}  
    = F(\Q[x_1,\ldots,x_n,t]/\Ih, s) 
  \end{split}
\end{equation*}
by Corollary~\ref{cor:hilbert of ordinary coh} and Corollary~\ref{lemm:5-1.5},
as desired. 
\end{proof}

\bigskip
\section{Second part of proof of Theorem~\ref{theorem:reg nilp Hess cohomology}\ and proof of 
  Theorem A}\label{sec:proof of first
  main theorem} 
The purpose of this section is to complete the proof of 
Theorem~\ref{theorem:reg nilp Hess cohomology} and hence also of
Theorem A. 
Specifically, we prove that 
the graded $\Q[t]$-algebra homomorphism 
\begin{equation*}
\varphi_h : \mathbb{Q}[x_1,\dots,x_n,t]/\Ih \rightarrow H^{\ast}_{S}(\mathcal \Hess(\mathsf{N},h))
\quad ; \quad 
x_i \mapsto \SChNil_i , \quad t\mapsto t
\end{equation*}
(which was shown to be well-defined in Corollary \ref{cor for
  well-def}) is in fact an isomorphism. 
Before launching into the proof we sketch the essential idea. As mentioned in
the introductory remarks to Section~\ref{sec:hilbert}, 
if two vector spaces are a priori known to
have the same dimension, then a linear map between them is an
isomorphism if and only if it is injective if and only if it is
surjective. We will now use this elementary linear algebra fact to its
full effect, given that we have shown  
that the dimensions of
$H^*(\Hess(\N,h))$ and $\Q[x_1,\ldots,x_n]/\check \Ih$ coincide (Corollary~\ref{cor:hilbert of ordinary coh}), and
that the dimensions of (the graded pieces of) $H^*_S(\Hess(\N,h))$ and
$\Q[x_1,\ldots,x_n,t]/\Ih$ coincide (Proposition~\ref{prop:hilbert series equal}). 
The other essential trick we use is that of localization: instead of
directly attacking the problem of showing that $\varphi_h$ is
injective (which, as we said above, would suffice to show that
$\varphi_h$ is an isomorphism), we show first that a certain localization
$R^{-1}\varphi_h$ is an isomorphism by using the localization theorem in
equivariant topology. We make this more precise below.

Let $R=\Q[t]\backslash\{0\}$ and consider the induced homomorphism 
$R^{-1}\varphi_h$ between the $R^{-1}\Q[t]$-algebras
\begin{align}\label{localization of varphi}
R^{-1}\varphih\colon R^{-1}\big(\Q[x_1,\dots,x_n,t]/\Ih\big)\to R^{-1}H^*_S(\Hess(\mathsf{N},h)).
\end{align}

Recall from 
Section~\ref{subsec:setup} that the $S$-equivariant cohomology of the full flag
variety $\Flags(\C^n)$ is generated as an $H^*_S(\pt)$-module by the
$S$-equivariant first Chern classes of the tautological line bundles. 
In our setting, this means that for the special case 
$h=(n,n,\ldots,n)$, by the definition of $\varphih$ we already know
that 
\begin{equation*}
\Q[x_1,\ldots,x_n,t] \to H^*_S(\Hess(\N,h)) = H^*_S(\Flags(\C^n))
\end{equation*}
is surjective. We harness this fact, together with the localization
theorem in equivariant topology and our explicit 
description of the $S$-fixed point set of $\Hess(\N,h)$, to show that
$R^{-1}\varphih$ is surjective for general $h$. The fact that
$R^{-1}\varphih$ is an isomorphism then follows from a simple
dimension-counting argument over the field 
$R^{-1}\Q[t] = \Q(t)$ of rational functions in one variable, as
suggested in the introductory remarks above. 

\begin{lemma} \label{lemm:5-3}
The map $R^{-1}\varphih$ in \eqref{localization of varphi} is an isomorphism.
\end{lemma}

\begin{proof}
For simplicity of notation in what follows, for the special case
$h=(n,n,\ldots,n)$ with
$\Hess(\N,h)=\Flags(\C^n)$, we denote the corresponding ideal by $I$
and the corresponding map by $\varphi$. Then, as discussed above,
$\varphi$ is surjective. In particular, $R^{-1}\varphi$ is also
surjective. 

Next, recall from Lemma~\ref{relation for Ih and Ih'} that if we have $h \subset h'$ for two Hessenberg
functions $h$ and $h'$ then $I_{h'} \subset 
\Ih$ and hence there exists a natural induced map
$\Q[x_1,\ldots,x_n, t]/I_{h'} \to \Q[x_1,\ldots,x_n, t]/\Ih$. In our case,
for any Hessenberg function $h$ it is always true that $h \subset
h'$ for the ``largest'' Hessenberg function $h':=(n,n,\ldots,n)$, so
we conclude from Lemma~\ref{relation for Ih and Ih'} and a localization that there
exists a natural map 
\begin{equation}\label{eq:localized I to Ih} 
R^{-1}(\Q[x_1,\ldots,x_n,t]/I) \to R^{-1}(\Q[x_1,\ldots,x_n,t]/\Ih).
\end{equation}
In fact, we may enlarge this to the following commutative diagram
\[
\begin{CD}
R^{-1}\big(\Q[x_1,\dots,x_n,t]/I\big) @>R^{-1}\varphi >> R^{-1}H^*_S(\Flags(\C^n)) @>>\cong> R^{-1}H^*_S(\Flags(\C^n)^S) \\
@VVV @VVV @VVV \\
 R^{-1}\big(\Q[x_1,\dots,x_n,t]/\Ih\big) @>R^{-1}\varphih >>
 R^{-1}H^*_S(\Hess(\N, h)) @>>\cong> R^{-1}H^*_S(\Hess(\N, h)^S) 
\end{CD}
\]
where the map in~\eqref{eq:localized I to Ih} is the leftmost vertical
arrow, and all other unlabelled maps are 
induced
from the geometric inclusion maps. 
The two horizontal maps in the square on the right 
are both isomorphisms by the localization theorem in equivariant
topology \cite[p.40]{hsi}. Moreover, 
the right-most vertical map is a surjection since the $S$-fixed point set
$\Hess(\N, h)^S$ is a subset of the $S$-fixed point set of
$\Flags(\C^n)$. Thus the middle vertical map must be surjective, and
from there it also follows that 
$R^{-1}\varphih$ is surjective. 

To show that $R^{-1}\varphih$ is in fact an isomorphism, we now compare
the dimensions of $R^{-1}(\Q[x_1,\ldots,x_n,t]/\Ih)$ and
$R^{-1}H^*_S(\Hess(\N,h))$ as vector spaces over 
$R^{-1}\Q[t]=\Q(t)$ the field of rational functions in one variable
$t$.
Since we have $H^*_S(\Hess(N,h)) \cong H^*_S(\pt) \otimes_{\Q} H^*(\Hess(\N,h))$ by \eqref{eq:equiv formality of HessN}, 
we obtain
\[
R^{-1}H^*_S(\Hess(\N, h))\cong R^{-1}H^*_S(\pt) \otimes_{\Q}
H^*(\Hess(\N,h)) \cong \Q(t) \otimes_{\Q} H^*(\Hess(\N,h)).
\]
In particular, the dimension of $R^{-1}H^*_S(\Hess(\N,h))$ as a
$\Q(t)$-vector space is the dimension of $H^*(\Hess(\N,h))$ as a
$\Q$-vector space. 
On the other hand, from Corollary~\ref{lemm:5-1.5} 
we also know that, as a $\Q[t]$-module, we have 
\[
\Q[x_1,\ldots, x_n,t]/\Ih \cong \Q[t] \otimes_{\Q}
(\Q[x_1,\ldots,x_n]/\check \Ih)
\]
which means that 
\[
R^{-1}(\Q[x_1,\ldots, x_n,t]/\Ih) \cong \Q(t) \otimes_{\Q}
(\Q[x_1,\ldots,x_n]/\check \Ih)
\]
and hence the dimension of $R^{-1}(\Q[x_1,\ldots,x_n,t]/\Ih)$ as a $\Q(t)$-vector space is the
dimension of $\Q[x_1,\ldots,x_n]/\check \Ih$ as a $\Q$-vector space. But we
have just seen in Corollary~\ref{cor:hilbert of ordinary coh} 
that these two rings $H^*(\Hess(\N,h))$ and $\Q[x_1,\ldots,x_n]/\check \Ih$ have the same Hilbert series, and in
particular are of the same dimension. 
Since $R^{-1}\varphih$ has been shown to be a surjective map
between vector spaces of the same dimension, it must be an
isomorphism, as desired. 
\end{proof}

We now wish to use the fact that $R^{-1}\varphih$ is an isomorphism to
deduce that $\varphih$ must be an isomorphism. 

\begin{proof}[Proof of Theorem~\ref{theorem:reg nilp Hess cohomology}]
Consider the following commutative diagram: 
\[
\begin{CD}
\Q[x_1,\dots,x_n,t]/\Ih @>\varphih>> H^*_S(\Hess(\N, h)) \\
@VVV @VVV \\
R^{-1}(\Q[x_1,\dots,x_n,t]/\Ih) @>R^{-1}\varphih >\cong >
R^{-1}H^*_S(\Hess(\N, h))
\end{CD}
\]
The left vertical map is injective since $\Q[x_1,\ldots, x_n,t]/\Ih$ is a  
free $\Q[t]$-module by Corollary~\ref{lemm:5-1.5}.
We just saw that the bottom horizontal
map is an isomorphism in Lemma~\ref{lemm:5-3}. 
From the commutativity of the diagram we may conclude that $\varphih$
is an injection. But by Proposition~\ref{prop:hilbert series equal}  
we know
that the Hilbert series of the target and the domain agree, showing
that their graded pieces have the same dimension. Thus, $\varphih$ is an isomorphism, as desired. 

Finally, this implies that the $S$-equivariant cohomology ring $H^*_S(\Hess(\N, h))$ is generated by the first Chern class $\SChNil_i$ of the $i$-th tautological line bundle (over $\Flags(\C^n)$) restricted to $\Hess(\N, h)$, and hence the restriction map $H^*_T(\Flags(\C^n))\rightarrow H^*_S(\Hess(\N, h))$ is surjective.
\end{proof} 

\bigskip
We can now prove Theorem A.  
Indeed, since the odd degree cohomology groups of $\Hess(\mathsf{N},h)$ vanish as discussed in Section~\ref{sec:H}, by setting $t=0$ we obtain the ordinary cohomology. 
Recall from Lemma~\ref{lem:formula for f check}, we have
\begin{align*}
\check{f}_{i,j} = f_{i,j}(x_1,\dots,x_n,t)|_{t=0}=\sum_{k=1}^j\Big(\x_k\prod_{\ell=j+1}^i(\x_k-\x_\ell)\Big)
\quad (n\geq i\geq j\geq 1).
\end{align*}
Let 
\begin{align}\label{eq:def of ord Ch for flag}
\ChNil_i \in H^2(\Hess(\N,h))
\end{align}
be the image of the Chern class $\ChFlag_i\in H^2(\Flags(\C^n))$ under the restriction map $H^*(\Flags(\C^n))\rightarrow H^*(\Hess(\N,h))$. That is, $\ChNil_i$ is the first Chern class of the tautological line bundle over $\Flags(\C^n)$ restricted to $\Hess(\N,h)$.
(Its equivariant version $\SChNil_i \in H_S^2(\Hess(\N,h))$ is defined in \eqref{def of S eq ch in Nil}.)

\begin{proof}[Proof of Theorem A]
Consider the forgetful map $H_S^*(\Hess(\N,h))\rightarrow H^*(\Hess(\N,h))$ which sends the ideal of $H_S^*(\Hess(\N,h))$ generated by $t$ to zero. This map is surjective since $\Hess(\N,h)$ admits a paving by complex affines \cite{ty} as mentioned in Section~\ref{sec:H} and hence the Serre spectral sequence collapses at $E_2$-stage \cite[Chapter III, Theorem 2.10]{mi-to}.
Thus, from Theorem~\ref{theorem:reg nilp Hess cohomology} together with \eqref{formula for f check}, we obtain 
\[
\Q[x_1,\dots,x_n]/\check \Ih \cong H_S^*(\Hess(\N,h))/(t) \cong H^*(\Hess(\N,h)) 
\]
by sending each $x_i$ to $\ChNil_i$ defined above where $\check \Ih=\big(\cf_{h(j),j} \mid 1\leq j\leq n\big)$.
Since the (equivariant) restriction map $H^*_T(\Flags(\C^n))\rightarrow H^*_S(\Hess(\N, h))$ is surjective from Theorem~\ref{theorem:reg nilp Hess cohomology}, so is the restriction map $H^*(\Flags(\C^n))\rightarrow H^*(\Hess(\N,h))$.
\end{proof}

\bigskip
\section{The equivariant cohomology rings of regular semisimple Hessenberg
  varieties}\label{sec:background on reg ss Hess}

Our second main result, Theorem B, relates
  the ordinary cohomology ring $H^*(\Hess(\mathsf{N},h))$ of the
  regular nilpotent Hessenberg variety and the $\Sn$-invariant subring
  $H^*(\Hess(\mathsf{S},h))^{\Sn}$ 
  of the ordinary cohomology of the regular semisimple Hessenberg
  variety. In this section, we recall the definition of the
  $\Sn$-action on the $T$-equivariant cohomology
  $H^*_T(\Hess(\mathsf{S},h))$ -- which then induces an $\Sn$-action on
  $H^*(\Hess(\mathsf{S},h))$ -- where $T$ is the usual maximal torus defined in \eqref{eq:S and T}. It
  is this $\Sn$-action with respect to which we take the
  invariants in our Theorem B. We also record some preliminary results concerning this
  action. 

Let $h \in H_n$ be a Hessenberg function and $\Hess(\mathsf{S},h)$ the regular semisimple
Hessenberg variety associated to $h$ as defined in~\eqref{eq:def reg ss Hess}. 
Here and below, we use the notation
\begin{align}\label{eq:dim of HessS}
\dimX:=\dim_{\C} \Hess(\mathsf{S},h) 
\end{align}
where the computation of this quantity in terms of the Hessenberg function $h$ is described in \eqref{eq:dim of HessN and HessS}.

As we saw in Section~\ref{sec:H}, the maximal torus $T$ of $\text{GL}(n,\C)$ acts on
$\Flags(\C^n)$ preserving $\Hess(\mathsf{S}, h)$. It is also 
straightforward to see that 
\begin{equation*}
\Hess(\mathsf{S},h)^{\Tn}= \Flags(\C^n)^{\Tn} \cong \Sn.
\end{equation*}

Tymoczko described 
the equivariant cohomology ring
$H^*_{\Tn}(\Hess(\mathsf{S},h))$ as 
an algebra over $H^*(B\Tn)$ by using techniques introduced by
Goresky, Kottwitz, and MacPherson \cite{Go-Ko-Ma}, also called GKM
theory. For details we refer to \cite{tymo08} and below we only
state the facts relevant for our situation. 
Any Hessenberg variety (in Lie type A) admits a paving by complex affines (\cite[Theorem 7.1]{ty}), 
and hence the localization theorem of torus-equivariant topology
implies that the inclusion map of the fixed point set induces an injection
\begin{align*}
  &\iota_3:H^*_T(\Hess(\mathsf{S},h)) \into
  H^*_T(\Hess(\mathsf{S},h)^T)\cong \bigoplus_{w\in\Sn}\Q[t_1,\dots,t_n]. 
\end{align*}
where we identify $\Hess(\mathsf{S},h)^T=\Flags(\C^n)^T\cong\Sn$ as
above. 
For $\alpha \in H^*_T(\Hess(\mathsf{S},h)))$, since $\iota_3$ is
injective, by abuse of notation we denote also 
by $\alpha$ its image in the RHS. 
In particular we denote by $\alpha(w) \in \Q[t_1,\ldots,t_n]$ the
$w$-th component of $\alpha$ in the decomposition above. 

In this setting, GKM theory yields the following concrete description
of the image of $\iota_3$ \cite{tymo08}: 
\begin{equation}\label{GKM for X(h)}
H^*_{\Tn}(\Hess(\mathsf{S},h)) \cong \left\{ \alpha \in \bigoplus_{w\in\Sn} \Q[t_1,\dots,t_n] \left|
\begin{matrix} \text{ $\alpha(w)-\alpha(w')$ is divisible by $t_{w(i)}-t_{w(j)}$} \\
\text{ if there exist $1\leq j<i\leq n$ satisfying} \\
 \text{ $w'=w(j \ i)$ and $i \leq h(j)$} \end{matrix} \right.\right\}
\end{equation}
where $(j \ i) \in \mathfrak{S}_n$ denotes the element of
$\mathfrak{S}_n$ which transposes $i$ and $j$. 
We call the condition described in the right hand side of ~\eqref{GKM for X(h)} \textbf{the GKM condition (for $\Hess(\mathsf{S},h)$).}

Note that if the Hessenberg function $h$ is chosen to be
  $h=(n,n,\ldots,n)$, then the corresponding Hessenberg variety
  $\Hess(\mathsf{S},h)$ is equal to $\Flags(\C^n)$. Since 
  the condition $i\leq h(j)=n$ is always satisfied, 
  the corresponding presentation of $H^*_T(\Flags(\C^n))$ imposes \emph{more} conditions on a collection $(\alpha(w))_{w \in
  \mathfrak{S}_n}$ than the RHS of~\eqref{GKM for X(h)}.
  From this we can see that for an arbitrary Hessenberg function $h$, the restriction map
\[
H^*_T(\Flags(\C^n)) \to H^*_T(\Hess(\mathsf{S},h))
\]
is an injection by considering the following commutative diagram
\begin{equation}\label{eq:cd2 moved}
\begin{CD}
H^{\ast}_{T}(\Flags(\C^n))@>{\iota_1}>> H^{\ast}_{T}(\Flags(\C^n)^{T})\cong\displaystyle \bigoplus_{w\in \Sn} \Q[t_1,\dots,t_n]\\
@V{ }VV @VV{\text{id}}V\\
H^{\ast}_{T}(\Hess(\mathsf{S},h))@>{\iota_3}>> H^{\ast}_{T}(\Hess(\mathsf{S},h)^{T})\cong\displaystyle \bigoplus_{w\in \Sn} \Q[t_1,\dots,t_n].
\end{CD}
\end{equation}
It turns out that
the image
of $t_i \in H^*_T(\mathrm{pt}) = H^*(BT) \cong \Q[t_1,\ldots,t_n]$
under the composition 
\[
H^*_T(\mathrm{pt}) \to H^*_T(\Hess(\mathsf{S},h)) \stackrel{\iota_3}{\rightarrow}
H^*_T(\Hess(\mathsf{S},h)^T)
\]
is given simply by attaching the monomial $t_i$ to each vertex
in the
graph (see also \eqref{eq:constant polynomials}). 
Henceforth, by slight abuse of
notation, we denote also by $t_i$ the equivariant cohomology
class in $H^*_T(\Hess(\mathsf{S},h))$ obtained in this way. In
particular it is immediate, with this notation, that 
\begin{equation}\label{eq:constant classes}
t_i(w) = t_i \textup{ for all } w \in \mathfrak{S}_n.
\end{equation}
This class $t_i\in H^2_T(\Hess(\mathsf{S},h))$ is in fact the image of $t_i\in H^2(BT)$ (defined after~\eqref{eq:equiv coh of pt}) under the canonical homomorphism $H^*(BT)\rightarrow H^*_T(\Hess(\mathsf{S},h))$.

We now describe the $\Sn$-action on $H^*_{\Tn}(\Hess(\mathsf{S},h))$
constructed explicitly by 
Tymoczko \cite{tymo08}. 
(This action
also induces an $\Sn$-action on the
ordinary cohomology $H^*(\Hess(\mathsf{S},h))$, as we explain
below.) First, define an
$\Sn$-action on the polynomial ring $\Q[t_1,\ldots,t_n]$ in the
standard way by permuting the indices of the variables, i.e.  
for $t_i \in \Q[t_1,\ldots, t_n]$ and $v \in \Sn$ we define 
$v \cdot t_i := t_{v(i)}$.
This induces an $\Sn$-action on $\Q[t_1,\ldots,t_n]$ by
$\Q$-linear ring homomorphisms. Recall that by~\eqref{GKM for X(h)}
the data of an element $\alpha \in H^*_T(\Hess(\mathsf{S},h))$ is
equivalent to a list $(\alpha(w))_{w \in \Sn}$ of polynomials in
$\Q[t_1,\ldots,t_n]$ satisfying the GKM conditions. With this
understanding, Tymoczko defines, for $v \in \Sn$ and $\alpha = (\alpha(w))_{w
  \in \Sn}$, the element $v \cdot \alpha$ by the formula 
\begin{equation}\label{def of Tymoczko rep} 
(v\cdot \alpha)(w) := v \cdot \alpha(v^{-1}w) \textup{ for all } w \in
\Sn.
\end{equation}
It is also straightforward to check that the class $v \cdot \alpha$ thus
defined 
again satisfies the GKM conditions
and hence this action is well-defined. 
It is straightforward to check that the ``constant'' class $t_i \in H^*_T(\Hess(\mathsf{S},h))$ corresponding to $t_i \in H^*(BT)$ described in~\eqref{eq:constant classes} satisfies $v\cdot t_i=t_{v(i)}$ for $v\in\Sn$.

The above lemma implies that the ideal of $H^*_{\Tn}(\Hess(\mathsf{S},h))$
generated by the classes $t_1,\dots,t_n$ is preserved by Tymoczko's 
$\Sn$-action. 
Since the odd degree cohomology of $\Hess(\mathsf{S},h)$ vanishes, 
the forgetful map $H^*_T(\Hess(\mathsf{S},h)) \to
H^*(\Hess(\mathsf{S},h))$ is surjective \cite[Ch III, Theorem 2.10 and Theorem 4.2]{mi-to}, and the kernel is precisely
the ideal generated by the $t_i$. Thus, we obtain an isomorphism 
\begin{equation}\label{eq:HT(Hess) surjects to H(Hess)}
H^*(\Hess(\mathsf{S},h)) \cong
H^*_T(\Hess(\mathsf{S},h))/(t_1,\ldots,t_n)
\end{equation}
and the fact that the ideal $(t_1,\ldots,t_n)$ is $\Sn$-invariant
implies that the RHS, and hence also the LHS, has a well-defined
$\Sn$-action. The following is then straightforward from the
definitions. 

\begin{lemma}\label{lemma:Flags and Hess equivariant} 
The diagram
\begin{equation*}\label{eq:Sn equivariant diagram} 
\xymatrix{
H^*_T(\Flags(\C^n)) \ar[r] \ar[d] & H^*_T(\Hess(\mathsf{S},h)) \ar[d]
\\
H^*(\Flags(\C^n)) \ar[r] & H^*(\Hess(\mathsf{S},h)) 
}
\end{equation*}
commutes, where the horizontal maps are induced from the inclusion map $\Hess(\mathsf{S},h) \into \Flags(\C^n)$ and the vertical arrows are forgetful maps. Moreover, all maps in the diagram are $\Sn$-equivariant. 
\end{lemma} 

Note that $\Sn$ naturally acts on $\Flags(\C^n)$ on
  the left by multiplication by permutation matrices, and it is
  well-known (see e.g. \cite{knu}, \cite{tymo08}) that this induces
  Tymoczko's $\Sn$-representation on $H_T^*(\Flags(\C^n))$ and
  $H^*(\Flags(\C^n))$. 
Note that this $\Sn$-action is obtained by restricting the natural
$\text{GL}(n,\C)$-action on $\Flags(\C^n)$. The path-connectedness of
$\text{GL}(n,\C)$ implies that the induced
$\text{GL}(n,\C)$-representation on $H^*(\Flags(\C^n))$ is trivial.
Hence we obtain the following. 

\begin{lemma}\label{lem:rep on flag is trivial}
  \emph{(}\cite[Proposition 4.4]{tymo08}\emph{)}
  The $\Sn$-representation on $H^*(\Flags(\C^n))$
  is trivial.
\end{lemma}

\bigskip
\section{Properties of the $\Sn$-action on
  $H^*_{\Tn}(\Hess(\mathsf{S},h))$}\label{sec:properties of Tymoczko action}
In this section, we
  prepare for the proof of Theorem B by analyzing in more
  detail the properties of the $\Sn$-action on
  $H^*_T(\Hess(\mathsf{S},h))$ defined in Section~\ref{sec:background on reg ss Hess}. Our first result is Proposition~\ref{prop:invariant subring is polynomial ring}, which
  explicitly identifies the $\Sn$-invariant subring of
  $H^\ast_T(\Hess(\mathsf{S},h))$ (hence also of
  $H^\ast_T(\Flags(\C^n))$ as a special case). Our second result is
  Proposition~\ref{prop:nondegenerate and invariant}, which states that there exists an
  $\Sn$-invariant non-degenerate pairing on the ordinary cohomology
  groups of complementary degree 
of $\Hess(\mathsf{S},h)$.

We begin with Proposition~\ref{prop:invariant subring is polynomial
  ring}. It will turn out that 
the $\Sn$-invariant subring of $H^\ast_T(\Hess(\mathsf{S},h))$ (and
hence also $H^\ast_T(\Flags(\C^n))$) is a copy of the polynomial ring $H^*_T(\pt)
\cong \Q[t_1,\ldots,t_n]$, 
but some care must be
taken in defining the embedding of $H^\ast_T(\pt)$ into
$H^\ast_T(\Hess(\mathsf{S},h))$ (and $H^\ast_T(\Flags(\C^n))$) that
achieves this isomorphism. Specifically, the embedding does \emph{not} 
take the element $t_i
\in H^*_T(\pt) \cong \Q[t_1,\ldots,t_n]$ to the ``constant class'' in
$H^*_T(\Flags(\C^n))$ (and in $H^*_T(\Hess(\mathsf{S},h))$) described
in~\eqref{eq:constant classes} which takes the constant value
$t_i(w)=t_i$ at all $w \in \Sn$, as one might initially expect. Instead, the images are defined to be
certain characteristic classes, as we now explain.
Recall from \eqref{def of T eq ch in flag} that
  $\TChFlag_i\in H^2_T(\Flags(\C^n))$ 
denotes the $T$-equivariant 
  first Chern class of the tautological line bundle $E_i/E_{i-1}$ over
  $\Flags(\C^n)$.  We denote by
\begin{align}\label{Chern classes for semisimple}
\TChSemi_i\in H_T^2(\Hess(\mathsf{S},h))
\end{align} 
the image of $\TChFlag_i$ under the restriction map 
$H_T^*(\Flags(\C^n))\rightarrow H_T^*(\Hess(\mathsf{S},h))$.

\begin{lemma}\label{lemma:chern classes invariant}
Let $i\in[n]$. The classes $\TChFlag_i\in H_T^2(\Flags(\C^n))$ and $\TChSemi_i\in H_T^2(\Hess(\mathsf{S},h))$ are $\Sn$-invariant. 
\end{lemma}

\begin{proof}
The following proof is independent of the choice of the Hessenberg function $h$, so since the
choice $h=(n,n,\ldots,n)$ yields $\Hess(\mathsf{S},h)=\Flags(\C^n)$
as a special case, it suffices to show the claim for
$\Hess(\mathsf{S},h)$. 
We have already seen from \eqref{loc of tilde tau} that $\TChFlag_i(w) =
t_{w(i)}$. 
By the definition of $\TChSemi_i$ and the commutativity of
\eqref{eq:cd2 moved}, we also have 
\begin{align}\label{localization of chern class for semisimple}
\TChSemi_i(w) =
t_{w(i)}.
\end{align} 
By \eqref{localization of chern class for semisimple} and the definition \eqref{def of Tymoczko rep} of the $\Sn$-action on $H_T^*(\Hess(\mathsf{S},h))$, we can compute that for any $w \in \Sn$ we have 
  \begin{align*}
    (v \cdot \TChSemi_i)(w) & = v \cdot \TChSemi_i(v^{-1}w)
    = v \cdot t_{v^{-1}w(i)} 
    = t_{v(v^{-1} w(i))} 
    = t_{w(i)} 
    = \TChSemi_i(w) 
  \end{align*}
as desired. 
\end{proof}

We now define a graded $\Q$-algebra homomorphism $\Psi: H^*_T(\pt) \cong
\Q[t_1,\ldots,t_n] \to H^*_T(\Flags(\C^n))$ by sending the generator
$t_i$ to the $i$-th equivariant Chern class $\TChFlag_i$, and define
$\widehat{\Psi}: H^*_T(\pt) \to H^*_T(\Hess(\mathsf{S},h))$ by
composing $\Psi$ with the natural restriction $H^*_T(\Flags(\C^n)) \to
H^*_T(\Hess(\mathsf{S},h))$. 
To discuss properties of $\Psi$ and $\widehat{\Psi}$,
it is useful to observe that the 
invariant subspace $H_T^*(\Hess(\mathsf{S},h))^{\Sn}$ of
$H_T^*(\Hess(\mathsf{S},h))$ (respectively $H^*(\Hess(\mathsf{S},h))^{\Sn}$ of
$H^*(\Hess(\mathsf{S},h))$) in fact forms a subring. 
\begin{lemma}\label{lemma:acts by ring hom}
The symmetric group $\Sn$ acts on $H^*_{\Tn}(\Hess(\mathsf{S},h))$ and
$H^*(\Hess(\mathsf{S},h))$ via
ring automorphisms, i.e. for $v \in \Sn$ and $\alpha, \beta \in
H^*_T(\Hess(\mathsf{S},h))$, we have $v\cdot(\alpha \beta) = (v\cdot \alpha)(v\cdot
\beta)$, and similarly for $H^*(\Hess(\mathsf{S},h))$. 
Moreover, the identity elements of the rings $H_T^*(\Hess(\mathsf{S},h))$ and $H^*(\Hess(\mathsf{S},h))$ are $\Sn$-invariant.
\end{lemma}

\begin{proof}
Straightforward from the definition of the $\Sn$-action on $H_T^*(\Hess(\mathsf{S},h))$ and $H^*(\Hess(\mathsf{S},h))$ given at \eqref{def of Tymoczko rep} and \eqref{eq:HT(Hess) surjects to H(Hess)}, respectively.
\end{proof}

By Lemma~\ref{lemma:chern classes invariant}, the images
of $\Psi$ and $\widehat{\Psi}$ are contained in the $\Sn$-invariant
subrings of $H^*_T(\Flags(\C^n))$ and $H^*_T(\Hess(\mathsf{S},h))$
respectively. In fact, we can say more. 

\begin{proposition}\label{prop:invariant subring is polynomial ring}
  The
    $\Q$-algebra homomorphisms $\Psi$ and $\widehat{\Psi}$ induce
    isomorphisms from $H^*_T(\pt) \cong \Q[t_1,\ldots,t_n]$ to the
    subrings $H^*_T(\Flags(\C^n))^{\Sn}$ and
    $H^*_T(\Hess(\mathsf{S},h))^{\Sn}$ of $\Sn$-invariants, respectively. In
    particular, 
the two subrings
$H^*_T(\Flags(\C^n))^{\Sn}$ and $H^*_T(\Hess(\mathsf{S},h))^{\Sn}$ are
isomorphic. 
\end{proposition}

\begin{proof}
The proof we give below applies to $\Hess(\mathsf{S},h)$ for any $h
\in H_n$, thus includes $\Flags(\C^n)$ as a special case. In
particular, showing that $\widehat{\Psi}: H^*_T(\pt) \to
H^*_T(\Hess(\mathsf{S},h))^{\Sn}$ is an isomorphism for any $h \in
H_n$ implies all the claims made in the proposition. 

We first show injectivity of $\widehat{\Psi}$, for which it is useful to consider the
projection $\pi_e:H^*_T(\Hess(\mathsf{S},h)^T) \cong \bigoplus_{w \in \Sn}
H^*_T(\pt) \rightarrow H^*_T(\pt)$ of $\bigoplus_{w \in \Sn} H^*_T(\pt)$
to the component corresponding to the identity element $e \in
\Sn$. Then, from the above computation $\TChSemi_i(w) = 
t_{w(i)}$ for $w \in \Sn$, it follows that $\TChSemi_i(e) = t_i$ for
all $i$ and hence the composition $H^*_T(\pt) \stackrel{\widehat{\Psi}}{\rightarrow}
H^*_T(\Hess(\mathsf{S},h))^{\Sn} \stackrel{\pi_e}{\rightarrow} H^*_T(\pt)$ is the
identity map. In particular, $\widehat{\Psi}$ must be injective. 

Next, observe that an $|\Sn|$-tuple $\alpha = (\alpha(w))_{w \in \Sn}$
is $\Sn$-invariant if and only if $\alpha(w) = w \cdot\alpha(e)$ for all $w \in \Sn$.
Since the
classes $\TChSemi_i$ satisfy both $\TChSemi_i(e)= t_i$ and
$\TChSemi_i(w) = t_{w(i)} = w\cdot t_i = w\cdot \TChSemi_i(e)$ for all $w \in
\Sn$, it follows that any $\Sn$-invariant $|\Sn|$-tuple
$\alpha=(\alpha(w))_{w \in \Sn}$ can be written as a polynomial in the
$\TChSemi_i$: namely, if $\alpha(e) = F(t_1, \ldots, t_n) \in
H^*_T(\pt) \cong \Q[t_1,\ldots,t_n]$, then $\alpha = F(\TChSemi_1,
\ldots, \TChSemi_n)$. In particular, $\widehat{\Psi}$ is surjective,
as desired. 
\end{proof} 

Our second goal for this section is
  to show that there exists an $\Sn$-invariant and non-degenerate
  pairing on the ordinary cohomology groups of $\Hess(\mathsf{S},h)$
  of complementary degree. 
 This pairing is straightforward in
  the sense that it is essentially the usual Poincar\'e duality pairing,
  although care is needed since our variety
  $\Hess(\mathsf{S},h)$ need not be connected (it is, however,
  pure-dimensional \cite{ma-pr-sh}); indeed, it is not hard to see
  that $\Hess(\mathsf{S},h)$ is disconnected if and only if $h(r)=r$
  for some $r\in[n]$ (see \cite{ma-pr-sh, teff11}).
  To see that the pairing 
is compatible with the $\Sn$-action, we 
  first work in $T$-equivariant cohomology and then deduce the desired
  results in ordinary cohomology. 

We need some terminology. Recall from~\eqref{eq:dim of HessS} that
  $d=\dim_{\C}\Hess(\mathsf{S},h)$. Recall
  also that the collapsing map $\mathrm{pr}: \Hess(\mathsf{S},h) \rightarrow \pt$
induces a map 
\begin{align*}
\mathrm{pr}^T_! \colon H^*_T(\Hess(\mathsf{S},h)) \rightarrow
  H^{*-2d}_T(\pt) = H^{*-2d}(BT)
\end{align*}
often called the ``equivariant integral'' or
``equivariant Gysin map''. The equivariant integral is well-known to be an $H^*_T(\pt)$-module
homomorphism. Moreover, by the famous
Atiyah-Bott-Berline-Vergne formula \cite{at-bo, be-ve} we may compute the
equivariant integral by fixed point data as follows: 
\begin{equation}
  \label{eq:ABBV}
  \mathrm{pr}^T_!(\alpha)= \sum_{w \in \Sn} \frac{\alpha(w)}{e_w}
\end{equation}
where $\alpha(w)$ denotes the restriction of $\alpha$
  to the fixed point $w$, $e_w$ denotes the $T$-equivariant Euler
  class of the normal bundle to the fixed point $w$ in
  $\Hess(\mathsf{S},h)$, and we have used our 
fixed 
  identification $\Hess(\mathsf{S},h)^T$ with $\Sn$.  
Finally we recall that the equivariant and ordinary Gysin maps commute with forgetful maps;
\begin{equation}\label{eq:Gysin maps diagram}
\begin{split}
\xymatrix{
    H^*_{\Tn}(\Hess(\mathsf{S},h)) \ar[r]^{\ \ \ \mathrm{pr}^T_!}
    \ar[d] &  H^{*-2d}_{\Tn}(\pt) \ar[d] \\
   H^*(\Hess(\mathsf{S},h)) \ar[r]^{\ \ \ \mathrm{pr}_!} & H^{*-2d}(\pt). 
}
\end{split}
\end{equation}

All the cohomology groups in the
  diagram~\eqref{eq:Gysin maps diagram} are equipped with
  $\Sn$-actions, where the $\Sn$-action on the ordinary cohomology
  $H^*(\pt) \cong \Q$ of a point is induced from that on $H^*_T(\pt)$
  by the isomorphism $H^*(\pt) \cong H^*(BT)/(t_1, \ldots, t_n) \cong
  \Q[t_1,\ldots, t_n]/(t_1,\ldots, t_n) \cong \Q$. In particular, the
  forgetful map $H^*_T(\pt) \to H^*(\pt)$ is $\Sn$-equivariant by
  definition and the $\Sn$-action on $H^*(\pt)$ is trivial. We record
  the following. 

\begin{lemma}\label{lemma:Gysin is equivariant}
The ordinary Gysin map $\mathrm{pr}_!$ in~\eqref{eq:Gysin maps diagram} is $\Sn$-equivariant. 
\end{lemma} 

\begin{proof} 
The definition given above of the $\Sn$-action on $H^*(\pt) \cong \Q$
implies that the right vertical arrow in~\eqref{eq:Gysin maps diagram}
is $\Sn$-equivariant, and we saw in Lemma~\ref{lemma:Flags and Hess equivariant} that the left vertical arrow in~\eqref{eq:Gysin maps
  diagram} is $\Sn$-equivariant. Recalling also that the left vertical
arrow is surjective (see e.g. \eqref{eq:HT(Hess) surjects to
  H(Hess)}), in order to prove the lemma, it therefore suffices to show that the
top horizontal arrow $\mathrm{pr}^T_!$ is $\Sn$-equivariant. Thus we
wish to show 
\[
\mathrm{pr}^T_!(v \cdot \alpha)  
= v \cdot \mathrm{pr}^T_!(\alpha)
\]
for $v \in \Sn$ and $\alpha \in H^*_T(\Hess(\mathsf{S},h))$. Before
proceeding it is useful to observe that the $T$-equivariant Euler class
$e_w$ of $\Hess(\mathsf{S},h)$ at $w \in \Sn$ is 
\begin{align*}
e_w = \prod_{j < i \leq h(j)} (t_{w(j)} - t_{w(i)}) \in H^*_T(\pt)
\end{align*}
(e.g. \cite{ma-pr-sh})
where we have written $w = (w(1) \ w(2) \ \cdots \ w(n)) \in \Sn$ in
one-line notation. 
In particular, since $vw = (vw(1) \ vw(2) \ \cdots \ vw(n))$ in one-line
notation, we conclude from the above that 
\[
v \cdot e_w = v \cdot \prod_{j<i\leq h(j)} (t_{w(j)} -t_{w(i)}) =
\prod_{j<i \leq h(j)} (t_{vw(j)} - t_{vw(i)}) = e_{vw}
\]
for any $v, w \in \Sn$. Now, using the Atiyah-Bott-Berline-Vergne
formula~\eqref{eq:ABBV}, we have 
\[
\begin{split}
\mathrm{pr}^T_!(v \cdot \alpha)
=\sum_{w\in \Sn}\frac{(v\cdot \alpha)(w)}{e_w}=\sum_{w\in \Sn}\frac{v\cdot \alpha({v^{-1}w)}}{e_w}=\sum_{u\in \Sn}\frac{v\cdot \alpha(u)}{e_{vu}}=\sum_{u\in \Sn}\frac{v\cdot \alpha(u)}{v\cdot e_{u}}=v\cdot \mathrm{pr}^T_!(\alpha)
\end{split}
\]
where the second equality follows from the definition \eqref{def of Tymoczko rep} of the $\Sn$-representation on $H_T^*(\Hess(\mathsf{S},h))$ and the fourth equality follows from the above observation for $v\cdot e_{w}$.
This proves the lemma. 
\end{proof} 

We now define a pairing $\langle \cdot, \cdot \rangle$ on the ordinary
cohomology $H^*(\Hess(\mathsf{S},h))$ as follows: for $0 \leq k \leq
d$, we define 
\begin{align}\label{eq:definition of pairing} 
\langle\cdot ,\cdot \rangle\colon 
H^{2k}(\Hess(\mathsf{S},h))\times
H^{2\dimX-2k}(\Hess(\mathsf{S},h))
\rightarrow 
\Q\cong H^0(\pt)
\quad ; \quad 
(\alpha,\beta) \mapsto \langle \alpha, \beta \rangle :=\mathrm{pr}_!(\alpha\beta).
\end{align}

\begin{proposition}\label{prop:nondegenerate and invariant}
The pairing $\langle \cdot, \cdot \rangle$ defined
in~\eqref{eq:definition of pairing} is non-degenerate and
$\Sn$-invariant. 
\end{proposition}

\begin{proof}
We begin with $\Sn$-invariance. Let $\alpha \in
  H^{2k}(\Hess(\mathsf{S},h))$ and $\beta \in
  H^{2d-2k}(\Hess(\mathsf{S},h))$ and $v \in \Sn$. Then 
\[
\langle v \cdot \alpha, v \cdot \beta \rangle = \mathrm{pr}_!((v \cdot
\alpha)(v \cdot \beta)) = \mathrm{pr}_!(v \cdot (\alpha \beta)) = v
\cdot \mathrm{pr}_!(\alpha \beta) = \mathrm{pr}_!(\alpha \beta) =
\langle \alpha, \beta \rangle 
\]
where the second equality uses Lemma~\ref{lemma:acts by ring hom}, the
third uses Lemma~\ref{lemma:Gysin is equivariant}, and the fourth
equality is because the $\Sn$-action on $H^*(\pt)\cong \Q$ is trivial,
as observed above. 

Next we claim that the pairing $\langle \cdot, \cdot
  \rangle$ is non-degenerate. This is an elementary argument which is
  clearer when stated more generally. It is useful to recall that for
  a disconnected complex manifold $X = \sqcup_{a \in \mathcal{S}} X_a$ with
  connected components $X_a$ each of real dimension $2d$, the
  cohomology ring $H^*(X)$ is a direct sum $\bigoplus_{a \in \mathcal{S}}
  H^*(X_a)$ (in particular, the cup product among different components vanishes) and the Gysin map 
  is simply the sum of the
  individual Gysin maps associated to the projections $\mathrm{pr}_a:
  X_a \to \pt$, i.e. $\mathrm{pr}_! = \sum_{a \in \mathcal{S}}
  (\mathrm{pr}_a)_!$. Thus it suffices to show that the given pairing
  is non-degenerate when restricted to the $a$-th component. But on
  each such component $X_a$, the Gysin map is given by capping with
  the fundamental homology class $[X_a] \in H_{2d}(X_a)$ and the non-degeneracy
  becomes the usual statement of Poincar\'e duality. Applying this
  argument to the case $X = \Hess(\mathsf{S}, h)$ yields the desired
  result. 
\end{proof}

Finally, we prove a fact which we use in the next section. 

\begin{lemma}\label{dimension is 1} 
$\dim_{\Q}H^{0}(\Hess(\mathsf{S},h))^{\Sn}=\dim_{\Q}H^{2\dimX}(\Hess(\mathsf{S},h))^{\Sn}=1$
\end{lemma} 

\begin{proof}
We have seen in Proposition~\ref{prop:nondegenerate and invariant}
that the pairing~\eqref{eq:definition of pairing} is non-degenerate
and $\Sn$-invariant, so it follows that
$H^0(\Hess(\mathsf{S},h))$ and
$H^{2\dimX}(\Hess(\mathsf{S},h))$ are dual representations. 
This implies that 
$H^0(\Hess(\mathsf{S},h))^{\Sn} \cong
H^{2\dimX}(\Hess(\mathsf{S},h))^{\Sn}$.
Now from the 
GKM description of $H^\ast_T(\Hess(\mathsf{S},h))$ in~\eqref{GKM for
  X(h)} and the explicit formula for Tymoczko's $\Sn$-action, it is
not difficult to see directly that $H^0_T(\Hess(\mathsf{S},h))^{\Sn}$ is $\Q$-spanned
by the identity element (whose component at each fixed point $w$ is
$1$); from this it also follows that $H^0(\Hess(\mathsf{S},h))^{\Sn}$ is $\Q$-spanned
by the identity element, so 
$\dim_\Q
H^0(\Hess(\mathsf{S},h))^{\Sn} = 1$. By the above, this in turn implies $\dim_\Q
H^{2\dimX}(\Hess(\mathsf{S},h))^{\Sn} = 1$, as desired.  
\end{proof}

\bigskip
\section{Proof of Theorem B}\label{sec:definition Abe map}  
In this section we prove Theorem B.
As a first step, we prove the
following.

\begin{proposition}\label{prop:Abe welldefined and surjective} 
There exists a well-defined homomorphism of graded
  $\Q$-algebras 
  $\Abe\colon H^*(\Hess(\mathsf{N},h)) \to H^*(\Hess(\mathsf{S},h))$
  making the diagram 
\vspace{20pt}
\begin{align}\label{eq:Abe map take 2}
 \ 
\end{align}
\vspace{-50pt}
\begin{center}
\begin{picture}(160,50)
   \put(10,35){$H^*(\Flags(\C^n))$}
   \put(77,35){$\overrightarrow{\qquad}$}
   \put(50,18){\rotatebox[origin=c]{-45}{$\overrightarrow{\qquad\ }$}}
   \put(50,18.7){\rotatebox[origin=c]{-45}{$\overrightarrow{\hspace{22pt}}$}}
   \put(110,18){\rotatebox[origin=c]{45}{$\overrightarrow{\qquad\ }$}}
   \put(100,35){$H^*(\Hess(\mathsf{S},h))$}
   \put(122,15){$\footnotesize{\text{$\Abe$}}$}
   \put(55,0){$H^*(\Hess(\mathsf{N},h)).$}
\end{picture} 
\end{center}
\vspace{5pt}
commute, where the other two maps are the induced homomorphisms.
Moreover, the
  image of $\Abe$ lies in $H^*(\Hess(\mathsf{S},h))^{\Sn}$, and when the target
  is restricted to $H^*(\Hess(\mathsf{S},h))^{\Sn}$, then 
\begin{equation*}
\Abe\colon H^*(\Hess(\mathsf{N},h))\to H^*(\Hess(\mathsf{S},h))^{\Sn} 
\end{equation*}
is surjective. 
\end{proposition} 

To prove
 Proposition~\ref{prop:Abe welldefined and surjective} we (once again)
  first work with the equivariant cohomology ring
  $H^*_T(\Hess(\mathsf{S},h))$; we also capitalize on our explicit
  presentation of $H^*(\Hess(\mathsf{N},h))$ obtained in Theorem A. 
 More precisely, recall that the
  ordinary cohomology ring $H^*(\Flags(\C^n))$ is generated by the
  first Chern classes $\ChFlag_i$ of the tautological line bundles
  $E_i/E_{i-1}$ as described in \eqref{eq:cohofl}, and by Theorem A, the map
  $H^*(\Flags(\C^n))\to H^*(\Hess(\mathsf{N},h))$, induced by the
  inclusion $\Hess(\mathsf{N},h) \into \Flags(\C^n)$, is
  surjective.
  Recall from \eqref{eq:def of ord Ch for flag} that 
\[
\ChNil_i \in H^2(\Hess(\mathsf{N},h))
\]
is the image of $\ChFlag_i$ (see \eqref{def of ch in flag}). Then Theorem A shows
that $H^*(\Hess(\mathsf{N},h))$ is generated by the $\ChNil_i$, and
that the map sending the polynomial variable $x_i$ to $\ChNil_i$ gives
an isomorphism
\[
H^*(\Hess(\mathsf{N},h)) \cong \Q[x_1,\ldots,
x_n]/(\cf_{h(j),j}(x_1,\ldots,x_n) \mid 1 \leq j \leq n).
\]
Now denote by 
\begin{align*}
\ChSemi_i\in H^2(\Hess(\mathsf{S},h))
\end{align*} 
the image of $\ChFlag_i$ under the restriction map
$H^*(\Flags(\C^n))\rightarrow H^*(\Hess(\mathsf{S},h))$ (whereas the
corresponding $T$-equivariant Chern class $\TChSemi_i\in
H_T^2(\Hess(\mathsf{S},h))$ was defined in \eqref{Chern classes for
  semisimple}).  That is, $\ChSemi_i$ is the first Chern class of the
tautological line bundle over $\Flags(\C^n)$ restricted to
$\Hess(\mathsf{S},h)$. In order to show that there
exists a ring homomorphism $\Abe\colon H^*(\Hess(\mathsf{N},h)) \to
H^*(\Hess(\mathsf{S},h))$ making~\eqref{eq:Abe map take 2} commute, from the
above discussion it follows that it suffices to show that the images
$\ChSemi_i$ of the
$\ChFlag_i$ in
$H^*(\Hess(\mathsf{S},h))$ also satisfy the relations specified by the
$\{\cf_{h(j),j}\}_{1 \leq j \leq n}$, i.e. that
\begin{equation}\label{eq:yj relations}
  \cf_{h(j),j}(\ChSemi_1, \ldots, \ChSemi_n) = 0 \in H^*(\Hess(\mathsf{S},h))
  \textup{ for all } 1 \leq j \leq n.
\end{equation}
In order to prove~\eqref{eq:yj relations}, we will first
work in the equivariant cohomology ring
$H^*_T(\Hess(\mathsf{S},h))$. Specifically, recall from \eqref{def of T eq ch in flag} that $\TChFlag_i$
is the $T$-equivariant first Chern class of the tautological line bundle $E_i/E_{i-1}$ in $H^*_T(\Flags(\C^n))$, so that
$\TChFlag_i$ maps to $\ChFlag_i$ under the forgetful map
$H^*_T(\Flags(\C^n)) \to H^*(\Flags(\C^n))$. Similarly
$\TChSemi_i$ is the image of $\TChFlag_i$ in
$H^*_T(\Hess(\mathsf{S},h))$ as defined in \eqref{Chern classes for semisimple}.  
Recall also that the kernel of the
forgetful map $H^*_T(\Hess(\mathsf{S},h)) \to
H^*(\Hess(\mathsf{S},h))$ is the ideal $(t_1,\ldots,t_n) \subset 
H^*_T(\Hess(\mathsf{S},h))$ generated by the classes $t_i
\in H^*_T(\Hess(\mathsf{S},h))$. Thus,
in order to show the vanishing relations~\eqref{eq:yj relations}
it suffices to show that 
\begin{equation*}
  \cf_{h(j),j}(\TChSemi_1, \ldots, \TChSemi_n) \in (t_1, \ldots,
  t_n) \subset  H^*_T(\Hess(\mathsf{S},h)) \textup{ for all } 1 \leq
  j \leq n.
\end{equation*}
This is precisely the goal of the next two lemmas. 

We first define some classes in
  $H^*_T(\Hess(\mathsf{S},h))$.  
Fix $j,k$ with $j, k \in [n]$. For each $w \in
\Sn$, we define a polynomial $\gjk(w) \in \Q[t_1,\ldots,t_n]$ by 
\begin{equation}\label{eq:def gjk}
\gjk(w):=
\begin{cases} \prod_{\ell=j+1}^{h(j)}(t_k-t_{w(\ell)}) &\quad\text{if $k\in \{w(1),\dots,w(j)\}$}\\
0 &\quad \text{otherwise,}
\end{cases}
\end{equation}
where we take the convention
$\prod_{\ell=j+1}^{j}(t_k-t_{w(\ell)})=1$. 
Thus, for fixed $j$ and $k$, the collection
  $\{\gjk(w)\}_{w \in \Sn}$ specifies an element of
  $H^*_T(\Hess(\mathsf{S},h)^T) \cong \bigoplus_{w \in \Sn}
  \Q[t_1,\ldots,t_n]$. 

\begin{lemma}\label{lem:gjk is GKM}
The polynomials $\{\gjk(w)\}_{w \in \Sn}$ in~\eqref{eq:def gjk} satisfy
  the GKM conditions~\eqref{GKM for X(h)} for $\Hess(\mathsf{S},h)$, and hence
  $\gjk := \{\gjk(w)\}_{w\in \Sn}$ is (the image under $\iota_3$ of) an equivariant cohomology class
  in $H^*_T(\Hess(\mathsf{S},h))$. 
\end{lemma}

\begin{proof}
Fix $j, k \in [n]$. For each $r\in[n]$, let us denote 
\[
\Sn^r:=\{w\in\Sn \mid w(r)=k\}
\]
which is the set of permutations having $k$ at the $r$-th position in the one-line notation. 
Then we have a decomposition $\Sn=\bigcup_{r=1}^n \Sn^r$, and the condition $k\in \{w(1),\dots,w(j)\}$ is equivalent to $w\in\bigcup_{r\leq j} \Sn^r$.
Recalling that the equivariant Chern class $\TChSemi_i$ satisfies $\TChSemi_i(w)=t_{w(i)}$ for $w\in\Sn$ by \eqref{localization of chern class for semisimple}, we can rewrite $\gjk$ as
\begin{align}\label{rewriting with chern class}
\gjk(w)
&=
\begin{cases} 
\prod_{\ell=j+1}^{h(j)}(t_k-\TChSemi_{\ell})(w) &\quad\text{if $w\in \bigcup_{r\leq j} \Sn^r$},\\
0 &\quad \text{otherwise}.
\end{cases}
\end{align}

We now check that the collection $\{\gjk(w)\}_{w\in\Sn}$ satisfies the GKM condition \eqref{GKM for X(h)} for $\Hess(\mathsf{S},h)$ by using \eqref{rewriting with chern class}.
Let $w,w' \in \Sn$ with $w'=w(a\ b)$ for some $a,b\in[n]$, and suppose that $w$ and $w'$ are connected by an edge of the GKM graph of $\Hess(\mathsf{S},h)$.
We show that the difference $\gjk(w)-\gjk(w')$ is divisible by $t_{w(a)}-t_{w(b)}$ by taking cases.

\medskip {\bf Case 1.} Suppose $w,w'\in\bigcup_{r\leq j} \Sn^r$.
Note that the collection $\{\prod_{\ell=j+1}^{h(j)}(t_k-\TChSemi_{\ell})(w)\}_{w\in\Sn}$ satisfies the GKM condition for $\Hess(\mathsf{S},h)$ since $\prod_{\ell=j+1}^{h(j)}(t_k-\TChSemi_{\ell})$ is an element of $H^*_T(\Hess(\mathsf{S},h))$ and we have the isomorphism \eqref{GKM for X(h)}.
Thus the claim holds in this case by \eqref{rewriting with chern class}. 

\medskip {\bf Case 2.}
Suppose $w, w'\in\bigcup_{r> j} \Sn^r$. In this case, the claim is immediate since $\gjk(w)=\gjk(w')=0$ by \eqref{rewriting with chern class}.

\medskip {\bf Case 3.} Suppose $w\in\bigcup_{r\leq j}
  \Sn^r$ and $w'\in\bigcup_{r> j} \Sn^r$. In this case, the condition
  $w'=w(a\ b)$ implies that we have $w(a)=k$ or $w(b)=k$. Without loss
  of generality, we may assume that $w(a)=k$. This means $a\leq j$
  because $w\in\bigcup_{r\leq j} \Sn^r$. Similarly since we have
  $w'(b)=k$ and $w'\in\bigcup_{r> j} \Sn^r$, it follows that $b>j$.
  Combining this with $a\leq j$, we obtain $a<b$. Hence the assumption
  that $w$ and $w'$ are connected by an edge of the GKM graph of
  $\Hess(\mathsf{S},h)$ implies that $b\leq h(a)$. In particular, we obtain
  $j+1\leq b\leq h(j)$ since $a\leq j$ implies $h(a)\leq h(j)$. Now
  from \eqref{rewriting with chern class} we have \begin{align*}
    \gjk(w)-\gjk(w') = \prod_{\ell=j+1}^{h(j)}(t_k-\TChSemi_{\ell})(w)
    - 0 = \prod_{\ell=j+1}^{h(j)}(t_k-t_{w(\ell)}). \end{align*} Since
  we have $w(a)=k$ and $j+1\leq b\leq h(j)$ as discussed above, the above
  product contains $t_k-t_{w(b)}=t_{w(a)}-t_{w(b)}$, and hence
  $\gjk(w)-\gjk(w')$ is divisible by $t_{w(a)}-t_{w(b)}$, as desired.
\end{proof}

Next, we explicitly show (using the classes $\gjk$
  introduced above) that the classes $\cf_{h(j),j}(\TChSemi_1,
  \ldots, \TChSemi_n)$ are contained in the ideal of $H^*_T(\Hess(\mathsf{S},h))$ generated by the
  $t_i$. 

\begin{lemma} \label{lem: f check is sent to zero}
Let $j \in [n]$. Then 
\begin{align*}
\cf_{h(j),j}(\TChSemi_1, \ldots, \TChSemi_n) =\sum_{k=1}^n t_k\gjk 
\quad\textup{ in } H^*_T(\Hess(\mathsf{S},h)).
\end{align*}
In particular, $\cf_{h(j),j}(\TChSemi_1, \ldots, \TChSemi_n)$
lies in the ideal 
$( t_1, \ldots, t_n ) \subset 
H^\ast_T(\Hess(\mathsf{S},h))$ for all $j \in [n]$. 
\end{lemma} 

\begin{proof}
Since the restriction map $H^*_T(\Hess(\mathsf{S},h))
  \stackrel{\iota_3}{\rightarrow} H^*_T(\Hess(\mathsf{S},h)^T)$ is injective, in order to prove
  the lemma it suffices to prove that for all $w \in \Sn$ we have 
  \begin{equation*}
    \cf_{h(j),j}(\TChSemi_1, \ldots,\TChSemi_n)(w) = \sum_{k=1}^n
    t_k \gjk(w) \in \Q[t_1,\ldots,t_n].
  \end{equation*}
Now recall that by definition, if $k \not \in \{w(1),\ldots, w(j)\}$ then $\gjk(w)=0$. Hence 
\[
t_k\gjk(w)=
\begin{cases} t_k\prod_{\ell=j+1}^{h(j)}(t_k-t_{w(\ell)}) &\quad\text{if $k\in \{w(1),\dots,w(j)\}$},\\
0 &\quad \text{otherwise. }
\end{cases}
\]
Thus if we take the sum of the $t_k\gjk(w)$ over $k=1,\cdots,n$, it in fact suffices to
take the sum only for $k=w(1),w(2),\dots,w(j)$.
Hence, we obtain by \eqref{formula for f check} that
\begin{align*}
\sum_{k=1}^n t_k\gjk(w)
= \sum_{k=1}^j \Big( t_{w(k)}\prod_{\ell=j+1}^{h(j)}(t_{w(k)}-t_{w(\ell)}) \Big)
= \cf_{h(j),j}(w)
\end{align*}
as desired. 
\end{proof} 

From the above discussion and by Lemma~\ref{lem:gjk is GKM} and Lemma~\ref{lem: f check is sent to zero}, it is now clear that there exists a unique ring homomorphism
\[
\Abe\colon H^*(\Hess(\N, h))\to H^*(\Hess(\mathsf{S},h))
\]
which makes the diagram \eqref{eq:Abe map} in Theorem B commute.
We note that $\Abe$ maps $\ChNil_i$ to $\ChSemi_i$ for $i=1,\dots,n$.
We are now ready to prove Proposition~\ref{prop:Abe welldefined and surjective}. 

\begin{proof}[Proof of Proposition~\ref{prop:Abe welldefined and surjective}]
  First we claim that the image of $\Abe$ lies in the
  $\Sn$-invariants. From the definition of $\Abe$, it is clear that
  its image 
  coincides with the image of the restriction map
  $H^*(\Flags(\C^n)) \to H^*(\Hess(\mathsf{S},h))$. Hence the claim
  follows from the facts that the $\Sn$-representation on
  $H^*(\Flags(\C^n))$ is trivial (Lemma~\ref{lem:rep on flag is
    trivial}) and that the bottom map in \eqref{eq:Sn equivariant
    diagram} is a homomorphism of $\Sn$-representations. Now 
  consider the map with restricted target
\begin{equation*}
\Abe\colon H^*(\Hess(\N, h))\to H^*(\Hess(\mathsf{S},h))^{\Sn}. 
\end{equation*}
We wish to we show that this is surjective. 
Recalling the commutative diagram \eqref{eq:Abe map take 2} and Lemma~\ref{lem:rep on flag is trivial}, it suffices to show that the map 
\begin{equation} \label{eq: restriction map on invariant}
H^*(\Flags(\C^n))^{\Sn} \to H^*(\Hess(\mathsf{S},h))^{\Sn}
\end{equation} 
is surjective. For this, we know from Proposition~\ref{prop:invariant
  subring is polynomial ring} that the map
$H_T^*(\Flags(\C^n))^{\Sn}\rightarrow
H_T^*(\Hess(\mathsf{S},h))^{\Sn}$ on equivariant cohomology is an
isomorphism. We also know that the forgetful map
$H_T^*(\Hess(\mathsf{S},h))\rightarrow H^*(\Hess(\mathsf{S},h))$ is
surjective, and hence so is its restriction
$H_T^*(\Hess(\mathsf{S},h))^{\Sn}\rightarrow
H^*(\Hess(\mathsf{S},h))^{\Sn}$. Indeed, for any $\Sn$-invariant
element $x\in H^*(\Hess(\mathsf{S},h))^{\Sn}$ we can take a lift
$\widetilde{x}\in H_T^*(\Hess(\mathsf{S},h))$; averaging
$\widetilde{x}$ over $\Sn$ yields an $\Sn$-invariant element which
maps to $x$. Now, we see that \eqref{eq: restriction map on invariant}
is surjective by taking $\Sn$-invariant subrings in the commutative
diagram in Lemma~\ref{lemma:Flags and Hess equivariant}, completing
the proof. 
\end{proof}

It remains to show that $\Abe: H^\ast(\Hess(\mathsf{N},h)) \to
H^\ast(\Hess(\mathsf{S},h))^{\Sn}$ is also injective (and hence an
isomorphism). 
We achieve this by employing some basic commutative algebra facts
concerning Poincar\'e duality algebras. 
The basic idea, encapsulated in
  Lemma~\ref{lemma:PDAs imply isom} below, is the simple fact that if $\varphi: R \to
  S$ is a surjective graded algebra homomorphism from a
  Poincar\'e duality algebra and $\varphi$ induces
  an isomorphism between $R^{\mathrm{max}}$ and $S^{\mathrm{max}}$
  (where $R^{\mathrm{max}}$ and $S^{\mathrm{max}}$ denote the
  highest-degree component of $R$ and $S$ respectively), then
  $\varphi$ must be an isomorphism. Since we have already shown above
  that $\Abe\colon
  H^*(\Hess(\mathsf{N},h)) \to H^*(\Hess(\mathsf{S},h))^{\Sn}$ is
  surjective, Lemma~\ref{lemma:PDAs imply isom} essentially reduces
  the question to showing that the domain is a Poincar\'e duality
  algebra and that $\Abe$ induces an isomorphism on the top
  degree.

There exist different definitions of Poincar\'e duality algebras in
  the literature, but we use the following. 

\begin{definition}
\label{def of PDA}
  Suppose that $R = \bigoplus_{i=0}^\dimX R_i$ is a graded algebra
  over some fixed field $\mathfrak{k}$, finite-dimensional over
  $\mathfrak{k}$. Suppose $R_0 \cong R_{\dimX} \cong \mathfrak{k}$.  We say $R$ is a
  \textbf{Poincar\'e duality algebra \emph{(}PDA\emph{)}} if the bilinear pairing 
$
R_i\times R_{\dimX-i} \rightarrow R_{\dimX}
$
defined by the multiplication in $R$ is non-degenerate for all 
$i=0,\dots,\dimX$.
\end{definition}

The following straightforward lemma 
is the essence of our argument. 

\begin{lemma}\label{lemma:PDAs imply isom}
  Let $R=\bigoplus_{i=0}^d R_i$ and $R'=\bigoplus_{i=0}^{d} R'_i$ be graded algebras such that $R_d\neq\{0\}$ and $R'_{d}\neq\{0\}$. 
  Let
  $\varphi:R\rightarrow R'$ be a graded ring homomorphism.
  Suppose that $R$ is a Poincar\'e duality algebra.  If
  $\varphi$ is surjective and it restricts to an isomorphism between
  $R_d$ to $R'_d$, then $\varphi$ is an isomorphism.
\end{lemma}

It remains to show that our rings $H^*(\Hess(\mathsf{N},h))$ and
$H^*(\Hess(\mathsf{S},h))^{\Sn}$ and the map $\Abe$ satisfies the
conditions of Lemma~\ref{lemma:PDAs imply isom}.  In particular, we
wish to show that the ring $H^*(\Hess(\mathsf{N},h)) \cong
\Q[x_1,\ldots, x_n]/(\cf_{h(j),j} \mid j \in [n] )$ is a PDA. Recall
that we showed in Lemma~\ref{lemm:4-2} that our generators
$\cf_{h(1),1}, \cf_{h(2),2}, \cdots, \cf_{h(n),n}$ form a regular
sequence of length equal to the number of variables in the polynomial
ring.  As is well-known in commutative algebra, these
  facts together imply that the cohomology ring $H^*(\Hess(\N,h))$ is
  a complete intersection, and in particular, is a PDA (see e.g. \cite[Theorem
  21.2.]{matsumura}, \cite[Theorem 21.3]{matsumura}, and \cite[Theorem
  2.79]{hari13}).
We record this in the following
proposition. 

\begin{proposition}\label{PD for Hess(h)}
Let $h \in H_n$ be a Hessenberg function and let $\Hess(\N,h)$ denote the
associated regular nilpotent Hessenberg variety. Then, with respect to
the usual grading and multiplication in cohomology, the 
ordinary cohomology ring $H^*(\Hess(\N, h))$ is a Poincar\'e duality
algebra. 
\end{proposition}

Now we can prove Theorem B.
\begin{proof}[Proof of Theorem B]
We apply Lemma~\ref{lemma:PDAs imply isom} to our $\Q$-algebra homomorphism
$\Abe$ in~\eqref{eq:Abe map take 2}. 
We already know that the map is surjective by Proposition~\ref{prop:Abe welldefined and surjective} and that the domain of this map is Poincar\'e duality algebra from Proposition~\ref{PD for Hess(h)}.
Also, since we know that 
\[
  \dim_{\Q}H^{2\dimX}(\Hess(\mathsf{N},h))=\dim_{\Q} H^{2\dimX}(\Hess(\mathsf{S},h))^{\Sn}=1
\]
from the computation of the Hilbert polynomial of
$H^\ast(\Hess(\mathsf{N},h))$ and Lemma~\ref{dimension is 1}, the surjectivity of $\Abe$ shows that the map $\Abe$ restricted on degree $2\dimX$ is an isomorphism.
Hence, by Lemma~\ref{lemma:PDAs imply isom} the $\Q$-algebra
homomorphism $\Abe:H^*(\Hess(\mathsf{N},h))\rightarrow
H^*(\Hess(\mathsf{S},h))^{\Sn}$ is an isomorphism, as desired. 
\end{proof}

\bigskip
\section{Connection to the Shareshian-Wachs conjecture} 
\label{sect: Shareshian-Wachs conjecture}
As mentioned in the Introduction, our work on Hessenberg varieties
turns out to be related to combinatorics through the Shareshian-Wachs
conjecture. Although this conjecture has recently been proved by
Brosnan and Chow, the approach taken in this paper offers a different
perspective on the problem and, as we noted in the Introduction, our
Theorem B proves (at least, for the coefficient of the Schur function
$s_n(x)$ corresponding to the trivial representation) a statement
which is strictly stronger than the corresponding statement in
\cite{br-ch}. For this reason, in this section we briefly review the
context, give the precise statement of the Shareshian-Wachs
conjecture, and explain the relationship between the conjecture and
our Theorem B. 

In \cite{sh-wa11, sh-wa14}, the Shareshian-Wachs conjecture is
formulated in terms of natural unit interval orders and
incomparability graphs, but for the purposes of this paper it is
convenient to rephrase it more directly in terms of Hessenberg
functions. Fix a Hessenberg function $h:[n]\rightarrow[n]$. 
Let $P(h)$ denote the partially ordered set whose underlying set is
$[n]$ and with partial order defined by $j <_P i$ if and only if $h(j)
< i$ \cite[Section 4]{sh-wa14}. The following characterizes natural
unit interval orders in terms of such posets.

\begin{proposition} \emph{(}\cite[Proposition 4.1]{sh-wa14}\emph{)} 
Let $P$ be a poset on $[n]$. Then $P$ is a natural unit interval order
if and only if $P = P(h)$ for some Hessenberg function $h$. 
\end{proposition}

Furthermore, the incomparability graph of a
poset $P$ as defined in \cite[Section 1]{sh-wa14} has as its vertices the elements of $P$, and an edge between
two elements precisely when the two elements are incomparable with
respect to the given partial order. From the definition of $P(h)$
above, it is then immediate that the incomparability graph $G$ of
$P(h)$ is the graph 
with vertex set $[n]$ and with edges $E$ given by 
\begin{align*}
E :=\{\{i,j\} \mid i, j \in [n], j < i \leq h(j) \}
\end{align*}
i.e. there is an edge between $i$ and $j$ (where without loss of
generality $i>j$) exactly when $i \leq h(j)$. 
For example, if $h=(1,2,\ldots,n)$, then
evidently $E$ is empty, and the corresponding incomparability graph $G$ has $n$ vertices and no
edges.
At the other extreme, if $h=(n,n,\ldots,n)$, then its comparability graph $G$ is the complete
graph on $n$ vertices. 

Next, let $x_1, x_2, x_3, \dots$ be a countably infinite set of
variables. Denoting by $\mathbb{P}$ the set of positive integers, 
we call a map $\kappa:V=[n]\rightarrow \mathbb{P}$ 
a \textit{coloring of $G$} if $\kappa$ satisfies
$\kappa(i)\neq\kappa(j)$ for any $\{i,j\}\in E$, i.e. if $\kappa(i)$
is the ``color'' of the vertex $i$, then we require that adjacent vertices must be 
colored differently. Let $C(G)$ denote the set of all colorings of
$G$, and let ${\bf x}_{\kappa}$ denote the monomial
$\prod_{i\in[n]}x_{\kappa(i)}$ for any coloring $\kappa$. 
We also define 
\[
\text{asc}(\kappa) := |\{\{i,j\}\in E \mid j<i \textup{ and } \kappa(j)<\kappa(i)\}|
\]
Then the \textbf{chromatic quasisymmetric function of $G$} is defined
to be 
\[
X_G({\bf x},t) := \sum_{\kappa\in C(G)} t^{\text{asc}(\kappa)}{\bf x}_{\kappa}.
\]
In our situation, where $G = \mathrm{inc}(P(h))$ is the incomparability graph
of a natural unit interval order $P(h)$, it is known that when we consider $X_G({\bf x},t)$
as a polynomial in $t$, each coefficient is an element of the
algebra $\Lambda_{\Z}$ of symmetric functions in the variables $\bf x$ \cite[Theorem 4.5]{sh-wa14}. That is, we have
$X_G({\bf x},t)\in \Lambda_{\Z}[t]$. 
In the following example, for $i$ a positive integer,
we denote by $e_i({\bf x})$ the $i$-th elementary symmetric function in the variables $\textbf{x}$. 

Finally, following standard
notation in the theory of symmetric functions, we denote by $\omega$
the involution of $\Lambda_\Z$, the algebra of symmetric functions,
which exchanges the elementary basis $\{e_\lambda\}$ with the complete
homogeneous basis $\{h_\lambda\}$ (as $\lambda$ ranges over partitions)
\cite[Section 6]{fult97}. 
For our purposes it is useful to note that, for $\omega$ defined as above, we have $\omega(s_\lambda)=s_{\lambda^*}$, 
where $s_\lambda$ denotes the Schur function associated to a partition
$\lambda$ \cite[Section 6]{fult97} and $\lambda^*$ denotes
the partition conjugate to $\lambda$.
Based on the above discussion, the reader may easily check that the
formulation of the Shareshian-Wachs conjecture recorded below is
equivalent to that given in \cite[Conjecture 1.4]{sh-wa14}. 

\begin{conjecture}\label{SW conjecture}
Let $h: [n] \to [n]$ be a Hessenberg function, $P(h)$ its associated
poset and $G$ the incomparability graph of $P(h)$. Let
$X_G({\bf x}, t)$ denote the chromatic quasisymmetric function of
$G$, and let $\Hess(\mathsf{S},h)$ be the regular semisimple
Hessenberg variety associated to $h$. Then 
\begin{align}\label{SW conjecture 200}
 \omega X_G({\bf x},t) = \sum_{j=0}^{|E(G)|} \emph{ch} H^{2j}(\Hess(\mathsf{S},h))t^j
\end{align}
where $\emph{ch}$ denotes the Frobenius characteristic of Tymoczko's
$\Sn$-representation on $H^{2j}(\Hess(\mathsf{S},h))$. 
\end{conjecture}

Since~\eqref{SW conjecture 200} takes place within the ring of symmetric functions, 
expanding both sides in
terms of (the basis of) Schur functions $s_\lambda({\bf x})$, we may interpret~\eqref{SW conjecture 200} as 
the statement that the coefficient of $s_\lambda({\bf x})$ on both sides must be equal for each partition $\lambda$. 
In \cite[Theorem 6.9]{sh-wa14} Shareshian
and Wachs also obtain a closed formula for the coefficient of
$s_n({\bf x})$, i.e. the coefficient corresponding to the trivial
representation.

\begin{theorem} \emph{(}\cite[Theorem 6.9]{sh-wa14}\emph{)}
In the Schur basis expansion of $X_G({\bf x},t)$, the coefficient of $s_{1^n}({\bf x})$ is 
$\prod_{j=1}^{n}[h(j)-j+1]_t$
where $[i]_t = \frac{1-t^i}{1-t}$.
\end{theorem}

Finally, since $\omega s_{1^n}({\bf x})= s_n({\bf x})$ is the Frobenius
characteristic of the trivial representation and the polynomial
$\prod_{j=1}^{n}[h(j)-j+1]_t$ is exactly the Hilbert series $F(H^*(\Hess(\mathsf{N},h)),s)$ by Lemma~\ref{lemma:hesshil} (after replacing $s^2$ by $t$), 
it follows from Theorem B that Shareshian-Wachs
conjecture holds for the component of the trivial representation. We
record the following. 

\begin{corollaryB*} 
The coefficients of $s_n({\bf x})$ are the same on the both sides of \eqref{SW conjecture 200}.
\end{corollaryB*}

\bigskip
\section*{Appendix: Proof of Lemma~\ref{lemma:fij in terms of e and b equation (short)}}

\renewcommand{\thesection}{A}

  The purpose of this section is to give a proof of
  Lemma~\ref{lemma:fij in terms of e and b equation (short)} which we
  postponed to do that in Section~\ref{sec:fproperty}. We divide the
  lemma into 
  two claims, Lemma~\ref{app lemma:fij in terms of e
    and b symmetric} and Lemma~\ref{app lemma:fij in terms of e and b
    equation} below.

\begin{lemma} \label{app lemma:fij in terms of e and b symmetric}
For any pair $k,j \in [n]$ with $k\geq j$, the polynomial 
$b_{k,j}$ 
is symmetric in the variables $u_1, \ldots, u_j$.
\end{lemma}
\begin{proof}

We argue by induction on the indices $(k,j) \in [n]^2$ for $k \geq j$,
with respect to the partial order defined by: $(k', j') < (k,j)$ if
and only if $j'<j$, or $j'=j$ and $k'<k$. Note also that the claim of
the lemma clearly holds for any $b_{j,j}$ with $j \in [n]$, by its
definition~\eqref{eq:definition bjj}. Now we wish to show that the
claim holds for $b_{k+1,j}$ for $k \geq j$ where we may assume by
induction that the claim holds for $b_{k', j'}$ with
$(k',j')<(k+1,j)$. From the definition of $b_{k+1,j}$
in~\eqref{eq:definition bkj} it then follows that $b_{k+1,j}$ is
symmetric in the first $j-1$ variables $u_1, \ldots,
u_{j-1}$. Therefore, it now suffices to show that $b_{k+1,j}$ is also
symmetric in $u_{j-1}$ and $u_j$.

For $k=j$ for any $j$, we may
explicitly compute from \eqref{eq:definition bkj} and \eqref{eq:definition bjj} as follows:
\begin{align*}
    b_{j+1,j} 
 & = b_{j, j-1} - t b_{j-1,j-1} + u_j (b_{j,j} - b_{j-1,j-1})  \\
 & = b_{j,j-1} - t b_{j-1,j-1} + u_j (u_j -( j-1)t) \\
 & = (b_{j-1,j-2} - t b_{j-2,j-2} + u_{j-1}(u_{j-1}-(j-2)t)) \\
 &\qquad\quad
 - t (b_{j-2,
   j-2} + (u_{j-1}-(j-2)t)) + u_j (u_j-(j-1)t) \\
 & = b_{j-1, j-2} - 2 t b_{j-2, j-2} + ( u_{j-1}^2 + u_j^2 - (j-1)(u_{j-1}
 + u_j) t + (j-2)t^2 ). 
\end{align*}
Since $b_{j-1,j-2}$ and $b_{j-2,j-2}$ are functions in the variables $u_1,\ldots, u_{j-2}$, it follows from the explicit expression above that $b_{j+1,j}=b_{j+1,j}(u_1,\dots,u_j,t)$ is symmetric in $u_{j-1}$ and $u_j$. 

We now claim $b_{k+1,j}$ is symmetric in $u_{j-1}$ and $u_j$ for $j
\geq 2$ and $k>j$. 
Generalizing the argument for the case $k=j$ above, we may derive the following by 
repeated use of the inductive definition of the $b_{k,j}$: 
\begin{align*} 
b_{k+1,j}&=b_{k,j-1}-tb_{k-1,j-1}+u_j (b_{k,j}-b_{k-1,j-1}) \\ 
&=b_{k,j-1}-tb_{k-1,j-1} -u_j t b_{k-2,j-1} +u_j^2(b_{k-1,j}-b_{k-2,j-1}) \\ 
&=\dots \\ 
&=b_{k,j-1}-\sum_{q=j-1}^{k-1}u_j^{k-1-q}tb_{q,j-1}+u_j^{k+1-j}(b_{j,j}-b_{j-1,j-1}) \\ 
&=b_{k,j-1}-\sum_{q=j-1}^{k-1}u_j^{k-1-q}tb_{q,j-1}+u_j^{k+1-j}(u_j-(j-1)t).
\end{align*}
Using this expression 
several times, we may express $b_{k+1,j}$ explicitly in terms of functions $b_{\ell,
  j-2}$ and the variables $u_{j-1}$ and $u_j$: 
\begin{align*} 
b_{k+1,j} & = b_{k-1, j-2} - \sum_{r=j-2}^{k-2} u_{j-1}^{k-2-r}t b_{r,j-2} + u_{j-1}^{k+1-j}(u_{j-1}-(j-2)t)  \\
 & \quad\quad - u_j^{k-j} t b_{j-1,j-1} \\
 & \quad\quad - \sum_{q=j}^{k-1} u_j^{k-1-q}t \left( b_{q-1,j-2} - \sum_{r=j-2}^{q-2} u_{j-1}^{q-2-r} t b_{r,j-2} +
   u_{j-1}^{q+1-j} (u_{j-1}-(j-2)t) \right) \\
 & \quad\quad + u_j^{k+1-j}(u_j-(j-1)t)  \\
 & = b_{k-1, j-2} - \sum_{r=j-2}^{k-2} u_{j-1}^{k-2-r}t b_{r,j-2}
 + u_{j-1}^{k+1-j}(u_{j-1}-(j-2)t) \\
 & \quad\quad - u_j^{k-j} t (b_{j-2,j-2} + (u_{j-1}-(j-2)t)) \\
 & \quad\quad - \sum_{q=j}^{k-1} u_j^{k-1-q} t \left( b_{q-1,
     j-2} - \sum_{r=j-2}^{q-2} u_{j-1}^{q-2-r} t b_{r,j-2} +
   u_{j-1}^{q+1-j} (u_{j-1}-(j-2)t) \right) \\
 & \quad\quad + u_j^{k+1-j} (u_j-(j-1)t).  
\end{align*}
By exchanging the order of the sums with respect to $q$ and $r$, this is further equal to
\begin{align*}
 &b_{k-1, j-2} - \sum_{r=j-2}^{k-2} (u_{j-1}^{k-2-r} +
 u_j^{k-2-r}) t b_{r, j-2} + \sum_{r=j-2}^{k-3} \left( \sum_{q=r+2}^{k-1}
   u_{j-1}^{q-(2+r)} u_j^{k-1-q} \right) t^2 b_{r,j-2} \\
 & \quad\quad +  (u_{j-1}^{k+2-j} + u_j^{k+2-j}) - (j-1)(u_{j-1}^{k+1-j} +
 u_j^{k+1-j})t \\
 & \quad\quad + u_{j-1}^{k+1-j}t - \sum_{q=j-1}^{k-1} (u_{j-1} -(j-2)t)
 u_{j-1}^{q+1-j} u_j^{k-1-q}t. 
 \end{align*}
By separating the last summand, we obtain the equality
\begin{align*}
 b_{k+1,j}=&b_{k-1, j-2} - \sum_{r=j-2}^{k-2}  (u_{j-1}^{k-2-r}+
 u_j^{k-2-r}) t b_{r, j-2} + \sum_{r=j-2}^{k-3} \left( \sum_{q=r+2}^{k-1}
   u_{j-1}^{q-(2+r)} u_j^{(k-1)-q} \right) t^2 b_{r,j-2} \\
 & \quad\quad + (u_{j-1}^{k+2-j} + u_j^{k+2-j}) - (j-1)(u_{j-1}^{k+1-j} +
 u_j^{k+1-j})t \\
 & \quad\quad + \sum_{q=j-1}^{k-1} (j-2) u_{j-1}^{q-(j-1)}
 u_j^{(k-1)-q}t^2 - \sum_{q=j-1}^{k-2} u_{j-1}^{q-(j-1)+1} u_j^{(k-2)-q+1} t.
\end{align*}
Since $b_{k-1,j-2}$ and $b_{r,j-2}$ are functions 
in the variables $u_1, \ldots, u_{j-2}$, it can be seen from the final explicit expression above that $b_{k+1,j}=b_{k+1,j}(u_1, \ldots, u_{j},t)$
is symmetric in the $u_{j-1}$ and $u_j$, as desired. 
\end{proof}

We now prove the second claim of Lemma~\ref{lemma:fij in terms of e and b equation (short)}.
For this purpose, we make the substitution
\begin{align}\label{app eq:def of ur}
u_r = (w(r)-1)t \quad \textup{ for } r \in [n]. 
\end{align}

\begin{lemma} \label{app lemma:fij in terms of e and b equation}
Let $k \geq j$, $k, j \in [n]$. Let $b_{k,j}=b_{k,j}((w(1)-1)t, \ldots)
\in \Q[t]$ denote the polynomial $b_{k,j}$ defined
in~\eqref{eq:definition bjj} and~\eqref{eq:definition bkj} evaluated
at $u_r = (w(r)-1)t$ as in~\eqref{app eq:def of ur}. Then
for any pair $i,j\in[n]$ with $i\geq j$ we have
\begin{equation} \label{app eq:fij in terms of e and b}
f_{i,j}(w)=\sum_{k=j}^i(-1)^{i-k}e_{i-k}(w(j+1),\dots,w(i))t^{i-k}b_{k,j}
\quad \text{in } \Q[t]. 
\end{equation}
\end{lemma}

\begin{proof}
We prove the claim by induction on pairs $(i,j)$ with respect to
  the same partial order considered in the proof of Lemma~\ref{app lemma:fij in terms of e and b symmetric}. 
First suppose $j=1$. 
In this case
$b_{1,1}=u_1$ by definition and since $b_{*,0}=0$ the
equation~\eqref{eq:definition bkj} reduces to $b_{k+1,1}= u_1 b_{k,1}$, we obtain $b_{k,1}=u_1^k$. Thus, for $j=1$
  and any $i \geq j$, the claim~\eqref{app eq:fij in terms of e and b} is precisely the
assertion~\eqref{eq:fi1 in u} obtained in Section~\ref{sec:fproperty}. 

Now by induction suppose the equality~\eqref{app eq:fij in terms of e and b} holds for
$j-1$ and for any $i \geq j-1$. Moreover, for $i=j$ we have
$f_{j,j}(w) = b_{j,j}$ by definition of the $b_{j,j}$ so the assertion
also holds in this case. Also for the case $i=j+1$ we can compute
explicitly from \eqref{eq:fij at w} that the LHS of~\eqref{app eq:fij in terms of e and b} is 
\begin{align*}
f_{j+1,j}(w) & = f_{j, j-1}(w) + (w(j)-w(j+1)-1) t f_{j,j}(w) \\
 & = \sum_{k=j-1}^j (-1)^{j-k} e_{j-k}(w(j)) t^{j-k} b_{k,j-1} + (u_j -
 w(j+1)t) b_{j,j} \\
 & = - w(j) t b_{j-1,j-1} + b_{j,j-1}  + (u_j - w(j+1)t) 
 b_{j,j} \\
 & = - (u_j + t) b_{j-1, j-1}  + b_{j, j-1}  + (u_j - w(j+1)t) 
 b_{j,j} 
\end{align*}
where we have used the inductive hypothesis, \eqref{app eq:def of ur}, and $f_{j,j}(w)=b_{j,j}$.
The RHS of \eqref{app eq:fij in terms of e and b} can similarly be computed to be  
\begin{align*}
  \sum_{k=j}^{j+1} (-1)^{j+1-k} e_{j+1-k}(w(j+1)) t^{j+1-k} b_{k,j} 
 &= - w(j+1) t b_{j,j} + b_{j+1,j}  \\
 &= - w(j+1) t b_{j,j} + (b_{j, j-1} + u_j b_{j,j}  - (u_j+t)b_{j-1,j-1}
 )  \\
 &= - (u_j+t) b_{j-1,j-1} + b_{j, j-1}  + (u_j - w(j+1)t) b_{j,j} 
\end{align*}
where we have used \eqref{eq:definition bkj}.
Comparing with the above, we may conclude that~\eqref{app eq:fij in terms of e and
  b} holds for $i=j+1$. 

We now wish to show that~\eqref{app eq:fij in terms of e and b} holds for
a pair $(i,j)$ with $i>j+1$, where 
we may also assume
$j>1$. 
We will use 
the following facts and conventions concerning the elementary symmetric
polynomials $e_k$: $e_{-1}=0$ and $e_0 = 1$ for any number of
variables, and $e_\ell(y_1, \ldots, y_s) = 0$ if $\ell > s$, i.e. if
the expected degree is greater than the number of variables. 
With these conventions and from the definition of the elementary
symmetric polynomials we may derive the identity 
\begin{equation} \label{app eq:relation on elem symmetrics}
\begin{split}
e_{i-k-1}(w(j),\dots,w(i-1)) =e_{i-k-1}(w(j+1),\dots,w(i-1)) +w(j)e_{i-k-2}(w(j+1),\dots,w(i-1))
\end{split}
\end{equation}
for any $k$ with $j-1\leq k \leq i-1$. 
Now by the recursive description \eqref{eq:fij at w} of $f_{i,j}(w)$,
the inductive hypotheses, \eqref{app eq:relation on elem symmetrics}, and \eqref{app eq:def of ur},
we can compute $f_{i,j}(w)$ to be 
\begin{align*}
f_{i,j}(w) & = f_{i-1,j-1}(w) + (u_j - w(i)t) \cdot f_{i-1,j}(w)  \quad\text{by \eqref{eq:fij at w}}  \\
 & = \sum_{k=j-1}^{i-1} (-1)^{i-k-1} e_{i-k-1}(w(j), \ldots, w(i-1))
 t^{i-k-1} b_{k,j-1}  \\
 &\qquad\qquad + (u_j - w(i)t) \cdot\left( \sum_{k=j}^{i-1} (-1)^{i-k-1}
   e_{i-k-1}(w(j+1), \ldots, w(i-1)) t^{i-k-1} b_{k,j} \right) \\
   &\hspace{300pt} \text{by the inductive hypothesis} \\
 & = \sum_{k=j-1}^{i-1} (-1)^{i-k-1} \Big( e_{i-k-1}(w(j+1),\dots,w(i-1))
 +w(j)e_{i-k-2}(w(j+1),\dots,w(i-1)) \Big) 
 t^{i-k-1} b_{k,j-1}  \\ 
 & \qquad\qquad + (u_j - w(i)t) \cdot \left( \sum_{k=j}^{i-1} (-1)^{i-k-1}
   e_{i-k-1}(w(j+1), \ldots, w(i-1)) t^{i-k-1}b_{k,j} \right) \quad\text{by \eqref{app eq:relation on elem symmetrics}} \\
 &=\sum_{k=j-1}^{i-1} (-1)^{i-k-1} e_{i-k-1}(w(j+1),\ldots,w(i-1))
 t^{i-k-1}b_{k,j-1} \\
 & \qquad\qquad + \sum_{k=j-1}^{i-1} (-1)^{i-k-1}
 e_{i-k-2}(w(j+1),\ldots,w(i-1)) (u_j+t)  t^{i-k-2} b_{k,j-1} \\
 &\qquad\qquad
 +\sum_{k=j}^{i-1} (-1)^{i-k-1} e_{i-k-1}(w(j+1),\ldots,w(i-1)) u_j  t^{i-k-1} b_{k,j} \\
 & \qquad\qquad - \sum_{k=j}^{i-1} (-1)^{i-k-1} e_{i-k-1}(w(j+1), \ldots, w(i-1))
 w(i) t^{i-k} b_{k,j} \qquad\text{by \eqref{app eq:def of ur}}.
\end{align*}
By the convention that $e_{-1}=0$, $e_\ell(y_1,\ldots,y_s)=0$ if $\ell >s$ 
and by a re-indexing in order to gather terms, it follows that this is further equal to
\begin{align*}
 &\sum_{k=j-1}^{i-1} (-1)^{i-k-1} e_{i-k-1}(w(j+1), \ldots, w(i-1))
 t^{i-k-1} \big( b_{k,j-1} + u_j b_{k,j} - (u_j+t) b_{k-1, j-1} \big)
 \\
 & \qquad\qquad - \sum_{k=j}^{i} (-1)^{i-k-1} e_{i-k-1}(w(j+1),
 \ldots, w(i-1)) w(i) t^{i-k} b_{k,j}. 
\end{align*} 
By the recursive definition~\eqref{eq:definition bkj} of $b_{k+1,j}$ for $k\geq j$, and because
the term 
$e_{i-k-1}(w(j+1), \ldots, w(i-1))$ vanishes for $k=j-1$, we can replace the
expressions $b_{k,j-1} + u_j b_{k,j} - (u_j+t) b_{k-1, j-1}$ in the first
summand above with $b_{k+1,j}$. 
Then by re-indexing the first summand and using (a re-indexed version
of) the equality~\eqref{app eq:relation on elem symmetrics}, we obtain the following equality
\begin{align*}
 f_{i,j}(w)=\sum_{k=j}^{i}(-1)^{i-k} e_{i-k}(w(j+1), \ldots, w(i)) t^{i-k}
 b_{k,j}.
\end{align*}
as, desired. This proves the claim. 
\end{proof}

\end{document}